\documentclass{article}
\usepackage{CJK}
\usepackage{comment}
\usepackage{amsmath}
\usepackage{amssymb}
\usepackage{amsthm}
\usepackage{amscd}
\usepackage{graphicx}
\usepackage{indentfirst}
\usepackage{titlesec}
\usepackage[english,francais]{babel}

\newtheorem*{maintheo}{Main Theorem}
\newtheorem*{theoa}{Theorem A}
\newtheorem*{theob}{Theorem B}

\newtheorem{pro}{Proposition}
\newtheorem{ques}{Question}
\newtheorem{conj}{Conjecture}
\newtheorem{theo}{Theorem}[section]
\newtheorem{lem}[theo]{Lemma}
\newtheorem{claim}[theo]{Claim}
\newtheorem{cor}[theo]{Corollary}
\newtheorem{rem}[theo]{Remark}
\theoremstyle{definition}
\newtheorem{defi}[theo]{Definition}

\begin{document}
\def \diff {\operatorname{Diff}}
\def \vep {\operatorname{\varepsilon}}
\def \dim {\operatorname{dim}}
\def \orb {\operatorname{Orb}}
\def \ind {\operatorname{Ind}}
\def \int {\operatorname{Int}}
\def \supp {\operatorname{supp}}
\def \Per {\operatorname{Per}}
\def \max {\operatorname{max}}
\def \min {\operatorname{min}}
\selectlanguage{english}
\title{Hyperbolicity versus weak periodic orbits inside homoclinic classes}

\author{Xiaodong Wang \footnote{The author was financially supported by China Scholarship Council (CSC) (201306010008) and ``Bo Xin Project'' (BX201600105) funded by China Postdoctoral Science Foundation.}}
\date{}

\maketitle
\thispagestyle{empty}

\begin{abstract}
We prove that, for $C^1$-generic diffeomorphisms, if the periodic orbits contained in a homoclinic class $H(p)$ have all their Lyapunov exponents bounded away from 0, then $H(p)$ must be (uniformly) hyperbolic. This is in sprit of the works of the stability conjecture, but with a significant difference that the homoclinic class $H(p)$ is not known isolated in advance, hence the ``weak" periodic orbits created by perturbations near the  homoclinic class have to be guaranteed strictly inside the homoclinic class. In this sense the problem is of an ``intrinsic" nature, and the classical proof of the stability conjecture does not pass through. In particular, we construct in the proof several perturbations which are not simple applications of the connecting lemmas.
\end{abstract}

\section{Introduction}

\subsection{Backgrounds and main results}

We study in this paper a problem that is in sprit of the works of the stability conjecture but with an ``intrinsic" nature. Let $M$ be a compact manifold without boundary, and $\diff^1(M)$ be the space of $C^1$-diffeomorphisms of $M$. Recall the stability conjecture formulated by Palis and Smale claims that if a diffeomorphism $f$ is structurally stable then it is hyperbolic. Here a diffeomorphism $f$ is called {\it hyperbolic} if the chain recurrent set $R(f)$ of $f$ (see Definition~\ref{chain recurrence}) is hyperbolic. A stronger version of the conjecture is to claim that if $f$ is $\Omega$-stable then it is hyperbolic. These two remarkable conjectures are solved by Ma\~{n}\'{e}~\cite{m3} and Palis~\cite{pa1}, respectively.

During the long way of study of the stability conjectures, the attention was more and more concentrated on periodic orbits of the (unperturbed) diffeomorphism $f$ as well as its perturbations $g$. Liao~\cite{l1} and Ma\~{n}\'{e}~\cite{m2} raised independently a conjecture (more precisely, a problem without a tentative answer), known as the star conjecture, stating that if $f$ has no, robustly, non-hyperbolic periodic orbits then it is hyperbolic. Being an assumption, the star condition is clearly weaker than the $\Omega$-stability. Hence the star conjecture is regarded another (strong) version of the stability conjecture.  It is solved by  Aoki and Hayashi~\cite{ao,h2}. To compare more precisely with our Main Theorem below we state their results in a generic version. Recall
that if $p$ is a periodic point with period $\tau$ of a diffeomorphism $f$, and if $\lambda_1,\lambda_2,\cdots,\lambda_d$ are the eigenvalues of $Df^{\tau}$ (counted by multiplicity), then the $d$ numbers $\chi_i=\frac{1}{\tau}\log|\lambda_i|$, $i=1,\cdots,d$ are called the \emph{Lyapunov exponents} of $\orb(p)$.

\begin{theo}[Aoki and Hayashi]\label{ah}
For a $C^1$-generic $f\in\diff^1(M)$, if $f$ is not hyperbolic, then there is a periodic orbit of $f$ that has a Lyapunov exponent arbitrarily close to 0.
\end{theo}

Now we state our main result. Recall that two hyperbolic periodic points are {\it homoclinically related} if $W^s(\orb(p))$ has non-empty transverse intersection with $W^u(\orb(q))$ and $W^u(\orb(p))$ has non-empty transverse intersection with $W^s(\orb(q))$. To be homoclinically related is an equivalent relation, and the \emph{homoclinic class} $H(p)$ of a hyperbolic periodic point $p$ is the closure of the union of periodic orbits that are homoclinically related to $p$. Two different homoclinic classes may intersect. Nevertheless by  the result of~\cite{bc}, for $C^1$-generic diffeomorphisms, every homoclinic class is a maximal invariant compact set that is {\it chain transitive} (see Definition~\ref{pseudo-orbit}), hence they are pairwise disjoint. Homoclinic classes are generally infinite in number, even for generic diffeomorphisms.

\begin{maintheo}
For a $C^1$-generic $f\in\diff^1(M)$, if a homoclinic class $H(p)$ of $f$ is not hyperbolic,
then there is a periodic orbit of $f$  that is homoclinically related to $\orb(p)$ and has a Lyapunov exponent arbitrarily close to 0.
\end{maintheo}

Note that here the ``weak" periodic orbit (the one with a Lyapunov exponent arbitrarily close to 0) is  homoclinically related to $\orb(p)$, that is, is ``inside" the homoclinic class $H(p)$. This is the main point of this paper. In fact, under the assumptions of the Main Theorem, it is straightforward to prove (following the classical proof of the stability conjecture) that, there must be a weak periodic orbit arbitrarily near $H(p)$. In contrast, here the Main Theorem claims there must be a weak periodic orbit not only near, but actually inside $H(p)$. Of course, if the homoclinic class $H(p)$ is assumed to be isolated, then  being ``near" will be equivalent to being ``inside". The point is that here $H(p)$ is not known to be isolated hence, at each step, the periodic orbits created by  perturbations have to be guaranteed to lie strictly inside the homoclinic class. It is in this sense we say the problem is of an ``intrinsic" nature, and the classical proof of the stability conjecture does not pass through.

There are some other results that concern whether the newly created periodic orbits are inside or outside the homoclinic class, for instance~\cite{bcdg,bs}. The lack of domination of the homoclinic class yields periodic orbits having multiple Lyapunov exponents and weak hyperbolicity (\cite{bcdg}) or allows to produce sinks or sources (\cite{bs}). The difference is that~\cite{bcdg} applies Franks-Gourmelon Lemma~\cite{g2} (see also Lemma~\ref{f-g lemma} in Section 3) to guarantee that the newly created periodic orbits are still inside the homoclinic class, while~\cite{bs} claims the newly created sinks or sources are outside the initial class by producing an attracting or repelling region. In our paper, we build new periodic orbits by mixing a hyperbolic periodic orbit and a \emph{weak set}, which has very weak hyperbolicity. The new periodic orbits have weak hyperbolicity because they spend a long time close to the weak set. However, by the property of hyperbolic periodic orbit, the amount of time that close to the initial periodic orbit (which is chosen after the amount of time close to the weak set has been fixed) can be any large number. Hence by choosing the amount of time properly, the hyperbolicity of the new periodic orbits are still uniform in some sense. This is the reason that the new periodic orbits are homoclinically related to the initial one. We have to connect first the initial periodic orbit and the weak set by heteroclinic orbits through several perturbations. The difficulty is that, the connection between the new periodic orbit and the homoclinic class may be destroyed in any step of the process. To avoid this, we have to use the generic properties to guarantee the perturbations to lie inside the homoclinic class.


There are other conjectures aimed to give a dichotomy of global dynamics. Recall that a \emph{homoclinic tangency} of a hyperbolic periodic point $p$ is a non-transverse intersection between $W^u(p)$ and $W^s(p)$. A diffeomorphism is with a \emph{heterodimensional cycle} if there are two hyperbolic periodic points $p$ and $q$ with different stable dimensions such that $W^s(p)\cap W^u(q)\neq \emptyset$ and $W^s(q)\cap W^u(p)\neq \emptyset$. It is obvious that any diffeomorphism with either a tangency or a heterodimensional cycle is not hyperbolic. Palis conjectured that these two phenomenons are the only obstacles for hyperbolicity. More precisely, the union of hyperbolic diffeomorphisms and diffeomorphisms with tangencies or heterodimensional cycles are dense in the space of diffeomorphisms, see~\cite{pa}. Based on the results afterwards, Bonatti and D\'iaz conjectured that the union of diffeomorphisms that are hyperbolic and those with heterodimensional cycles are dense in the space of diffeomorphisms, see~\cite{b,bd}. There are many works related to this subject, like~\cite{ps,c5,csy,cp}.~\cite{ps} solved this conjecture for dimension 2, and for higher dimension,~\cite{c5,csy,cp} got progress that far from homoclinic bifurcations, the systems has some weak hyperbolicity (partially hyperbolic or essentially hyperbolic).

By the Franks' lemma~\cite{f,g2}, we can perturb weak periodic orbits to get periodic orbits with different stable dimensions. But it is not clear whether these periodic orbits are still contained in the non-hyperbolic homoclinic class after perturbation. Thus we have the following conjecture, which is an intrinsic version of Palis conjecture for homoclinic classes.

\begin{conj}[\cite{b,bcdg}]\label{conj b}
There is a residual subset $\mathcal{R}\subset \diff^1(M)$, such that for all $f\in\mathcal{R}$, if a homoclinic class $H(p)$ is not hyperbolic, then there is a periodic point $q\in H(p)$, whose stable dimension is different from that of $p$.
\end{conj}

By~\cite{conley}, one can decompose the dynamics into pieces, and each piece is called a {\it chain recurrence class} (see Definition~\ref{chain class}). By~\cite{bc}, for $C^1$-generic diffeomorphism, a chain recurrence class is either a homoclinic class or contains no periodic point. We call a chain recurrence class without periodic point an \emph{aperiodic class}. Recall that a \emph{dominated splitting} $E\oplus F$ on an invariant compact $\Lambda$ set is an invariant splitting of $T_{\Lambda}M$ and the norm of $Df$ along $E$ is controlled by that along $F$, and $\Lambda$ is \emph{partially hyperbolic} if $T_{\Lambda}M$ splits into three bundles which is a dominated splitting such that the extremal bundles are hyperbolic and the center bundle is neutral (see Definition~\ref{dominated splitting}). By~\cite{c5,csy}, for $C^1$-generic diffeomorphisms far from homoclinic bifurcations (or just homoclinic tangencies), an aperiodic class is partially hyperbolic with center bundle of dimension 1. In~\cite{bd-aperiodic class}, they proved that if $\dim(M)\geq 3$, then there are an open set $\mathcal{U}$ of $\diff^1(M)$, and a residual subset $\mathcal{V}$ of $\mathcal{U}$, such that any $g\in\mathcal{V}$ has infinitely many aperiodic classes, and each of them has no non-trivial dominated splitting. We state here a conjecture by S. Crovisier for aperiodic classes, which implies the non-existence of aperiodic classes for $C^1$-generic diffeomorphisms far from homoclinic bifurcations.

\begin{conj}[\cite{c3}]
Let $\Lambda$ be an aperiodic class for a $C^1$-generic diffeomorphism $f$, and $E^s \oplus E^c\oplus E^u$ the dominated splitting such that $E^s$ (resp. $E^u$) is the maximal uniformly contracting (resp. expanding) sub-bundle. Then $E^c$ has dimension larger than or equal to 2 and does not admit a finer dominated splitting.
\end{conj}

\subsection{Main theorem restated}
We give here a more general result rather than the main theorem. For a hyperbolic periodic point $p$, denote by $\ind(p)$ its stable dimension.

\begin{theoa}
For $C^1$-generic $f\in \diff^1(M)$, assume that $p$ is a hyperbolic periodic point of $f$. If the homoclinic class $H(p)$ has a dominated splitting $T_{H(p)}M=E\oplus F$, with $dim E\leq \ind(p)$, such that the bundle $E$ is not contracting, then there are periodic orbits in $H(p)$ with index $\dim(E)$ that have the maximal Lyapunov exponents along $E$ arbitrarily close to 0.
\end{theoa}

\begin{rem}\label{rem:theorem a}
If in the assumption of the Theorem A, $dim E=\ind(p)$, then the periodic orbits $\mathcal{O}_k=\orb(q_k)$ obtained have the same index as $p$. Thus by the genericity assumption and the item 4 of Lemma~\ref{generic properties}, these periodic orbits $\mathcal{O}_k$ are homoclinically related with $\orb(p)$.
\end{rem}

We give an explanation why Theorem A implies the main theorem. We assume that all Lyapunov exponents of periodic orbits that are homoclinically related to $\orb(p)$ are uniformly away from 0. Then by the genericity assumption, $H(p)$ has a dominated splitting $T_{H(p)}M=E\oplus F$, with $dim E= \ind(p)$, (see~\cite{gy} and Proposition 4.8 of~\cite{bcdg}). By the conclusion of Theorem A and the assumption of no existence of weak periodic orbits homoclinically related to $\orb(p)$, we get that the bundle $E$ is contracting. With the same argument for $f^{-1}$, we get that the bundle $F$ is expanding for $f$. Hence $T_{H(p)}M=E\oplus F$ is a hyperbolic splitting which is a contradiction to the assumption of the main theorem.

In~\cite{m3}, Ma\~n\'e introduced a very useful lemma (Theorem \textbf{II.1}) to get weak periodic orbits under certain hypotheses. The statement is very technical and the original proof of Ma\~{n}\'{e} is difficult, thus we will not state it here. Based on a modification of the proof of Ma\~{n}\'{e}, Bonatti, Gan and Yang have a result for homoclinic classes, see~\cite{bgy}.

Here we point out that, different from Theorem \textbf{II.1} of~\cite{m3} and the result of~\cite{bgy}, there is a genericity assumption in the main theorem and Theorem A. That is to say, the conclusion of the main theorem is a perturbation result and may not be valid for all diffeomorphisms. Thus one asks the following question naturally, whether the genericity assumption is essential in the main theorem.

\begin{ques}\label{question 1}
Is there a homoclinic class $H(p)$ for a diffeomorphism $f$ satisfying that all the Lyapunov exponents of all periodic orbits homoclinically related to $\orb(p)$ are uniformly away from 0 but $H(p)$ is not hyperbolic?
\end{ques}

In some cases, we can give a positive answer to this question. In~\cite{r}, Rios proved that there is a diffeomorphism on the boundary of the set of hyperbolic diffeomorphisms on surface, with a homoclinic class containing a tangency inside. Hence it is not hyperbolic (it does not have a dominated splitting because of the existence of tangency). In~\cite{clr}, they proved that for this homoclinic class, all the Lyapunov exponents of all periodic orbits contained in the class are uniformly away form 0. In fact, they proved more that all the Lyapunov exponents of all ergodic measures are uniformly away from 0. Examples of non-hyperbolic homoclinic classes with a dominated splitting can be found in like~\cite{c4,dg,dhrs,ps2}, but the homoclinic classes in these examples contain weak periodic orbits. For $C^2$ diffeomorphisms on surfaces, by the conclusions of~\cite{ps2}, one can not give a non-hyperbolic homoclinic class with domination and without weak periodic orbits, which is unknown in the $C^1$ dynamics. Hence we have the following question which is a stronger version of Question~\ref{question 1}.

\begin{ques}\label{question 2}
Is there a non-hyperbolic homoclinic class $H(p)$ with a non-trivial dominated splitting for a diffeomorphism $f$ satisfying that all the Lyapunov exponents of all periodic orbits homoclinically related to $\orb(p)$ are uniformly away from 0?
\end{ques}

\subsection{Some applications of the main theorem}

In this subsection, we give some applications whose proof will be given later.

\subsubsection{Structural stability and hyperbolicity}

Recall that a diffeomorphism $f\in\diff^1(M)$ is \emph{structurally stable}, if there is a $C^1$ neighborhood $\mathcal{U}$ of $f$, such that, for any $g\in\mathcal{U}$, there is a homeomorphism $\phi:M\rightarrow M$, satisfying $\phi\circ f=g\circ\phi$. The orbital structure of a structurally stable diffeomorphism remains unchanged under perturbations. Ma\~{n}\'{e} proved that the chain recurrent set of a structurally stable diffeomorphism is hyperbolic, see~\cite{m3}. Here we give a local version about this result.

It is known that a hyperbolic periodic point has a continuation. More precisely, for a hyperbolic periodic point $p$ of a diffeomorphism $f$ with period $\tau$, there is a neighborhood $U$ of $\orb(p)$ and a $C^1$ neighborhood $\mathcal{U}$ of $f$, such that, for any $g\in\mathcal{U}$, the maximal invariant compact set of $g$ in $U$ is a unique periodic orbit with period $\tau$ and with the same index as $p$. We denote this \emph{continuation} of $p$ by $p_g$ for such a diffeomorphism $g$, and denote the homoclinic class of $p_g$ by $H(p_g)$. Thus we say that a homoclinic class $H(p)$ of a diffeomorphism $f$ is \emph{structurally stable}, if there is a $C^1$ neighborhood $\mathcal{U}$ of $f$, such that, for any $g\in\mathcal{U}$, there is a homeomorphism $\phi:H(p)\rightarrow H(p_g)$, satisfying $\phi \circ f|_{H(p)}=g\circ\phi|_{H(p)}$, where $p_g$ is the continuation of $p$. 

\begin{ques}\label{stability}
Assume $p$ is a hyperbolic periodic point for a diffeomorphism, if $H(p)$ is structurally stable, then is it hyperbolic?
\end{ques}

There are many works related to this question, see for example~\cite{gy,sv,ww,wenx}. In~\cite{ww} and~\cite{wenx}, they prove that structural stability implies hyperbolicity for the chain recurrence class and the homoclinic class respectively of a hyperbolic periodic point, under the hypothesis that the diffeomorphism is far away from tangency, or that the stable or the unstable dimension of this periodic point is 1. With the conclusion of the main theorem, we can give a complete answer to Question~\ref{stability}.

\begin{cor}\label{stuctually stable}
Assume $f$ is a diffeomorphism in $\diff^1(M)$ and $p$ is a hyperbolic periodic point of $f$. If the homoclinic class $H(p)$ is structurally stable, then $H(p)$ is hyperbolic. 
\end{cor}

\subsubsection{Partial hyperbolicity}

Next result is that for a homoclinic class with a dominated splitting of a $C^1$-generic diffeomorphism, if the dimensions of the two bundles in the splitting satisfy certain hypotheses, then the splitting is a partially hyperbolic splitting (at least one bundle is hyperbolic).

\begin{cor}\label{application 1}
For $C^1$-generic $f\in \diff^1(M)$, if a homoclinic class $H(p)$ has a dominated splitting $T_{H(p)}M=E\oplus F$, such that $\dim(E)$ is smaller than the smallest index of periodic orbits contained in $H(p)$, then the bundle $E$ is contracting. Symmetrically, if $\dim(E)$ is larger than the largest index of periodic orbits contained in $H(p)$, then the bundle $F$ is expanding.
\end{cor}

As another consequence of the main theorem, we can give a proof of Theorem 1.1 (2) in~\cite{csy} with a different argument. More precisely, we can prove that for a $C^1$-generic diffeomorphism far from tangency, a homoclinic class has a partially hyperbolic splitting whose center bundle splits into 1-dimensional subbundles, and the Lyapunov exponents of the periodic orbits along each the center subbundle can be arbitrarily close to 0. Denote $\mathcal{HT}$ the set of diffeomorphisms of $\diff^1(M)$ that exhibit a tangency.

\begin{cor}[\cite{csy}]\label{application 2}
For $C^1$-generic $f\in \diff^1(M)\setminus \overline{\mathcal{HT}}$, a homoclinic class $H(p)$ has a partially hyperbolic splitting $T_{H(p)}M=E^s\oplus E^c_1\oplus\cdots\oplus E^c_k\oplus E^u$ such that each of the center subbundles $E^c_i$ is neither contracting nor expanding and $\dim(E^c_i)=1$, for all $i=1,\cdots,k$. Moreover, the minimal index of periodic points contained in $H(p)$ is $\dim(E^s)$ or $\dim(E^s)+1$, and symmetrically, the maximal index of periodic points contained in $H(p)$ is $d-\dim(E^u)$ or $d-\dim(E^u)-1$. For each $i=1,\cdots,k$, there exist periodic orbits contained in $H(p)$ with arbitrarily long periods with a Lyapunov exponent along $E^c_i$ arbitrarily close to $0$.
\end{cor}

\subsubsection{Lyapunov stable homoclinic classes}
Recall that an invariant compact set $\Lambda\subset M$ is \emph{Lyapunov stable for $f$}, if for any neighborhood $U$ of $\Lambda$, there is another neighborhood $V$ of $\Lambda$, such that $f^n(V)\subset U$ for all $n\geq 0$. We say that $\Lambda$ is \emph{bi-Lyapunov stable}, if $\Lambda$ is both Lyapunov stable for $f$ and for $f^{-1}$.

The following results are about $C^1$-generic Lyapunov stable homoclinic classes. First, for $C^1$-generic Lyapunov stable homoclinic classes, we can get a similar conclusion of Corollary~\ref{application 1} under a weaker hypothesis.

\begin{cor}\label{application 3}
For $C^1$-generic $f\in \diff^1(M)$, if a homoclinic class $H(p)$ is Lyapunov stable and has a dominated splitting $T_{H(p)}M=E\oplus F$ such that $\dim(E)$ is larger than or equal to the largest index of periodic orbits contained in $H(p)$, then the bundle $F$ is expanding.
\end{cor}

With the conclusion of Corollary~\ref{application 3}, we can give a positive answer to Conjecture~\ref{conj b} for bi-Lyapunov stable homoclinic classes.

\begin{cor}\label{application 4}
For $C^1$-generic $f\in \diff^1(M)$, where $M$ is connected, if a homoclinic class $H(p)$ is bi-Lyapunov stable, then we have:
\begin{itemize}
\item either $H(p)$ is hyperbolic, hence $H(p)=M$ and $f$ is Anosov,
\item or $f$ can be $C^1$ approximated by diffeomorphisms that have a heterodimensional cycle.
\end{itemize}
\end{cor}

From~\cite{csy} (or Corollary~\ref{application 2}), we know that for $C^1$-generic diffeomorphisms far away from tangencies, a homoclinic class has a partially hyperbolic splitting with all central bundles dimension 1. We have the following result about the index of periodic orbits for Lyapunov stable homoclinic classes. It is a direct corollary of Corollary~\ref{application 3} and we omit the proof.

\begin{cor}
For $C^1$-generic $f\in \diff^1(M)\setminus \overline{\mathcal{HT}}$, if a homoclinic class $H(p)$ is Lyapunov stable and assume $T_{H(p)}M=E^s\oplus E^c_1\oplus\cdots\oplus E^c_k\oplus E^u$ is the partially hyperbolic splitting, then the largest index of periodic points contained in $H(p)$ equals $d-\dim(E^s)$.
\end{cor}

\subsection{Propositions for the proof of Theorem A}

To prove Theorem A, we will use the following three propositions. Proposition~\ref{time control} tells that for any hyperbolic periodic orbit $\orb(p)$ and any invariant compact set $K$ of a diffeomorphism $f$ linked by heteroclinic orbits, we can get a periodic orbit that spends a given proportion of time close to $\orb(p)$ and $K$ by arbitrarily $C^1$ small perturbation. In the whole paper, we denote by $\orb^+(x,f)$ (resp. $\orb^-(x,f)$) the positive (resp. negative) $f$-orbit of the point $x$, i.e. $\orb^+(x,f)=\{x,f(x),f^2(x),\cdots\}$ (resp. $\orb^-(x,f)=\{x,f^{-1}(x),f^{-2}(x),\cdots\}$). If there is no misunderstanding, we just denote it by $\orb^+(x)$ (resp. $\orb^-(x)$) for simplicity.

\begin{pro}\label{time control}
Let $f\in\diff^1(M)$. Consider a hyperbolic periodic point $p$ and an invariant compact set $K$, satisfying that all periodic points contained in $K$ are hyperbolic and $p\notin K$. Assume moreover that there are two points $x$ and $y$, satisfying that:
\begin{itemize}
\item $x \in W^{u}(p)$ and $\omega(x)\cap K \neq \emptyset$,
\item $y \in W^{s}(p)$ and $\alpha(y)=K$.
\end{itemize}

Then for any neighborhood $\mathcal{U}$ of $f$ in $\diff^1(M)$, any neighborhood $U_p$ of $\orb(p)$, and any neighborhood $U_K$ of $K$, there are two integers $l$ and $n_0$, such that, for any integer $T_K$,
\begin{enumerate}
\item there is $h\in\mathcal{U}$ such that:
\begin{itemize}
\item $h$ coincides with $f$ outside $U_K$,
\item the point $y$ is on the positive orbit of $x$ under $h$,
\item $\# (\orb(x,h)\cap U_K)\geq T_K$ and $\# ((\orb(x,h)\setminus (U_K\cup U_p))\leq n_0$.
\end{itemize}

\item for any $m\in\mathbb{N}$, there is $h_m\in\mathcal{U}$ such that:
\begin{itemize}
\item $h_m$ coincides with $h$ on $\orb(p)$ and outside $U_p$,
\item $h_m$ has a periodic orbit $O$, satisfying $O\setminus U_p=(\orb(x,h))\setminus U_p$, and $\# (O\cap U_p)\in \{l+m\tau,l+m\tau+1,\cdots,l+(m+1)\tau-1\}$.
\end{itemize}
\end{enumerate}
\end{pro}

\begin{rem}
$(1)$ It is easy to see that if we take $U_K$ small enough, then $h$ coincides with $f$ on $\orb(p)\cup \orb^-(x)\cup \orb^+(y)$.

$(2)$ For the diffeomorphism $h$, the point $x$ (also $y$) is a homoclinic point of the periodic orbit $\orb(p)$.

$(3)$ It is obvious that, in the settings of the proposition, if we change ``$\omega(x)\cap K \neq \emptyset$ and $\alpha (y) = K$'' to ``$\alpha (y)\cap K \neq \emptyset$ and $\omega(x) = K$'', the conclusion still holds.
\end{rem}

Proposition~\ref{asymptotic connecting 1} and~\ref{asymptotic connecting} are in some sense doing an asymptotic connecting process from a point to an invariant compact set. Proposition~\ref{asymptotic connecting 1} tells that if the closure of the unstable manifold $\overline{W^u(p)}$ of a hyperbolic periodic point $p$ intersects an invariant compact set $K$, then by an arbitrarily small perturbation, one can obtain a point which belongs to the unstable manifold of $p$ and whose positive limit set is contained in $K$. Moreover, the perturbation will not change certain pieces of orbit. In fact, the first property can be obtained from the proof of Proposition 10 in~\cite{c1}, but the second property is not a direct consequence.

\begin{pro}[A modified case of Proposition 10 in~\cite{c1}]\label{asymptotic connecting 1}
Let $f\in\diff^1(M)$. Consider an invariant compact set $K$ which contains no non-hyperbolic periodic point and a point $x\in M$ with $\alpha(x)\subset K$. Assume $p$ is a hyperbolic periodic point satisfying that $p\notin K$ and $\overline{W^u(p)}\cap K\neq \emptyset$.\\
Then for any neighborhood $\mathcal{U}$ of $f$ in $\diff^{1}(M)$, there is a diffeomorphism $g\in \mathcal{U}$, satisfying the following properties:
\begin{enumerate}
\item $g$ coincides with $f$ on $\orb(p)\cup K$;
\item $g$ and $Dg$ coincides with $f$ and $Df$ respectively on $\orb^-(x)$, hence $\alpha(x,g)\subset K$;
\item there is a point $y\in W^u(p,g)$, such that $\omega(y,g)\subset K$.
\end{enumerate}
\end{pro}

In the assumptions of the above two propositions, the point and invariant compact sets are linked by true orbits. However, Proposition~\ref{asymptotic connecting} deals with the case that they are linked by pseudo-orbits which is more complicated. We use the techniques of~\cite{bc,c1}.

\begin{pro}\label{asymptotic connecting}
Assume $f_0$ is a diffeomorphism in $\diff^1(M)$. For any neighborhood $\mathcal{U}$ of $f_0$ in $\diff^{1}(M)$, there are a smaller neighborhood $\mathcal{U}'$ of $f_0$ with $\overline{\mathcal{U}'}\subset\mathcal{U}$ and an integer $T$, with the following properties.\\
For any diffeomorphism $f\in\mathcal{U}'$, considering an invariant compact set $K$, a positively invariant compact set $X$ and a point $z\in X$, suppose that the following conditions are satisfied:
\begin{itemize}
\item all periodic orbits contained in $K$ are hyperbolic,
\item all periodic orbits contained in $X$ with period less than $T$ are hyperbolic,
\item for any $\vep>0$, there is an $\vep$-pseudo-orbit contained in $X$ connecting $z$ to $K$,
\end{itemize}
then for any neighborhood $U$ of $X\setminus K$ and for any $\gamma>0$, there is a diffeomorphism $g\in \mathcal{U}$, such that: $g|_{M\setminus U}=f|_{M\setminus U}$ and $\omega(z,g)\subset K$. Moreover, the $C^0$ distance between $g$ and $f$ is smaller than $\gamma$.
\end{pro}

\begin{rem}
$(1)$ Proposition~\ref{asymptotic connecting 1} is not a direct corollary of Proposition~\ref{asymptotic connecting}, because we wish to keep the negative orbit of a point that accumulates to the invariant compact set unchanged after perturbation in Proposition~\ref{asymptotic connecting 1}.

$(2)$ In Proposition~\ref{asymptotic connecting}, we can see that $X\cap K\neq\emptyset$. Thus $X\setminus K$ is not a compact set and we have that $\overline{U}\cap K\neq\emptyset$, where $U$ is the neighborhood of $X\setminus K$.
\end{rem}

\subsection{Organization of the paper.}
In Section~\ref{preliminary}, we give some basic definitions and well-known results that we will use in the proof. In Section~\ref{theorem b}, we give a slightly different version (Theorem B) of Theorem A, and we prove Theorem A using Theorem B. Later, we give the proof of Theorem B from Propositions 1, 2 and 3 in Section~\ref{proof of theorem b}. The proofs of Proposition 1, 2 and 3 will be given in Section~\ref{proposition 1},~\ref{proposition 2} and~\ref{proposition 3} respectively. At last, we give the proofs of the applications of the main theorem in Section~\ref{applications}.

\section{Preliminary}\label{preliminary}

In this section, we give some definitions and some well-known results. Denote by $\diff^1(M)$ the space of $C^1$-diffeomorphisms of $M$.

\subsection{Hyperbolicity and dominated splitting}

\begin{defi}
Assume that $f$ is a diffeomorphism in $\diff^1(M)$, $\Lambda$ is an invariant compact set of $f$ and $E$ is a $Df$-invariant subbundle of $T_{\Lambda}M$. We say that the bundle $E$ is \emph{$(C,\lambda)$-contracting} if there are constants $C>0$ and $\lambda\in (0,1)$, such that
   \begin{displaymath}
     \|Df^n|_{E(x)}\|<C\lambda^n,
   \end{displaymath}
for all $x\in \Lambda$ and all $n\geq 1$. And we say that $E$ is \emph{$(C,\lambda)$-expanding} if it is $(C,\lambda)$-contracting with respect to $f^{-1}$. If the tangent bundle of $\Lambda$ has an invariant splitting $T_{\Lambda}M=E^s\oplus E^u$, such that, $E^s$ is $(C,\lambda)$-contracting and $E^u$ is $(C,\lambda)$-expanding for some constants $C>0$ and $\lambda\in (0,1)$, then we call $\Lambda$ a \emph{hyperbolic set} and $\dim(E^s)$ the \emph{index} of the hyperbolic splitting. If a periodic orbit $\orb(p)$ is a hyperbolic set, then we call $p$ a hyperbolic periodic point, and the dimension of the contracting bundle $E^s$ in the hyperbolic splitting is called the \emph{index} of $p$, denoted by $\ind(p)$.
\end{defi}

\begin{defi}
For any point $x\in M$ and any number $\delta>0$, we define the \emph{local stable set} and \emph{local unstable set} of $x$ of size $\delta$ respectively as follows:\\
$W^s_{\delta}(x)=\{y: \forall n\geq 0, d(f^n(x),f^n(y))\leq\delta; \text{ and } \lim_{n\rightarrow +\infty} d(f^n(x),f^n(y))=0\}$;\\
$W^u_{\delta}(x)=\{y: \forall n\geq 0, d(f^{-n}(x),f^{-n}(y))\leq\delta; \text{ and } \lim_{n\rightarrow +\infty} d(f^{-n}(x),f^{-n}(y))=0\}$.\\
We define the \emph{stable set} and \emph{unstable set} of $x$ respectively as follows:\\
$W^s(x)=\{y:\lim_{n\rightarrow +\infty} d(f^n(x),f^n(y))=0\}$;\\
$W^u(x)=\{y:\lim_{n\rightarrow +\infty} d(f^{-n}(x),f^{-n}(y))=0\}$.
\end{defi}

\begin{rem}
$(1)$ It is obvious that, for any $\delta>0$, we have
   \begin{center}
     $W^s(x)=\cup_{n\geq 0}f^{-n}(W^s_{\delta}(f^n(x)))$
   \end{center}
 and
   \begin{center}
     $W^u(x)=\cup_{n\geq 0}f^{n}(W^u_{\delta}(f^{-n}(x)))$.
   \end{center}

$(2)$ To belong to a same stable set is an equivalent relation, thus two stable sets either coincide or are disjoint with each other. The same holds for the unstable set.
\end{rem}

For hyperbolic sets, the (local) stable (resp. unstable) set has the following properties, see for example~\cite{hps}.

\begin{lem}
Assume $\Lambda$ is an invariant compact set of $f$.
If $\Lambda$ is hyperbolic and $T_{\Lambda}M=E^s\oplus E^u$ is the hyperbolic splitting, then there is a number $\delta>0$, such that, for any $x\in\Lambda$, the local stable (resp. unstable) set $W^s_{\delta}(x)$ (resp. $W^u_{\delta}(x))$ is an embedded disk with dimension $\dim(E^s)$ (resp. $\dim(E^u)$) and is tangent to $E^s$ (resp. $E^u$) at $x$. Moreover, the stable (resp. unstable) set $W^s(x)$ (resp. $W^u(x)$) of $x$ is an immersed submanifold of $M$.
\end{lem}

\begin{defi}
Assume that $f$ is a diffeomorphism in $\diff^1(M)$ and $p,q\in M$ are two hyperbolic periodic points of $f$. We say $p$ and $q$ are \emph{homoclinically related} and denote the relation by $p\sim q$, if $W^u(\orb(p))$ has non-empty transverse intersections with $W^s(\orb(q))$, and $W^s(\orb(p))$ has non-empty transverse intersections with $W^u(\orb(q))$, denoted by $W^u(\orb(p))\pitchfork W^s(\orb(q))\neq \emptyset$ and $W^s(\orb(p))\pitchfork W^u(\orb(q))\neq\emptyset$. We call the closure of the set of periodic orbits homoclinically related to $\orb(p)$ the \emph{homoclinic class} of $p$ and denote it by $H(p,f)$ or $H(p)$ for simplicity.
\end{defi}

\begin{defi}\label{dominated splitting}
Let $f\in \diff^1(M)$. An invariant compact set $\Lambda$ of $M$ is said to have an \emph{$(m,\lambda)$-dominated splitting}, if the tangent bundle has a $Df$-invariant splitting $T_{\Lambda}M=E\oplus F$ and there are an integer $m$ and a constant $\lambda\in (0,1)$ such that
   \begin{center}
     $\|Df^m|_{E(x)}\|\cdot \|Df^{-m}|_{F(f^mx)}\|<\lambda$.
   \end{center}
We call $\dim(E)$ the \emph{index} of the dominated splitting. Moreover, we say $\Lambda$ has a \emph{partially hyperbolic splitting}, if the tangent bundle has an invariant splitting $T_{\Lambda}M=E^s\oplus E^c\oplus E^u$, such that the two splittings $(E^s\oplus E^c)\oplus E^u$ and $E^s\oplus (E^c\oplus E^u)$ are both dominated splittings and, moreover, the bundle $E^s$ is contracting, the bundle $E^u$ expanding and the central bundle $E^c$ is neither contracting nor expanding.
\end{defi}

\begin{rem}\label{bundle of ds}
We point out here that if an invariant compact set $\Lambda$ has two dominated splittings $T_{\Lambda}M=E_1\oplus F_1=E_2\oplus F_2$ such that $\dim(E_1)\leq \dim(E_2)$, then we have $E_1\subset E_2$. Hence two dominated splittings on an invariant compact set with the same index would coincide.
\end{rem}

By~\cite{g1}, there is always an \emph{adapted metric} for a dominated splitting, that is to say, an $(m,\lambda)$-dominated splitting is a $(1,\lambda)$-dominated splitting by considering a metric equivalent to the original one. Also, it is obvious that an $(m,\lambda)$-dominated splitting is always an $(mN,\lambda)$-dominated splitting for any positive integer $N$.

\subsection{Recurrence}

We give some definitions of recurrence.

\begin{defi}\label{pseudo-orbit}
For a diffeomorphism $f\in\diff^1(M)$ and a number $\vep>0$, we call a sequence of points $\{x_i\}_{i=a}^b$ of $M$ an \emph{$\vep$-pseudo orbit of $f$}, if $d(f(x_i),x_{i+1})<\vep$ for any $i=a,a+1,\cdots,b-1$, where $-\infty\leq a<b\leq \infty$. An invariant compact set $K$ is called a \emph{chain transitive set}, if for any $\vep>0$, there is a periodic $\vep$-pseudo-orbit contain in $K$ and $\vep$-dense in $K$.
\end{defi}

\begin{defi}\label{chain recurrence}
Assume $f\in\diff^1(M)$. We say a point $y$ is \emph{chain attainable} from $x$, if for any number $\vep>0$, there is an $\vep$-pseudo orbit of $f$ $(x_0,x_1,\cdots,x_n)$ such that $x_0=x$ and $x_n=y$, and we denote it by $x\dashv y$. The \emph{chain recurrent set} of a diffeomorphism $f\in\diff^1(M)$, denoted by $R(f)$, is the set of the points $x$ such that $x$ is chain attainable from itself.
\end{defi}

It is well-known that the chain recurrent set $R(f)$ of $f$ can be decomposed into a disjoint union of invariant compact "undecomposable" sets. More precisely, we give the definition as the following.

\begin{defi}\label{chain class}
Assume $f\in\diff^1(M)$. For any two points $x,y\in M$, denote $x\sim y$ if $x$ is chain attainable from $y$ and $y$ is chain attainable from $x$. Obviously $\sim$ is an equivalent relation on $R(f)$, and an equivalent class of $\sim$ is called a \emph{chain recurrence class}.
\end{defi}

\begin{defi}
Assume $f\in\diff^1(M)$ and $\Lambda$ is an invariant compact set of $f$. We say that $\Lambda$ is \emph{shadowable}, if for any $\vep>0$, there is $\delta>0$, such that for any $\delta$-pseudo orbit $\{x_i\}_{i=a}^b\subset \Lambda$ of $f$, where $-\infty\leq a<b\leq\infty$, there is a point $y\in M$, such that $d(f^i(y),x_i)<\vep$ for all $a\leq i\leq b$.
\end{defi}

Now we give another definition of a relation, which is denoted by $\prec$.

\begin{defi}
Assume $f$ is a diffeomorphism in $\diff^1(M)$ and $W$ is an open set of $M$. For any two points $x,y\in M$, we denote $x\prec y$ if for any neighborhood $U$ of $x$ and any neighborhood $V$ of $y$, there are a point $z\in M$ and an integer $n\geq 1$, such that $z\in U$ and $f^n(z)\in V$. We denote $x\prec_W y$ if for any neighborhood $U$ of $x$ and any neighborhood $V$ of $y$, there is a piece of orbit $(z,f(z),\cdots,f^n(z))$ contained in $W$ such that $z\in U$ and $f^n(z)\in V$. Moreover, let $K$ be a compact set of $M$, then we denote $x\prec K$ (resp. $x\prec_{W} K$) if there is a point $y\in K$, such that $x\prec y$ (resp. $x\prec_{W} y$).
\end{defi}

For the relation $\prec$, we have the following result, whose proof is similar to the proof of Lemma 6 in~\cite{c1}.

\begin{lem}\label{prec}
Assume that $K$ is an invariant compact set. Then for any two neighborhoods $U_2\subset U_1$ of $K$ and any point $y\in U_1$ satisfying $y\prec_{U_1} K$, there is a point $y'\in U_2$, such that $y\prec_{U_1} y'\prec_{U_2} K$ and the positive orbit of $y'$ is contained in $U_2$.
\end{lem}

It is obvious that $x\prec y$ implies $x\dashv y$, but the two relations are not equivalent.  In~\cite{bc}, they have proved that for generic diffeomorphisms, the two relations are equivalent.

\begin{lem}[\cite{bc}]\label{prec=dashv}
For generic diffeomorphism $f\in\diff^1(M)$, if $x\dashv y$, then $x\prec y$.
\end{lem}


\subsection{Pliss points and weak sets}

\begin{defi}\label{pliss point}
Assume that $\Lambda$ is an invariant compact set of a diffeomorphism $f$ in $\diff^1(M)$ and $E$ is an invariant sub-bundle of $T_{\Lambda}M$. For a constant $\lambda\in(0,1)$, we call $x\in \Lambda$ an \emph{$(m,\lambda)$-$E$-Pliss point}, if for any integer $n>0$, we have
   \begin{displaymath}
     \prod_{i=0}^{n-1} \|Df^{im}|_{E(f^{im}(x))}\|\leq {\lambda}^n.
   \end{displaymath}
Particularly, if $m=1$, we call $x$ a \emph{$\lambda$-$E$-Pliss point} for short.
\end{defi}

\begin{rem}\label{limit of pliss point}
It is not difficult to see that, if $\{x_i\}$ is a sequence of $\lambda$-$E$-Pliss points, then any limit point $y$ of the sequence is also a $\lambda$-$E$-Pliss point.
\end{rem}

\begin{defi}\label{weak set}

Consider a diffeomorphism $f\in \diff^1(M)$, an invariant compact set $K$ of $f$, an invariant sub-bundle $E$ of $T_K M$, an integer $m$ and a constant $\lambda\in (0,1)$. We say that $K$ is an \emph{$(m,\lambda)$-$E$-weak set}, if for any point $x\in K$, there is an integer $n_x$, such that
   \begin{displaymath}
     \prod_{i=0}^{n_x-1} \|Df^m|_{E(f^{im}(x))}\|> {\lambda}^{n_x}.
   \end{displaymath}
We denote $N_x$ the smallest integer that satisfies the above inequality. Particularly, if $m=1$, we call $K$ a \emph{$\lambda$-$E$-weak set} for short.

\end{defi}

\begin{rem}\label{rem of weak set}
If $K$ is an $(m,\lambda)$-$E$-weak set, by the compactness of $K$, we can see that $N_x$ is bounded by an integer $N_K$ for all $x\in K$. Also from the definition, we can see that an invariant compact set $K$ is an $(m,\lambda)$-$E$-weak set if and only if $K$ does not contain any $(m,\lambda)$-$E$-Pliss point.
\end{rem}

One can obtain Pliss points by the following lemma given by V. Pliss, see~\cite{p,ps}.

\begin{lem}[Pliss lemma]\label{pliss lemma}
Assume that $\Lambda$ is an invariant compact set of a diffeomorphism $f$ in $\diff^1(M)$ and $E$ is an invariant sub-bundle of $T_{\Lambda}M$. For any two numbers $0<\lambda_1<\lambda_2<1$, we have:
\begin{enumerate}
\item There are a positive integer $N=N(\lambda_1,\lambda_2,f)$ and a number $c=c(\lambda_1,\lambda_2,f)$ such that for any $x\in \Lambda$ and any number $n\geq N$, if
   \begin{displaymath}
     \prod_{i=0}^{n-1}\|Df|_{E(f^ix)}\|\leq {\lambda_1}^n,
   \end{displaymath}
then there are $0\leq n_1<n_2<\cdots<n_l\leq n$ such that $l\geq cn$, and, for any $j=1,\cdots,l$ and any $k=n_j+1,\cdots,n$,
   \begin{displaymath}
     \prod_{i=n_j}^{k-1}\|Df|_{E(f^ix)}\|\leq {\lambda_2}^{k-n_j}.
   \end{displaymath}
\item For any point $x\in \Lambda$, and any integer $m$, if for all $n\geq m$,
   \begin{displaymath}
     \prod_{i=0}^{n-1}\|Df|_{E(f^ix)}\|\leq {\lambda_1}^n,
   \end{displaymath}
then there is an infinite sequence $0\leq n_1<n_2<\cdots$, such that
   \begin{displaymath}
     \prod_{i=n_j}^{k-1}\|Df|_{E(f^ix)}\|\leq {\lambda_2}^{k-n_j},
   \end{displaymath}
for all $k>n_j$ and all $j=1,2,\cdots$.
\end{enumerate}
\end{lem}

\begin{cor}\label{cor of pliss}
For a diffeomorphism $f\in \diff^1(M)$ and an $f$-invariant continuous bundle $E\subset T_{\Lambda}M$ of an invariant compact set $\Lambda$, we have that, for any $x\in \Lambda$:
\begin{enumerate}
\item If $x$ is an $(m,\lambda)$-$E$-Pliss point, then there is a point $y\in \omega(x)$, such that $y$ is also a $(m,\lambda)$-$E$-Pliss point.
\item If for any $y\in\omega(x)$, there is an integer $n_y\in\mathbb{N}$, such that
   \begin{displaymath}
     \prod_{i=0}^{n_y-1}\|Df^m|_{E(f^{im}(y))}\|\leq {\lambda}^{n_y},
   \end{displaymath}
then for any $\lambda'\in (\lambda,1)$, there are infinitely many $(m,\lambda')$-$E$-Pliss points on $\orb^+(x)$.
\end{enumerate}
\end{cor}

\begin{proof}
By considering the diffeomorphism $f^m$ instead of $f$, we can assume that $m=1$. The proof of the general case is similar.

$(1)$ By item 2 of Pliss lemma, for any $\lambda'\in (\lambda,1)$, there are infinitely many $\lambda'$-$E$-Pliss points on $\orb^+(x)$. Take a limit point of these $\lambda'$-$E$-Pliss points, denote it by $y_{\lambda'}$, then $y_{\lambda'}\in \omega(x)$ is a $\lambda'$-$E$-Pliss point. We take a sequence of numbers $(\lambda_n)_{n\geq 1}$ such that $\lambda_n\in (\lambda,1)$ and $\lambda_n\rightarrow \lambda$ when $n$ goes to infinity. Then for any $n\geq 1$, there is a $\lambda_n$-$E$-Pliss point $y_{\lambda_n}\in\omega(x)$. Taking a subsequence if necessary, we assume $(y_{\lambda_n})_{n\geq 1}$ converges to a point $y\in \omega(x)$. Then $y$ is a $\lambda_n$-$E$-Pliss point for any $n\geq 1$. Since $\lambda_n\rightarrow \lambda$, the point $y$ is a $\lambda$-$E$-Pliss point.

$(2)$ By the compactness of $\omega(x)$, there is an integer $N$, such that $n_y\leq N$ for any $y\in\omega(x)$. There is a constant $C>0$, such that, for any $y\in \omega(x)$, we have $\forall n>0$
   \begin{displaymath}
     \prod_{i=0}^{n-1}\|Df|_{E(f^i(y))}\|<C\lambda^n.
   \end{displaymath}
Take a constant $\lambda'\in (\lambda,1)$. Take three constants $\lambda_1<\lambda_2<\lambda_3$ contained in $(\lambda,\lambda')$. There is $N\in\mathbb{N}$, such that $C\lambda^n<\lambda_1^n$ for any $n\geq N$. There is $\vep>0$, such that, for any two points $x_1,x_2\in \Lambda$, if $d(f(x_1),f(x_2))<\vep$, then $\frac{\|Df|_{E(f^i(x_1))}\|}{\|Df|_{E(f^i(x_2))}\|}<\frac{\lambda_2}{\lambda_1}$, for all $i=0,1,\cdots,N$. By considering an iterate of $x$ instead of $x$, we can assume that $d_H(\overline{\orb^+(x)},\omega(x))<\vep$, where $d_H(\cdot,\cdot)$ is the Hausdorff distance. Then for any $n\geq 1$, we have
   \begin{displaymath}
     \prod_{i=0}^{nN}\|Df|_{E(f^i(x))}\|<(C\lambda^N)^n\left(\frac{\lambda_2}{\lambda_1}\right)^{nN}<\lambda_2^{nN}.
   \end{displaymath}
There is $T>0$, such that, for any $k\geq T$, we have $\lambda_2^{kN}\|Df\|^j<\lambda_3^{kN+j}$ for all $j=0,1,\cdots,N-1$. Then for any $n>TN$, assume $n=kN+j$, where $0\leq j<N$, we have
   \begin{displaymath}
     \prod_{i=0}^{n}\|Df|_{E(f^i(x))}\|\leq\left(\prod_{i=0}^{kN}\|Df|_{E(f^i(x))}\|\right)\|Df\|^j<\lambda_2^{kN}\|Df\|^j<\lambda_3^n.
   \end{displaymath}
Then by item $2$ of Pliss lemma, there are infinitely many $(1,\lambda')$-$E$-Pliss points on $\orb^+(x)$.
\end{proof}

\begin{defi}\label{consecutive pliss}
Assume that $\Lambda$ is an invariant compact set of a diffeomorphism $f$ in $\diff^1(M)$ and $E$ is an invariant sub-bundle of $T_{\Lambda}M$. We call two $(m,\lambda)$-$E$-Pliss points $(f^{n}(x),f^{l}(x))$ on a single orbit \emph{consecutive} $(m,\lambda)$-$E$-Pliss points, if $n<l$ and for all $n<k<l$, $f^k(x)$ is not a $(m,\lambda)$-$E$-Pliss point. Furthermore, if there is a dominated splitting $T_{\Lambda}M=E\oplus F$ on $\Lambda$, we call $x\in \Lambda$ an \emph{$(m,\lambda)$-bi-Pliss point}, if it is an $(m,\lambda)$-$E$-Pliss point for $f$ and an $(m,\lambda)$-$F$-Pliss point for $f^{-1}$.
\end{defi}

For Pliss-points, we have the following lemma. The techniques of the proof can be found in many papers, for example~\cite{ps}.

\begin{lem}\label{property of pliss point}
Assume $\Lambda$ is an invariant compact set of a diffeomorphism $f\in \diff^1(M)$ with an $(m,\lambda^2)$-dominated splitting $T_{\Lambda}M=E\oplus F$. Let $x\in\Lambda$ and $\{x_i\}_{i\geq 0}$ be a sequence of points contained in $\Lambda$. We have that, for any $\lambda'\in (\lambda,1)$:
\begin{enumerate}
\item If a sequence of consecutive $(m,\lambda')$-$E$-Pliss points $(f^{n_i}(x_i),f^{l_i}(x_i))_{i\geq 0}$ satisfies that $l_i-n_i\rightarrow +\infty$, then, any limit point $y$ of the sequence $(f^{l_i}(x_i))$ is a $(m,\lambda')$-bi-Pliss point.
\item If there are both $(m,\lambda')$-$E$-Pliss points for $f$ on $\orb^+(x)$ and $(m,\lambda')$-$F$-Pliss points for $f^{-1}$ on $\orb^-(x)$, then there is at least one $(m,\lambda')$-bi-Pliss point on $\orb(x)$.
\item If $x\in \Lambda$ is an $(m,\lambda')$-$E$-Pliss point and there are no other $(m,\lambda')$-$E$-Pliss points on $\orb^-(x)$, then $x$ is also an $(m,\lambda)$-$F$-Pliss point for $f^{-1}$. Thus $x$ is an $(m,\lambda')$-bi-Pliss point.
\end{enumerate}
\end{lem}

We have the following selecting lemma of Liao to get weak periodic orbits (see~\cite{l2},~\cite{w-selecting}).

\begin{lem}[Liao's selecting lemma]\label{selecting}
Assume $f\in \diff^1(M)$. Consider an invariant compact set $\Lambda$ with a non-trivial $(m,\lambda)$-dominated splitting $T_\Lambda M=E\oplus F$, and $\lambda_0\in (\lambda,1)$, suppose that the following two conditions are satisfied:
\begin{itemize}
\item There is a point $b\in \Lambda$, such that, for all $n\geq 1$, we have:

   \begin{displaymath}
     \prod_{i=0}^{n-1} \|Df^m|_{E(f^{im}(b))}\|\geq 1.
   \end{displaymath}

\item For any invariant compact subset $K\subsetneqq \Lambda$, there is an $(m,\lambda_0)$-$E$-Pliss point $x\in K$.
\end{itemize}
Then for any neighborhood $U$ of $\Lambda$, for any $\lambda_1<\lambda_2$ contained in $(\lambda_0,1)$, there is a periodic orbit $\orb(q)\subset U$ with period $\tau(q)$ a multiple of $m$, such that, for all $n=1,\cdots,\tau(q)/m$, the following two inequalities are satisfied:

   \begin{displaymath}
     \prod_{i=0}^{n-1} \|Df^m|_{E(f^{im}(q))}\|\leq {\lambda_2}^n,
   \end{displaymath}
and
   \begin{displaymath}
     \prod_{i=n-1}^{\tau(q)/m-1} \|Df^m|_{E(f^{im}(q))}\|\geq {\lambda_1}^{\tau(q)/m-n+1}.
   \end{displaymath}
Particularly, one can find a sequence of periodic points that are homoclinic related with each other and converges to a point in $\Lambda$. Similar assertions for $F$ hold with respect to $f^{-1}$.
\end{lem}

\subsection{Perturbation techniques}
Consider a diffeomorphism $f$ and a neighborhood $\mathcal U$ of $f$ in $\diff^1(M)$.
For a perturbation $f_1\in\mathcal{U}$ of $f$, if there is an open set $U\subset M$ such that $f_1|_{M\setminus U}=f|_{M\setminus U}$, then $U$ is called the {\it perturbation neighborhood} of $f_1$.
Consider two perturbations $f_1$ and $f_2$ of $f$ with disjoint perturbation neighborhoods $U_1$ and $U_2$ respectively.
In general, the diffeomorphism $g$, where $g|_{M\setminus (U_1\cup U_2)}=f|_{M\setminus (U_1\cup U_2)}$ and $g|_{U_i}=f_i|_{U_i}$ for $i=1,2$, is not contained in $\mathcal U$ any more. However, there is a basis of neighborhoods $\mathcal U$ of $f$, such that if the element of $\mathcal{U}$ is of the form $f\circ\phi$ with $\phi\in\mathcal{V}$, where $\mathcal{V}$ is a $C^1$ neighborhood of $Id$, then $\mathcal{V}$ satisfies the following property (F), see Section 2 of~\cite{pr}.

\begin{defi}[Property (F)]\label{property F}
Assume $\mathcal{V}$ is a $C^1$ neighborhood of $Id$. We say $\mathcal{V}$ satisfies the {\it property (F)}, if for any perturbations $\phi $ and $\phi'$ of $Id$ in $\mathcal{V}$ with disjoint perturbation neighborhoods, the composed perturbation $\phi\circ \phi'$ is still in $\mathcal{V}$.
\end{defi}

We give some tools for $C^1$-perturbation. The first one is the famous Hayashi's connecting lemma, see~\cite{h,wx}.

\begin{theo}[Connecting lemma]\label{Thm:connecting lemma}
Assume that $f$ is a diffeomorphism in $\diff^1(M)$. For any $C^1$ neighborhood $\mathcal{U}$ of $f$ in $\diff^1(M)$, there is an integer $N\in\mathbb N$, satisfying the following properties:\\
Assume $z\in M$ is not a periodic point of period less than or equal to $N$. Then for any neighborhood $U_z$ of $z$, there is a smaller neighborhood $V_z\subset U_z$ of $z$, such that, for any two points $x$ and $y$ that are outside $\Delta=\bigcup_{0\leq i\leq N-1}f^i(U_z)$, if there are two positive integer $n_x$ and $n_y$, such that $f^{n_x}(x)\in V_z$ and $f^{-n_y}(y)\in V_z$, then there are a diffeomorphism $g\in\mathcal{U}$ and a positive integer $m$ such that $g^m(x)=y$ and $g=f$ outside $\Delta$. Moreover, the piece of orbit $\{x,g(x),\cdots,g^m(x)=y\}$ is contained in the set $\{x,f(x),\cdots,f^{n_x}(x)\}\cup \Delta\cup \{y,f^{-1}(y),\cdots,f^{-n_y}(y)\}$ and the number $m$ is no more than $n_x+n_y$.
\end{theo}

Theorem~\ref{Thm:connecting lemma} deals with a single diffeomorphism and a given neighborhood. Below we give a uniform version that is valid in a neighborhood of a diffeomorphism, see~\cite{w2}. We point out that Theorem~\ref{Thm:connecting lemma} is a corollary of Theorem~\ref{uniform connecting}. We put the two theorems here because in the proof of Proposition 1 and Proposition 2, we only need to apply Theorem~\ref{Thm:connecting lemma}, and the notation is more simple. Theorem~\ref{uniform connecting} is applied in the proof of Theorem B in Section~\ref{proof of theorem b} and in the proof of Proposition 3 in Section~\ref{proposition 3}.

\begin{theo}[A uniform connecting lemma, Theorem A of~\cite{w2}]\label{uniform connecting}
Assume that $f$ is a diffeomorphism in $\diff^1(M)$. For any $C^1$ neighborhood $\mathcal{U}$ of $f$ in $\diff^1(M)$, there are three numbers $\rho>1$, $\delta_0>0$ and $N\in\mathbb{N}$, together with a $C^1$ neighborhood $\mathcal{U}_1\subset\mathcal{U}$ of $f$ in $\diff^1(M)$, that satisfy the following property:\\
For any $f_1\in\mathcal{U}_1$, any point $z\in M$ and any number $0<\delta<\delta_0$, as long as the $N$ balls $(f_1^i(B(z,\delta)))_{0\leq i\leq N-1}$ are pairwise disjoint and each is of size smaller than $\delta_0$ (that is to say, $f_1^i(B(z,\delta))\subset B(f_1^i(z),\delta_0)$), then for any two points $x$ and $y$ that are outside the set $\Delta=\bigcup_{0\leq i\leq N-1}f_1^i(B(z,\delta))$, if there are two positive integers $n_x$ and $n_y$ such that $f_1^{n_x}(x)\in B(z,\delta/\rho)$ and $f_1^{-n_y}(y)\in B(z,\delta/\rho)$, then there are a diffeomorphism $g\in\mathcal{U}$ and a positive integer $m$ such that $g^m(x)=y$ and $g=f_1$ outside $\Delta$. Moreover, the piece of orbit $\{x,g(x),\cdots,g^m(x)=y\}$ is contained in the set $\{x,f_1(x),\cdots,f_1^{n_x}(x)\}\cup \Delta\cup \{y,f_1^{-1}(y),\cdots,f_1^{-n_y}(y)\}$ and the number $m$ is no more than $n_x+n_y$.
\end{theo}

To control the perturbing neighborhood when connecting two points that are close, we have the following lemma, see~\cite{a}.

\begin{lem}[Basic perturbation lemma]\label{basic perturbation} For any neighborhood $\mathcal{U}$ of a diffeomorphism $f\in \diff^1(M)$, there are two numbers $\theta>1$ and $r_0>0$ satisfying: for any two points $x,y\in M$ contained in a ball $B(z,r)$, where $r\leq r_0$, there is a diffeomorphism $g\in \mathcal{U}$, such that $g(x)=f(y)$, and $g$ coincides with $f$ outside the ball $B(z,\theta\cdot r)$.
\end{lem}

\begin{defi}
For a chart $\varphi:V\rightarrow \mathbb{R}^d$ of $M$, we call a set $C\subset V$ a \emph{cube} of $\varphi$ if $\varphi(C)$ is the image of $[-a,a]^d$ by a translation of $\mathbb{R}^d$. Here $a$ is called the \emph{radius} of the cube. If a cube $C'\subset V$ is of radius $(1+\vep)a$ and $\varphi(C')$ is of the same center of $\varphi(C)$, we also denote by $C'=(1+\vep)C$.
\end{defi}

\begin{defi}
Consider a chart $\varphi:V\rightarrow \mathbb{R}^d$. A \emph{tiled domain} according to the chart of $\varphi$ is an open set $U\subset V$ and a family $\mathcal{C}$ of cubes of $\varphi$ (called \emph{tiles} of domain), such that:
\begin{enumerate}
\item the interior of the tiles are pairwise disjoint;
\item the union of all tiles of $\mathcal{C}$ equals to $U$;
\item the geometry of the tiling is bounded, i.e.
\end{enumerate}
\begin{itemize}
\item the number of tiles around each point is uniformly bounded (by $2^d$), that is to say, there is a neighborhood for each point that meets at most $2^d$ tiles,

\item for any two pairs $(C,C')$ of intersecting tiles, the rate of their diameters is uniformly bounded (by $2$).
\end{itemize}
\end{defi}

By a standard construction, any open set $U\subset V$ can be tiled according to the coordinates of $\varphi$ (e.g.~\cite{bc,c2}).

\begin{defi}\label{perturbation domain}
Assume $f\in\diff^1(M)$. Consider a neighborhood $\mathcal{U}\subset \diff^1(M)$ and a number $N$. A tiled domain $(U,\mathcal{C})$ is called a \emph{perturbation domain} of order $N$ of $(f,\mathcal{U})$, if the following properties are satisfied.
\begin{enumerate}
\item $U$ is disjoint from its $N$ first iterates of $f$.
\item For any finitely many sequence of pairs of points $\{(x_i,y_i)\}_{1\leq i\leq l}$ in $U$, such that for any $i=1,2,\cdots,l$, the points $x_i$ and $y_i$ are contained in the same tile of $\mathcal{C}$, then there exist:
\begin{itemize}
\item a diffeomorphism $g\in\mathcal{U}$, that coincides with $f$ outside $\bigcup_{0\leq i\leq N-1}f^i(U)$,
\item a strictly increasing sequence $1=n_0<n_1<\cdots<n_k\leq l$, such that $g^N(x_{n_i})=f^N(y_{n_{i+1}-1})$ for any $i\neq k$, and $g^N(x_{n_k})=f^N(y_l)$.
\end{itemize}
\end{enumerate}
The union $\bigcup_{0\leq i\leq N-1}f^i(U)$ is called the \emph{support} of the perturbation domain $(U,\mathcal{C})$ and denoted by $\supp(U)$.
\end{defi}

\begin{defi}\label{jumps}
A pseudo-orbit $(x_0,x_1,\cdots,x_l)$ is said to \emph{keep the tiles} of a perturbation domain $(U,\mathcal{C})$ of order $N$ of $(f,\mathcal{U})$, if the intersection of the pseudo-orbit and $\supp(U)$ is a union of segments $x_{n_i},x_{n_i+1},\cdots,x_{n_i+N-1}$ of the form that $x_{n_i}\in U$ and for any $j=1,2,\cdots,N-1$, $x_{n_i+j}=f^j(y_{n_i})$, where $y_{n_i}$ is a point contained in the same tile of $\mathcal{C}$ as $x_{n_i}$. A pseudo-orbit $(x_0,x_1,\cdots,x_k)$ is said to \emph{have jumps only in tiles} of a perturbation domain $(U,\mathcal{C})$ of order $N$ of $(f,\mathcal{U})$, if it keeps the tiles and for any $x_i\notin \supp(U)$, we have $x_{i+1}=f(x_i)$. For a family of perturbation domains $(U_k,\mathcal{C}_k)_{k\geq 0}$ of order $N_k$ of $(f,\mathcal{U}_k)$ with disjoint support, we say that a pseudo-orbit $(x_0,x_1,\cdots,x_l)$ has \emph{jumps only in tiles} of the perturbation domains $(U_k,\mathcal{C}_k)_{k\geq 0}$, if it keeps the tiles of the perturbation domains and for any $x_i\notin\bigcup_k \supp(U_k)$, we have $x_{i+1}=f(x_i)$.
\end{defi}

By the proof of connecting lemma in~\cite{ar}, the perturbation domain always exists (see also Th\'{e}or\`{e}me 2.1 of~\cite{bc} and Th\'{e}or\`{e}me 3.3 of~\cite{c2}).

\begin{theo}[Another statement of the connecting lemma]\label{existence of perturbation domain}
For any neighborhood $\mathcal{U}$ of $f$, there is an integer $N\geq 1$, and for all point $p\in M$, there is a chart $\varphi:V\rightarrow \mathbb{R}^d$ such that any tiled domain $(U,\mathcal{C})$ according to $\varphi$ disjoint from its $N$ first iterates is a perturbation domain of order $N$ for $(f,\mathcal{U})$.
\end{theo}

From the definitions above, we can get the following lemma easily.

\begin{lem}[Lemme 2.3 of~\cite{bc}]\label{union of perturbation domains}
For a family of disjoint perturbation domains $(U_k,\mathcal{C}_k)$ of order $N_k$ of $(f,\mathcal{U}_k)$ with disjoint support, if there is a pseudo-orbit $(p=p_0,p_1,\cdots,p_m=q)$ that has only jumps in the tiles of $(U_k,\mathcal{C}_k)_{k\geq 0}$ and $p_0,p_m\not \in U_k\cup\cdots\cup f^{N_k-1}(U_k)$ for all $k\geq 0$, then for any $i$, there is $g_i\in \mathcal{U}_i$ and a new pseudo-orbit $(p=p_0',\cdots,p_{m'}'=q)$ of $g_i$ that has only jumps in the tiles of domains $(U_k,\mathcal{C}_k)_{k\geq 0,k\neq i}$. Moreover, $g_i=f$ outside $U_i\cup\cdots\cup f^{N_i-1}(U_i)$ and $\{p_0',\cdots,p_{m'}'\}\setminus (U_i\cup\cdots\cup f^{N_i-1}(U_i))\subset \{p_0,p_1,\cdots,p_m\}$, and $m'\leq m$.
\end{lem}

\subsection{Topological towers}

In this subsection, we introduce two lemmas of~\cite{bc} that are useful to get a true orbit by perturbing a pseudo-orbit. These two lemmas are the key tools in the proof of Proposition~\ref{asymptotic connecting}. First we give the following lemma to choose perturbation neighborhoods. In fact, it is a generalized result of Lemme 3.7 of~\cite{bc}, but one can get the conclusion from the proof in~\cite{bc}.

\begin{lem}\label{choose neighborhoods}
There is a constant $\kappa_d>0$ (which only depends on the dimension d of $M$) satisfying the following property.\\
Consider a diffeomorphism $f\in\diff^1(M)$ and an integer $T>0$. Assume that $W'$ and $V'$ are two compact $d$-dimensional sub-manifolds with boundary, satisfying $V'$ is disjoint from its $\kappa_d T$ first iterates. Then for any neighborhood $U_1$ of $W'$ and any neighborhood $U_2$ of $V'$, there is an open set $S$, such that:
\begin{enumerate}
\item $V'\subset \bigcup_{i=0}^{\kappa_d T}f^{-i}(S)$.
\item $S=W\cup V$ where $W$ and $V$ satisfy the following:
\begin{itemize}
\item $\overline{W}$ and $\overline{V}$ are two compact $d$-dimensional sub-manifolds with boundary;
\item $W'\subset W\subset\overline{W}\subset U_1$;
\item $\overline{V}$ is contained in $U_2\cup f(U_2)\cup\cdots\cup f^{\kappa_dT}(U_2)$ and disjoint from its $T$ first iterates.
\item $\overline{W}\cap (\bigcup_{i=-T}^{T} f^{i}(\overline{V}))=\emptyset$.
\end{itemize}
\end{enumerate}
\end{lem}

\begin{rem}
$(1)$ We point out here that, the two sets $W'$ and $V'$ in Lemma~\ref{choose neighborhoods} correspond to $U$ and $V$ respectively in Lemme 3.7 of~\cite{bc}. Lemme 3.7 of~\cite{bc} assumes more that $W'$ is disjoint from its $T$ first iterates, and in the conclusion the set $W\cup V$ is disjoint from its $T$ first iterates. Moreover, the statement of Lemme 3.7 of~\cite{bc} does not involve the two neighborhoods $U_1$ and $U_2$. But the proof of Lemme 3.7 of~\cite{bc} gives all the information stated in Lemma~\ref{choose neighborhoods}, see~\cite[Page 61--62]{bc}.

$(2)$ In Lemma~\ref{choose neighborhoods}, if we assume more that $W'$ is disjoint with its first $T'$ iterates where $T'\leq T$, then by taking $U_1$ small enough, we can obtain that $W$ is disjoint from its $T'$ first iterates, and the union $W\cup V$ is also disjoint from its $T'$ first iterates.
\end{rem}

Then, we give a lemma of~\cite{bc} for the construction of what they called \textit{topological tower} (see Th\'{e}or\`{e}me 3.1 and Corollaire 3.1 in~\cite{bc}). Denote by $Per_{T}(f)$ the set of periodic orbits with period less than $T$.

\begin{lem}[Topological tower]\label{ttower}
There is a constant $\kappa_d>0$ (which only depends on the dimension d of $M$), such that, for any $T\in \mathbb{N}$, any constant $\delta>0$, any compact set $K$ of $f\in \diff^1(M)$ that does not contain any non-hyperbolic periodic orbits with periods less than $\kappa_d T$ and any neighborhood $U$ of $K$, there exist an open set $V$ and a compact set $D\subset V$, satisfying the following properties:
\begin{enumerate}
\item\label{item:ttower-submanifold} The closure of $V$ is a compact $d$-dimensional sub-manifold with boundary.
\item\label{item:ttower-iterates} For any point $x\in K$ with $x\not \in \bigcup_{p\in Per_{T}(f)}W^s_{\delta}(p)$, there is $n>0$, such that $f^n(x)\in \int(D)$.
\item\label{item:ttower-disjointness} The sets $\overline{V},f(\overline{V}),\cdots,f^{T}(\overline{V})$ are pairwise disjoint.
\item\label{item:ttower-neighborhood} The set $\overline{V}$ is contained in $U\cup f(U)\cup \cdots \cup f^{\kappa_d T}(U)$.
\end{enumerate}
Moreover, the diameter of all connected components of $V$ can be arbitrarily small.
\end{lem}

\begin{rem}
$(1)$ In~\cite{bc}, Th\'{e}or\`{e}me 3.1 is stated for an invariant compact set $K$, and the items 1 and 4 in the conclusion of Lemma~\ref{ttower} are not stated. But from the proof (see~\cite[Page 62--63]{bc}), we can see that the conclusion is also true for non-invariant compact sets and also the items 1 and 4 are true.

$(2)$ We give a sketch of the proof of Lemma~\ref{ttower}. Take $\kappa_d$ to be the constant in Lemma~\ref{choose neighborhoods}. First, one can take a compact sub-manifold $U_0$ of $M$ with boundary that is disjoint from its first $T$ iterates, such that, any point in a small neighborhood $O$ of $Per_{T}(f)$ that is not on the local stable manifold of $Per_{T}(f)$ has a positive iterate in the interior of $U_0$. Then one can take a finite cover of the compact set $K\setminus O$ by open sets $\{V_i\}_{0\leq i\leq r}$ that are disjoint from their first $\kappa_d T$ iterates (they are not disjoint from each other in general). Moreover, for each $0\leq i\leq r$ the closure of $V_i$ is a compact $d$-dimensional sub-manifold with boundary. Then one can apply Lemma~\ref{choose neighborhoods} inductively with respect to $i$ to the triple $(T,U_i,V_i)$, where the triple $(T,U_i,V_i)$ corresponds to $(T,W',V')$ in Lemma~\ref{choose neighborhoods}, and obtain $U_{i+1}$ as $S$ in Lemma~\ref{choose neighborhoods}. Moreover, since in this setting, $U_0$ is disjoint from its first $T$ iterates, one can obtain that $U_i$ is disjoint from its first $T$ iterates for each $1\leq i\leq r+1$. Finally, one can take $V$ to be the interior of $U_{r+1}$. For more details, the reader should refer to~\cite[Page 62--63]{bc}.
\end{rem}

\subsection{Generic properties}

A set $R$ of a topological Baire space $X$ is called a \emph{residual} set, if $R$ contains a dense $G_{\delta}$ set of $X$. We say a property is a \emph{generic} property of $X$, if there is a residual set $R\subset X$, such that each element contained in $R$ satisfies the property. We give some well-known $C^1$-generic properties of diffeomorphisms in the following lemma. These results can be found in many papers such as~\cite{abcdw,bdv,c1,po}.

\begin{lem}\label{generic properties}
There is a residual set $\mathcal{R}$ in $\diff^1(M)$ of diffeomorphisms, such that any $f\in\mathcal{R}$ satisfies the following properties:
\begin{enumerate}
\item The diffeomorphism $f$ is Kupka-Smale: all periodic points of $f$ are hyperbolic and the stable and unstable manifolds of periodic orbits intersect transversely.
\item The periodic points are dense in the chain recurrent set and any chain recurrence class is either a homoclinic class or contains no periodic point.
\item For every periodic point $p$ of $f$, there exists  a $C^1$-neighborhood $\mathcal{U}_1$ of $f$, such that every $g\in\mathcal{U}_1\cap\mathcal{R}$ is a continuity point for the map $g\mapsto H(p_g,g)$ where $p_g$ is the continuation of $p$ for $g$, where the continuity is with respect to the Hausdorff distance between compact subsets of $M$.
\item If $H(p)$ is a homoclinic class of $f$, then there exists an interval $[\alpha,\beta]$ of natural numbers and a $C^1$-neighborhood $\mathcal{U}_2$ of $f$, such that for every $g\in\mathcal{U}_2$, the set of indices of hyperbolic periodic points contained in $H(p_g,g)$ is $[\alpha,\beta]$. Also, all periodic points of the same index contained in $H(p)$ are homoclinically related.
\item If a homoclinic class $H(p)$ contains periodic orbits with different indices, then $f$ can be $C^1$ approximated by diffeomorphisms having a heterodimensional cycle.
\item If a homoclinic class $H(p)$ is Lyapunov stable, then there is a $C^1$ neighborhood $\mathcal{U}_3$ of $f$, such that for any $g\in\mathcal{U}_3\cap \mathcal{R}$, the homoclinic class $H(p_g,g)$ is also Lyapunov stable.
\item For any two points $x,y\in M$, if $x\dashv y$, then $x\prec y$.
\end{enumerate}
\end{lem}

\section{Norm of products and product of norms: reduction of the proof of Theorem A}\label{theorem b}

Theorem A essentially follows from the theorem below.
\begin{theob}
For $C^1$-generic $f\in \diff^1(M)$, assume that $p$ is a hyperbolic periodic point of $f$ and that the homoclinic class $H(p)$ has a dominated splitting $T_{H(p)}M=E\oplus F$, with $dim E\leq \ind(p)$, such that the bundle $E$ is not contracting. Then there are a constant $\lambda_0\in (0,1)$, an integer $m_0\in \mathbb{N}$, satisfying: for any $m\in \mathbb{N}$ with $m\geq m_0$, any constants $\lambda_1,\lambda_2\in (\lambda_0,1)$ with $\lambda_1<\lambda_2$, there is a sequence of different periodic orbits $\mathcal{O}_k=\orb(q_k)$ with period $\tau(q_k)$ contained in $H(P)$, such that
   \begin{displaymath}
        {\lambda_1}^{\tau(q_k)}< \prod_{0\leq i<\tau(q_k)/m} \|Df^m|_{E(f^{im}(q_k))}\|< {\lambda_2}^{\tau(q_k)}.
   \end{displaymath}

\end{theob}

From Theorem B, we can get periodic orbits that have certain controls of the product of norms along the bundle $E$. To control Lyapunov exponents of the periodic orbits, we have to control the norm of products along the bundle $E$. We use the following two lemmas. The first is a perturbation lemma for matrices to control exponents, see~\cite{c2,p} (also see~\cite{l2,m2}).

\begin{lem}\label{matrix}
For any integer $d\geq 1$, $K\geq 1$, any constant $\vep>0$ and $\lambda>0$, there are two integers $N$ and $\tau_0$, such that for any $A_1,\cdots,A_{\tau}$ in $GL(d,\mathbb{R})$ with $\tau\geq \tau_0$, and $\max_{1\leq i\leq \tau}\{\|A_i\|,\|A_i^{-1}\|\}\leq K$, if
   \begin{displaymath}
     \prod_{0\leq i< \tau/N}\|A_{(i+1)N}\cdots A_{iN+2}A_{iN+1}\|\geq \lambda^{\tau},
   \end{displaymath}
then, there are $B_1,\cdots,B_{\tau}$ in $GL(d,\mathbb{R})$, with $\|B_i-A_i\|<\vep$ and $\|B_i^{-1}-A_i^{-1}\|<\vep$, for all $i=1,\cdots,\tau$, such that the maximal absolute value of eigenvalues of $B_{\tau}\circ\cdots \circ B_2\circ B_1$ is bigger than $\lambda$.
\end{lem}

\begin{rem}
In~\cite{c2}, it is presented for the constant $\lambda=1$. If $\lambda\neq 1$, then by considering $A_i'=\lambda^{-1} Id\circ A_i$ and applying the special case for the constant $1$, we can get the general statement as above.
\end{rem}

\begin{cor}\label{cor of matrix}
For any integer $d\geq 1$, $K\geq 1$, any constant $\vep>0$ and $\lambda_1<\lambda_2$, there are two integers $N$ and $\tau_0$, such that for any $A_1,\cdots,A_{\tau}$ in $GL(d,\mathbb{R})$ with $\tau\geq \tau_0$, and $\max_{1\leq i\leq \tau}\{\|A_i\|,\|A_i^{-1}\|\}\leq K$, if
   \begin{displaymath}
     \lambda_1^{\tau}<\prod_{0\leq i< \tau/N}\|A_{(i+1)N}\cdots A_{iN+2}A_{iN+1}\|< \lambda_2^{\tau},
   \end{displaymath}
then, there are $B_1,\cdots,B_{\tau}$ in $GL(d,\mathbb{R})$, with $\|B_i-A_i\|<\vep$ and $\|B_i^{-1}-A_i^{-1}\|<\vep$, for all $i=1,\cdots,\tau$, such that the maximal norm of eigenvalue of $B_{\tau}\circ\cdots \circ B_2\circ B_1$ is in the interval $(\lambda_1,\lambda_2)$.
\end{cor}

\begin{proof}
We take $\vep$ small enough such that, for any $A\in GL(d,\mathbb{R})$, if $\|A^{-1}\|\leq K$, then $B(A,\vep)\in GL(d,\mathbb{R})$, where $B(A,\vep)$ is the $\vep$ ball of $A$. By the assumption of $A_i$, we have that the maximal norm of eigenvalue of $A_{\tau}\circ\cdots \circ A_2\circ A_1$ is smaller than $\lambda_2$. By Lemma~\ref{matrix}, we can get $B^0_1,\cdots,B^0_{\tau}$ in $GL(d,\mathbb{R})$ that satisfies the conclusion for the number $\lambda_1$. We take a path $A_{i,t}|_{0\leq t\leq 1}$ contained in $B(A_i,\vep)$ that connects $A_i$ to $B^0_i$. We have that the maximal norm of eigenvalue of $B^0_{\tau}\circ\cdots \circ B^0_2\circ B^0_1$ is bigger than $\lambda_1$. Then there must be a time $0<t<1$, such that the maximal norm of eigenvalue of $A_{\tau,t}\circ\cdots \circ A_{2,t}\circ A_{1,t}$ is in the interval $(\lambda_1,\lambda_2)$. We take $B_i=A_{i,t}$ and get the conclusion.
\end{proof}

The next lemma is a generalized Frank's lemma by N. Gourmelon that preserves some pieces of invariant manifolds of hyperbolic period orbits, see~\cite{g2}.

\begin{lem}\label{f-g lemma}
Consider a constant $\vep>0$, a diffeomorphism $f\in \diff^1(M)$ and a hyperbolic periodic orbit $\mathcal{O}=\orb(q)$ of $f$ with period $\tau$. Assume there is a one-parameter family of linear maps $(A_{n,t})_{n=0,1,\cdots,\tau-1;t\in[0,1]}$ in $GL(d,\mathbb{R})$, satisfying:
\begin{itemize}
\item $(1)$ $A_{n,0}=Df(f^n(q))$,
\item $(2)$ for all $n=0,1,\cdots,\tau-1$ and $t\in[0,1]$, we have $\|Df(f^n(q))-A_{n,t}\|<\vep$ and $\|Df^{-1}(f^{n+1}(q))-A^{-1}_{n,t}\|<\vep$,
\item $(3)$ $A_{\tau-1,t}\circ\cdots\circ A_{0,t}$ is hyperbolic for all $t\in [0,1]$.\\
\end{itemize}
Then, for any neighborhood $V$ of $\mathcal{O}$, any $\delta>0$, and any pair of compact sets $K^s\subset W^s_{\delta}(\mathcal{O},f)$ and $K^u\subset W^u_{\delta}(\mathcal{O},f)$ disjoint from $V$, there is a diffeomorphism $g\in \diff^1(M)$ that is $\vep$-$C^1$ close to $f$, such that:
\item $(a)$ $g$ coincides with $f$ on $\mathcal{O}$ and outside $V$;
\item $(b)$ $K^s\subset W^s_{\delta}(\mathcal{O},g)$ and $K^u\subset W^u_{\delta}(\mathcal{O},g)$;
\item $(c)$ $Dg(g^n(q))=Dg(f^n(q))=A_{n,1}$ for all $n=0,\cdots,\tau-1$.

\end{lem}

Now we give the proof of Theorem A from Theorem B.

\begin{proof}
By Theorem B, we get two constants $\lambda_0\in (0,1)$ and $m_0\in \mathbb{N}$. We prove that for any $\lambda_0<\lambda_1<\lambda_2<1$ and any $\vep>0$, there is a diffeomorphism $g$ that is $C^1$-$\vep$ close to $f$ having a periodic orbit $\orb(q)$ homoclinic related to $p_g$ such that the largest Lyapunov exponent along $E$ of $\orb(q)$ is in the interval $(\log\lambda_1,\log\lambda_2)$. Then by the genericity of $f$ and Lemma 2.1 of~\cite{gy}, $f$ itself has such periodic orbits. Since $\lambda_1$ can be taken arbitrarily close to $1$, we get the conclusion of Theorem A.


Take $d=\dim(M)$ and $K=\max\{\|Df\|,\|Df^{-1}\|\}$. Now we fix the constants $\vep>0$ and $\lambda_1<\lambda_2$ in $(\lambda_0,1)$.
Since $E\oplus F$ is a dominated splitting, the two bundles $E$ and $F$ are transverse with each other, thus the angle between $E$ and $F$ has a lower bound. As a result, the perturbation of $f$ along the periodic orbit $\orb(q)$ can be realized by the perturbation restricted to the derivative of $f$ along the two bundles $E$ and $F$. That is to say, for the constant $\vep>0$, there is $\vep'>0$, such that any $\vep'$ perturbation of $Df$ on the bundles $E$ and $F$ independently gives an $\vep$ perturbation of $f$. For $\vep'>0$, we get two integers $N$ and $\tau_0$ by Corollary~\ref{cor of matrix}.

By Theorem B, there is a periodic orbit $\orb(q)$ of $f$ with period $\tau>\tau_0$ that is homoclinically related to $\orb(p)$, such that,
   \begin{displaymath}
        {\lambda_1}^{\tau}< \prod_{0\leq i<\tau/m} \|Df^m|_{E(f^{im}(q))}\|< {\lambda_2}^{\tau},
   \end{displaymath}
where $m>m_0$ is a multiple of $N$. Denote $A_i=Df|_{f^i(q)}$ for $i=0,\cdots,\tau-1$. By Corollary~\ref{cor of matrix}, there are $B_0,\cdots,B_{\tau-1}$ in $GL(d,\mathbb{R})$, with $\|(B_i-A_i)|_E\|<\vep'$ and $\|(B_i^{-1}-A_i^{-1})|_E\|<\vep'$, for all $i=0,\cdots,\tau-1$, and $B_i$ coincides with $A_i$ along the bundle $F$ such that the maximal norm of eigenvalue of $B_{\tau-1}\circ\cdots \circ B_1\circ B_0$ along the bundle $E$ is in the interval $(\lambda_1^{\tau},\lambda_2^{\tau})$. By the choice of $\vep'$, we have that $\|B_i-A_i\|<\vep$ and $\|B_i^{-1}-A_i^{-1}\|<\vep$. We take a path $A_{i,t}|_{0\leq t\leq 1}$ contained in $B(A_i,\vep)$ that connects $A_i$ to $B_i$ such that $A_{i,t}$ coincides with $A_i$ along the bundle $F$ for all $i=0,\cdots,\tau-1$ and all $t\in (0,1)$. If there is a time $t\in (0,1)$ such that $A_{\tau-1,t}\circ\cdots A_{0,t}$ is not hyperbolic, then there must be a time $t_0<t$, such that $A_{\tau-1,s}\circ\cdots A_{0,s}$ is hyperbolic for all $0\leq s\leq t_0$, and the maximal norm of eigenvalue of $A_{\tau-1,t_0}\circ\cdots A_{0,t_0}$ along the bundle $E$ is in the interval $(\lambda_1^{\tau},\lambda_2^{\tau})$. Otherwise, we can take $t_0=1$.

Take a small number $\delta>0$, since $\orb(q)$ is homoclinically related to $\orb(p)$, there exist two points $x\in W^s_{\delta}(\orb(q))\pitchfork W^u(\orb(p))$ and $y\in W^u_{\delta}(\orb(q))\pitchfork W^s(\orb(p))$. Consider the two compact sets $K^s=\{x\}$ and $K^u=\{y\}$. We take a neighborhood $V$ of $\orb(q)$ such that $V\cap (K^s\cup K^u)=\emptyset$ and $V\cap (\overline{\orb^-(x)\cup \orb^+(y)})=\emptyset$. By Lemma~\ref{f-g lemma}, considering the one-parameter family of linear maps $(A_{i,t})_{i=0,\cdots,\tau-1;t\in [0,t_0]}$, there is a diffeomorphism $g$ that is $C^1$-$\vep$ close to $f$, such that:
\item $(a)$ $g$ coincides with $f$ on $\orb(q)$ and outside $V$;
\item $(b)$ $K^s\subset W^s_{\delta}(\orb(q),g)$ and $K^u\subset W^u_{\delta}(\orb(q),g)$;
\item $(c)$ $Dg(g^i(q))=Dg(f^i(q))=A_{i,t_0}$ for all $i=0,\cdots,\tau-1$.
\medskip

By item $(a)$ and the choice of $V$, we have that $g$ coincides with $f$ on $\orb^-(x)\cup \orb^+(y)$, hence we have that $x\in W^u(\orb(p),g)$ and $y\in W^s(\orb(p),g)$. Then by item $(b)$, we can see that $x\in W^s_{\delta}(\orb(q),g)\cap W^u(\orb(p),g)$ and $y\in W^u_{\delta}(\orb(q),g)\cap W^s(\orb(p),g)$. By another small perturbation if necessary, we can assume that the two intersections are transverse. Then the two periodic orbits $\orb(q)$ and $\orb(p)$ of $g$ are still homoclinically related with each other, and the largest Lyapunov exponent of $\orb(q)$ along the bundle $E$ under the diffeomorphism $g$ is in the interval $(\log\lambda_1,\log\lambda_2)$. This finishes the proof of Theorem A.
\end{proof}

\section{Non-hyperbolicity implies existence of weak periodic orbits: proof of Theorem B}\label{proof of theorem b}

This section will give the proof of Theorem B. We assume that $\mathcal{R}$ is the residual set of $\diff^1(M)$ stated in Lemma~\ref{generic properties} and $f\in\mathcal{R}$ is a diffeomorphism that satisfies the hypotheses of Theorem B. Later we will assume also that $f$ belongs to another two residual subsets $\mathcal{R}_0$ and $\mathcal{R}_1$ defined below.

Since $E\oplus F$ is a dominated splitting and $dim E\leq \ind(p)$, we have that: there are $\lambda_0\in (0,1)$ and $m_0\in\mathbb{N}$, such that, for any $m\geq m_0$, the splitting $E\oplus F$ is $(m,\lambda_0^2)$-dominated, and, for the hyperbolic periodic orbit $\orb(p)$,
   \begin{displaymath}
     \|Df^{\tau(p)}|_{E(p)}\|<\lambda_0^{\tau(p)},
   \end{displaymath}
where $\tau(p)$ is the period of $\orb(p)$. In the following, we fix the constant $\lambda_0$ and the integer $m\geq m_0$. In order to simplify the notations, we will assume that $m=1$ and that $p$ is a fixed point of $f$, but the general case is identical.

\subsection{Existence of weak sets}

\begin{lem}\label{existence of weak sets}
For any $\lambda\in (\lambda_0,1)$, there is a $\lambda$-$E$-weak set contained in $H(p)$.
\end{lem}

\begin{proof}
Since $E$ is not contracting, by a compactness argument, there is a point $b\in H(p)$, such that, for any $n\geq 1$,
   \begin{displaymath}
     \prod_{i=0}^{n-1} \|Df|_{E(f^{i}(b))}\|\geq 1.
   \end{displaymath}
Then the first assumption for the bundle $E$ in Lemma~\ref{selecting} is satisfied.

Assume by contradiction that there is a constant $\lambda\in (\lambda_0,1)$, such that there is no $\lambda$-$E$-weak set contained in $H(p)$. Thus the second assumption in Lemma~\ref{selecting} is satisfied for the bundle $E$ and the constant $\lambda$. Hence, for any $\lambda_1,\lambda_2\in (\lambda,1)$ with $\lambda_1<\lambda_2$, there is a sequence of periodic orbits $\orb(q_k)$ with period $\tau(q_k)$ that are homoclinically related with each other and that converges to a subset of $H(p)$ such that for any $k\geq 0$, the following properties are satisfied:
   \begin{displaymath}
     {\lambda_1}^{\tau(q_k)}\leq \prod_{0\leq i<\tau(q_k)} \|Df|_{E(f^{i}(q_k))}\|\leq {\lambda_2}^{\tau(q_k)},
   \end{displaymath}
Then $H(p)=H(q_k)$ by item 2 of Lemma~\ref{generic properties}, hence $q_k\in H(p)$. It is obvious that $\orb(q_k)$ is a $\lambda_1$-$E$-weak set contained in $H(p)$, thus is also a $\lambda$-$E$-weak set. This contradicts the assumption that there is no $\lambda$-$E$-weak set contained in $H(p)$.
\end{proof}

\subsection{Existence of a bi-Pliss point accumulating backward to an $E$-weak set}\label{bi pliss point and weak set}
From now on, we fix any two numbers $\lambda_1<\lambda_2$ in $(\lambda_0,1)$. Then there is a $\lambda_2$-$E$-weak set contained in $H(p)$. By the domination, any $\lambda_2$-$E$-weak set $K$ is $(C,\lambda_0,F)$-expanding for some constant $C>0$ depending on $K$. By~\cite{hps}, any point $x\in K$ has a uniform local unstable manifold $W^u_{loc}(x)$ with uniform size depending on $K$. Recall Remark~\ref{limit of pliss point}, any limit point of a sequence of $\lambda$-$E$-Pliss points is still a $\lambda$-$E$-Pliss point.

We extend the dominated splitting $E\oplus F$ to the maximal invariant compact set of a small neighborhood $U$ of $H(p)$ and denote it still by $E\oplus F$. We take a constant $\lambda_3\in (\lambda_2,1)$.

\begin{lem}\label{weak set and pliss point}
There are a $\lambda_2$-$E$-weak set $K$ and a $\lambda_3$-bi-Pliss point $x\in H(p)\setminus K$ satisfying: $\alpha(x)= K$.
\end{lem}

It is obvious that any compact invariant subset of a $\lambda_2$-$E$-weak set is still a $\lambda_2$-$E$-weak set. So we only have to prove that: \emph{there are a $\lambda_2$-$E$-weak set $K$, and a $\lambda_3$-bi-Pliss point $x\in H(p)\setminus K$ satisfying: $\alpha(x)\subset K$.}

\begin{proof}
By Lemma~\ref{existence of weak sets}, there exists a $\lambda_2$-$E$-weak set in $H(p)$. To prove this lemma, we consider two cases: either all the $\lambda_2$-$E$-weak sets are uniformly $E$-weak or not. More precisely, if we take the closure of the union of all $\lambda_2$-$E$-weak sets contained in $H(p)$, and denote it by $\hat K$, then there are two cases: either $\hat K$ is still a $\lambda_2$-$E$-weak set or not.

\subsubsection{The uniform case: $\hat K$ is a $\lambda_2$-$E$-weak set}

In this case, $\hat K$ is the maximal $\lambda_2$-$E$-weak set in $H(p)$ and we will take $K=\hat{K}$.

\begin{claim}\label{maximal weak set}
The set $K$ is locally maximal in $H(p)$, and there exists a point $z\in H(p)\setminus K$, such that $\alpha(z)\subset K$.
\end{claim}

\begin{proof}
We prove by contradiction. Assume that $K$ is not locally maximal in $H(p)$. Take a decreasing sequence of neighborhoods $(U_n)_{n\geq 0}$ of $K$, such that $\cap_n U_n=K$. Then for any $n\geq 0$, there is a compact invariant set $K_n\subset U_n\cap H(p)$ such that $K\subsetneqq K_n$. Since $K$ is the maximal $\lambda_2$-$E$-weak set in $H(p)$, we have that $K_n$ is not a $\lambda_2$-$E$-weak set, thus there is a $\lambda_2$-$E$-Pliss point $y_n\in K_n$. Take a converging subsequence of $(y_n)$, and assume $y$ is the limit point. Then we have that $y\in K$ and $y$ is a $\lambda_2$-$E$-Pliss point. This contradicts the fact that $K$ is a $\lambda_2$-$E$-weak set.

Since $K$ is locally maximal in $H(p)$, there is a neighborhood $U$ of $K$ such that $K$ is the maximal compact invariant set contained in $\overline{U}\cap H(p)$. Now we consider the diffeomorphism $f^{-1}$. Since $K\subset H(p)$, we have that $p\prec K$ under $f^{-1}$. By Lemma~\ref{prec}, there exists a point $z\in U\setminus K$, such that $p\prec_M z\prec_U K$ under $f^{-1}$ and $\orb^+(z,f^{-1})\subset U$. Since $H(p)$ is a chain recurrence class both for $f$ and $f^{-1}$ and $K\subset H(p)$, the fact $p\prec z\prec K$ under $f^{-1}$ implies that $z\in H(p)$. Since $\orb^+(z,f^{-1})\subset U$ we have that $\alpha(z,f)=\omega(z,f^{-1})$ is an $f$-invariant compact set contained in $\overline{U}\cap H(p)$. By the fact that $K$ is the maximal compact invariant set contained in $\overline{U}\cap H(p)$, one has $\alpha(z)\subset K$. This ends the proof of Claim~\ref{maximal weak set}.
\end{proof}

Now we take the point $z\in H(p)\setminus K$ satisfying $\alpha(z)\subset K$ from Claim~\ref{maximal weak set}.

\begin{claim}\label{omega z}
There exists at least one $\lambda_2$-$E$-Pliss point contained in $\omega(z)$.
\end{claim}

\begin{proof}
We prove this claim by contradiction. If $\omega(z)$ contains no $\lambda_2$-$E$-Pliss points, by item $(1)$ of Corollary~\ref{cor of pliss}, $\orb(z)\cup \omega(z)$ contains no $\lambda_2$-$E$-Pliss points. Then $K\cup \orb(z)\cup \omega(z)$ is a $\lambda_2$-$E$-weak set, which contradicts the maximality of $\lambda_2$-$E$-weak set $K$ since $z\not\in K$. Thus $\omega(z)$ contains at least one $\lambda_2$-$E$-Pliss point.
\end{proof}

Since $K$ is a $\lambda_2$-$E$-weak set, by the domination, for any point $w\in K$, there is an integer $n_w$, such that $\prod_{i=0}^{n_w-1}\|Df^{-1}|_{F(f^{-i}(w))}\|\leq \left(\frac{\lambda_0^2}{\lambda_2}\right)^{n_w}<{\lambda_0}^{n_w}$. By item $2$ of Corollary~\ref{cor of pliss}, considering the bundle $F$, there are infinitely many $\lambda_1$-$F$-Pliss points for $f^{-1}$ on $\orb^-(z)$. We take all the $\lambda_1$-$F$-Pliss points $\{f^{n_i}(z)\}$ with $n_{i+1}>n_i$ on $\orb(z)$. Notice that the index $i\in \mathbb{Z}$ but not $i\geq 0$, since there are infinitely many $\lambda_1$-$F$-Pliss points for $f^{-1}$ on $\orb^-(z)$. We consider the following two cases:
\begin{itemize}
\item $\textbf{(a)}$ either the sequence $(n_i)$ has an upper bound or $(n_{i+1}-n_i)$ can be arbitrarily large;

\item $\textbf{(b)}$ the sequence $(n_i)$ has no upper bounds and $(n_{i+1}-n_i)$ is bounded.
\end{itemize}
\begin{claim}
In case \textbf{(a)}, there exists a $\lambda_2$-$E$-Pliss point $y\in H(p)$, such that, for any $\delta>0$, there is $n_i\in\mathbb{Z}$, satisfying $d(y,f^{n_i}(z))<\delta$. Thus, by taking $\delta$ small enough, we can take $x\in W^u(f^{n_i}(z))\cap W^s(y)$, such that $x$ is a $\lambda_3$-bi-Pliss point.
\end{claim}

\begin{proof}
If the sequence $\{n_i\}$ has an upper bound, we take the maximal $n_i$. That is to say, $f^{n_i}(z)$ is a $\lambda_1$-$F$-Pliss point for $f^{-1}$, and, there is no $\lambda_1$-$F$-Pliss point for $f^{-1}$ on $\orb^+(f^{n_i}(z))$. By item $3$ of Lemma~\ref{property of pliss point}, we have that $f^{n_i}(z)$ is also a $\lambda_1$-$E$-Pliss point, thus $f^{n_i}(z)$ is a $\lambda_1$-bi-Pliss point. We take $x=y=f^{n_i}(z)$ in this case.

Otherwise, the sequence $\{n_i\}$ has no upper bounds but $(n_{i+1}-n_i)$ can be arbitrarily large. By item $1$ of Lemma~\ref{property of pliss point}, we can take a subsequence of $\{n_i\}$ such that $f^{n_i}(z)$ converges to a $\lambda_1$-bi-Pliss point $y\in \omega(z)$. Then for any $\delta>0$, we can take $n_i$ large enough, such that $d(y,f^{n_i}(z))<\delta$, and moreover, we can take $x\in W^u(f^{n_i}z)\cap W^s(y)$, such that, $d(f^j(x),f^j(y))<\delta$ and $d(f^{-j}(x),f^{-j}(f^{n_i}z))<\delta$, for all $j\geq 0$. Thus by taking $\delta$ small enough, $x$ is a $\lambda_3$-bi-Pliss point.
\end{proof}

\begin{claim}
In case \textbf{(b)}, there is a $\lambda_2$-$E$-Pliss point $y\in \omega(z)$, such that, there is $n\in\mathbb{N}$, satisfying $W^u(f^n(z))\cap W^s(y)\neq \emptyset$. Thus we can take a point $\bar{x}\in W^u(f^n(z))\cap W^s(y)$, such that $\orb(\bar{x})$ contains some $\lambda_3$-bi-Pliss point.
\end{claim}

\begin{proof}
In this case, there are infinitely many $\lambda_1$-$F$-Pliss points for $f^{-1}$ on $\orb^+(z)$, and the time between any consecutive $\lambda_1$-$F$-Pliss points for $f^{-1}$ on $\orb^+(z)$ is bounded. Then for any point $w\in\overline{\orb^+(z)}$, there is an integer $n_w\in\mathbb{N}$, such that $\prod_{i=0}^{n_w-1}\|Df^{-1}|_{F(f^{-i}(w))}\|\leq {\lambda_1}^{n_w}$. Hence $\overline{\orb^+(z)}$ is a positively invariant $F$-expanding compact set, and any point $w\in\overline{\orb^+(z)}$ has a uniform unstable manifold. By Claim~\ref{omega z}, there is a $\lambda_2$-$E$-Pliss point $y\in \omega(z)$. For any $\delta>0$, there is $n\in\mathbb{N}$, such that, $d(y,f^n(z))<\delta$, and $W^u(f^n(z))\cap W^s(y)\neq \emptyset$. We take $\bar{x}\in W^s(y)\cap W^u(f^n(z))$. Then $\alpha(\bar{x})=\alpha(z)$ and by item $2$ of Corollary~\ref{cor of pliss}, there are $\lambda_3$-$F$-Pliss points for $f^{-1}$ on $\orb^-(\bar{x})$. Also by taking $\delta$ small enough, $d(f^i(\bar{x}),f^i(y))$ can be small for all $i\geq 0$. Since $y$ is a $\lambda_2$-$E$-Pliss point, we can take $\bar{x}$ to be a $\lambda_3$-$E$-Pliss point. Then, by item $2$ of Lemma~\ref{property of pliss point} there exists a $\lambda_3$-bi-Pliss point $x$ on $\orb(\bar{x})$. Then we have that $\alpha(x)=\alpha(\bar{x})=\alpha(z)\subset K$.
\end{proof}

From the above two claims, we get a $\lambda_3$-bi-Pliss point $x\in H(p)$, such that $\alpha(x)\subset K$. We have to show that $x\not\in K$. Notice that in the two cases, we both have $\omega(x)=\omega(y)$ where $y$ is a $\lambda_2$-$E$-Pliss point. By item $1$ of Corollary~\ref{cor of pliss}, $\omega(x)$ contains some $\lambda_2$-$E$-Pliss point. Since $K$ contains no $\lambda_2$-$E$-Pliss point, we have that $x\notin K$. This ends the proof of the uniform case.

\subsubsection{The non-uniform case: $\hat K$ is not a $\lambda_2$-$E$-weak set}

Recall that for any point $z$ contained in a $\lambda_2$-$E$-weak set, there is an integer $n_z$, such that
   \begin{displaymath}
     \prod_{i=0}^{n_z-1} \|Df|_{E(f^{i}(z))}\|> {\lambda_2}^{n_z}.
   \end{displaymath}
We denote $N_z$ the smallest integer that satisfies the above inequality.

\begin{claim}\label{almost bi-pliss point}
In this case, for any number $L>0$ there are a $\lambda_2$-$E$-weak set $K_L$ and a point $z\in K_L$, such that $z$ is a $\lambda_1$-$F$-Pliss point for $f^{-1}$, and, $N_z>L$.

\end{claim}

\begin{proof}
Since $\hat K$ is not a $\lambda_2$-$E$-weak set, for any number $L>0$ there is a $\lambda_2$-$E$-weak sets $K_L$, and a point $z\in K_L$, such that $N_z>L$, that is to say, for $1\leq n\leq N_z$,
   \begin{displaymath}
     \prod_{i=0}^{n-1}\|Df|_{E(f^i(z))}\|\leq \lambda_2^n,
   \end{displaymath}
and
   \begin{displaymath}
     \prod_{i=0}^{N_z-1}\|Df|_{E(f^i(z))}\|>\lambda_2^{N_z}.
   \end{displaymath}
We only have to show that we can choose $z$ to be a $\lambda_1$-$F$-Pliss point for $f^{-1}$. Since $K_L$ is a $\lambda_2$-$E$-weak set, similarly to the arguments above, by item $2$ of Corollary~\ref{cor of pliss}, there are $\lambda_1$-$F$-Pliss points for $f^{-1}$ on $\orb^-(z)$. If $z$ is not a $\lambda_1$-$F$-Pliss point for $f^{-1}$, we can take the minimal number $l\in \mathbb{N}$ such that $w=f^{-l}(z)$ is a $\lambda_1$-$F$-Pliss point for $f^{-1}$. We will prove that $N_w\geq N_z+l>L$. Hence by replacing $z$ by $w$, we get the conclusion of the claim. To prove this, we only have to show that for any $1\leq n\leq l$,
   \begin{displaymath}
     \prod_{i=0}^{n-1}\|Df|_{E(f^i(w))}\|\leq \lambda_2^n.
   \end{displaymath}
We prove this by contradiction. If the above statement is not true, then there is an integer $k\in\{1,2,\cdots,l\}$, such that
   \begin{displaymath}
     \prod_{i=0}^{k-1}\|Df|_{E(f^i(w))}\|>\lambda_2^k,
   \end{displaymath}
and for any $1\leq n< k$,
   \begin{displaymath}
     \prod_{i=0}^{n-1}\|Df|_{E(f^i(w))}\|\leq\lambda_2^n.
   \end{displaymath}
Thus, we have, for any $1\leq n\leq k$,
   \begin{displaymath}
     \prod_{i=1}^{n}\|Df|_{E(f^{k-i}(w))}\|=\left(\prod_{i=0}^{k-1}\|Df|_{E(f^{i}(w))}\|\right)/\left(\prod_{i=0}^{k-n-1}\|Df|_{E(f^i(w))}\|\right)>\lambda_2^n.
   \end{displaymath}
By the domination of $E\oplus F$, we have, for all $1\leq n\leq k$
   \begin{displaymath}
     \prod_{i=0}^{n-1}\|Df^{-1}|_{F(f^{k-i}(w))}\|\leq(\frac{\lambda_0^2}{\lambda_2})^n\leq \lambda_1^n.
   \end{displaymath}
Moreover, since $w$ is a $\lambda_1$-$F$-Pliss point for $f^{-1}$, we will have, for any $n\geq 1$,
   \begin{displaymath}
     \prod_{i=0}^{n-1}\|Df^{-1}|_{F(f^{k-i}(w))}\|\leq\lambda_1^n.
   \end{displaymath}
Thus $f^k(w)=f^{-l+k}(z)$ is a $\lambda_1$-$F$-Pliss point, which contradicts the choice of $w$. This finishes the proof of Claim~\ref{almost bi-pliss point}.
\end{proof}

By taking $L$ large enough, the point $z$ in Claim~\ref{almost bi-pliss point} is close to a $\lambda_2$-$E$-Pliss point $y\not\in K_L$. Since $z$ has a uniform local unstable manifold and $y$ has a uniform local stable manifold, we have that $W^s(y)\cap W^u(z)\neq \emptyset$ if we take these two points close enough. We take the proper $L$, $z$ and $K_L$, satisfying this property. Let $K=K_L$. We explain that the $\lambda_2$-$E$-weak set $K$ satisfies Lemma~\ref{weak set and pliss point}.

By similar arguments as in the proof of Case $\textbf{(b)}$ in the uniform case, we can take a point $\bar{x}\in W^s(y)\cap W^u(z)$ satisfying that $\orb(\bar{x})$ contains a $\lambda_3$-bi-Pliss point $x\in H(p)$. Then we have that $\alpha(x)=\alpha(z)\subset K$. Moreover, since $y$ is a $\lambda_2$-$E$-Pliss point, we know that $\omega(y)$ contains $\lambda_2$-$E$-Pliss points by item $1$ of Corollary~\ref{cor of pliss}. Hence $\omega(x)$ contains $\lambda_2$-$E$-Pliss points because $\omega(x)=\omega(y)$. This implies $x\notin K$ since $K$ is a $\lambda_2$-$E$-weak set. To sum up, we have obtained a $\lambda_2$-$E$-weak set $K$ and a $\lambda_3$-bi-Pliss point $x\in H(p)\setminus K$, satisfying that$\alpha(x)\subset K$. This finishes the proof of Lemma~\ref{weak set and pliss point}.
\end{proof}

\subsection{Continuation of Pliss points}\label{choice of f}

Denote by $\mathcal{M}$ the space of all compact subsets of $M$, associated with the Hausdorff topology. Denote by $\mathcal{S}$ the space of all finite subsets of $M\times\mathcal{M}$ associated with the Hausdorff topology. For any positive integer $N\in\mathbb{N}$, and a diffeomorphism $g\in\diff^1(M)$, denote by $Per_N(g)$ the set of periodic points of $g$ with period less than or equal to $N$, and denote by $\mathcal{C}(q,g)$ the chain recurrence class of a periodic point $q$ of $g$. It is well-known that for any $N\geq 1$, there is a dense and open subset $\mathcal{U}_N\subset\diff^1(M)$, such that, for any $g\in\mathcal{U}_N$, the set $Per_N(g)$ is a finite set and any point $q\in Per_N(g)$ is a hyperbolic periodic point.

We define a map $\Phi_{N}:\mathcal{U}_N\mapsto \mathcal{S}$, sending a diffeomorphism $g$ to the set of pairs $(q,P_{\lambda_3}(q,g))$, where $q\in Per_N(g)$, and $P_{\lambda_3}(q,g)$ is a compact set contained in $\mathcal{C}(q,g)$ defined as following:
\begin{itemize}
\item If $\mathcal{C}(q,g)$ has a $\lambda_0^2$-dominated splitting $E\oplus F$ such that $\dim(E)=\ind(q)$, then the set $P_{\lambda_3}(q,g)$ is the set of $\lambda_3$-$E$-Pliss points contained in $\mathcal{C}(q,g)$.
\item Otherwise, $P_{\lambda_3}(q,g)=\emptyset$.
\end{itemize}

\begin{lem}\label{continuation of pliss points}
For each positive integer $N\in\mathbb{N}$, the set of continuity points of $\Phi_N$, denoted by $\mathcal{B}_N$, is a residual subset of $\diff^1(M)$.
\end{lem}

\begin{proof}
Assume $g\in\diff^1(M)$ and $p_g$ is a hyperbolic periodic point of $g$. There is a $C^1$-neighborhood $\mathcal{U}$ of $g$, such that, for any $h\in\mathcal{U}$, the point $p_g$ has a continuation $p_h$. For any neighborhood $V$ of $\mathcal{C}(q,g)$, there is a $C^1$-neighborhood $\mathcal{U}_1\subset\mathcal{U}$ of $g$, such that $\mathcal{C}(q,h)\subset V$ for any $h\in\mathcal{U}_1$.

If $\mathcal{C}(q,g)$ has a $\lambda_0^2$-dominated splitting, then it is a robust $\lambda_0^2$-dominated splitting. More precisely, there is a $C^1$-neighborhood $\mathcal{U}_2\subset\mathcal{U}$ of $g$, such that $\mathcal{C}(q_h,h)$ has a $\lambda_0^2$-dominated splitting for any $h\in\mathcal{U}$. Hence by the choice of $\mathcal{U}_N$, there is an open and dense subset $\mathcal{U}'_N\subset\mathcal{U}_N$, such that, for any $g\in\mathcal{U}'_N$, any $q\in Per_N(g)$, the chain recurrence class $\mathcal{C}(q,g)$ either has a robust $\lambda_0^2$-dominated splitting or has no $\lambda_0^2$-dominated splitting robustly. Moreover, if there is a sequence of diffeomorphisms $\{g_n\}_{n\geq 0}$ such that $g_n$ converges to $g$, and $g_n$ has a $\lambda_3$-$E$-Pliss point $x_n\in \mathcal{C}(q_h,h)$, then, any limit point $x$ of the sequence $\{x_n\}$ is a $\lambda_3$-$E$-Pliss point of $g$.

By the above arguments, we can see that $\Phi_N$ is an upper-semi-continuous map restricted to $\mathcal{U}'_N$. It is known that the set of continuity points of a semi-continuous map is a residual subset. Then $\mathcal{B}_N$ contained a residual subset of $\mathcal{U}'_N$. Since $\mathcal{U}'_N$ is open and dense in $\mathcal{U}_N$, we know that $\mathcal{B}_N$ is a residual subset of $\mathcal{U}_N$. Hence $\mathcal{B}_N$ is a residual subset of $\diff^1(M)$, since $\mathcal{U}_N$ is open and dense in $\diff^1(M)$.
\end{proof}

Denote by $\mathcal{R}_0=\cap_{N\geq 1}\mathcal{B}_N$, then $\mathcal{R}_0$ is a residual subset of $\diff^1(M)$. In the following we take $f\in \mathcal{R}_0\cap \mathcal{R}$.

\subsection{The perturbation to make $W^u(p)$ accumulate to $K$}

We take the $\lambda_2$-$E$-weak set $K\subset H(p)$ of $f$ obtained by Lemma~\ref{weak set and pliss point}. By proposition~\ref{asymptotic connecting 1}, one can obtain a heteroclinic orbit connecting $p$ to $K$ by a $C^1$ perturbation, since $K\subset H(p)$. Hence the set $K$ is still a $\lambda_2$-$E$-weak set if the perturbation is $C^1$ small. Moreover, using the continuation of Pliss points (Section~\ref{bi pliss point and weak set} and~\ref{choice of f}), we can guarantee that the set $K$ is contained in the chain recurrence class of $p$ after the perturbation.

\begin{lem}\label{first perturbation}
Assume $f\in\mathcal{R}_0\cap \mathcal{R}$, then for any neighborhood $\mathcal{U}$ of $f$ in $\diff^1(M)$, there are a diffeomorphism $g_1\in\mathcal{U}$ and a point $y\in M$, such that,
\begin{itemize}
\item $(1)$ $g_1$ coincides with $f$ on the set $K\cup \orb(p)$, and $y\in W^u(p,g_1)$,
\item $(2)$ $\omega(y,g_1)\subset K$,
\item $(3)$ $K$ is contained in $\mathcal{C}(p,g_1)$.
\end{itemize}
\end{lem}

\begin{proof}

By Lemma~\ref{weak set and pliss point}, we obtain that, for the diffeomorphism $f$, there is a $\lambda_3$-bi-Pliss point $x\in H(p)\setminus K$ satisfying $\alpha(x)= K$. Since $K\subset H(p)$, we have that $K\subset\overline{W^u(p)}$. By Proposition~\ref{asymptotic connecting 1}, for any neighborhood $\mathcal{U}$ of $f$ in $\diff^1(M)$, there are a point $y\in W^u(p,f)$ and a diffeomorphism $g_1\in\mathcal{U}$, such that $\omega(y,g_1)\subset K$, and $y\in W^u(p,g_1)$. Moreover, the diffeomorphism $g_1$ coincides with $f$ on the set $\orb^-(x)\cup K\cup \orb(p)$ and $Dg_1$ coincides with $Df$ on $\orb^-(x)$. Thus items (1) and (2) are satisfied, and $x$ is a $\lambda_3$-$F$-Pliss point for $g_1^{-1}$.

Since $p$ is a hyperbolic fixed point and $x\in P_{\lambda_3}(p,f)$, by Lemma~\ref{continuation of pliss points} and the fact that $f$ is a continuity point of $\Phi_1$, if we choose $g_1$ close enough to $f$ (by taking the neighborhood $\mathcal{U}$ small), then there is a $\lambda_3$-$E$-Pliss $x'$ close to $x$, such that $x'\in \mathcal{C}(p,g_1)$. Moreover, if $x'$ is close enough to $x$ (by taking $g_1$ close to $f$), then $W^u(x,g_1)\cap W^s(x',g_1)\neq \emptyset$.

Now we prove that $K\subset \mathcal{C}(p,g_1)$. Take any constant $\varepsilon>0$. Since $y\in W^u(p,g_1)$ and $\omega(y,g_1)\subset K$, there is an $\varepsilon$-pseudo orbit connecting from $p$ to a point contained in $K$. In fact the pseudo orbit can be taken as an orbit segment of $\orb(y,g_1)$. On the other hand, since $\alpha(x,g_1)= K$ and $W^u(x,g_1)\cap W^s(x',g_1)\neq \emptyset$, there is an $\frac{\varepsilon}{2}$-pseudo orbit connecting from $K$ to $x'$. By the fact that $x'\in \mathcal{C}(p,g_1)$, there is an $\frac{\varepsilon}{2}$-pseudo orbit connecting from $x'$ to $p$. The composition of the two $\frac{\varepsilon}{2}$-pseudo orbits is an $\varepsilon$-pseudo orbit connecting from $K$ to $p$. Finally by the fact that $\varepsilon$ can be arbitrarily small and the fact that $K=\alpha(x,g_1)$ is a chain transitive set, we have that $K\subset \mathcal{C}(p,g_1)$. This finishes the proof of Lemma~\ref{first perturbation}.
\end{proof}

\subsection{The perturbations to connect $p$ and $K$ by true orbits}

In this subsection, we prove that we can get heteroclinic connections between the hyperbolic fixed point $p$ and the weak set $K$ for a diffeomorphism $C^1$ close to $f$. In the former subsection, we have obtained a diffeomorphism $g_1$ that is $C^1$ close to $f$, and an orbit $\orb(y)$ that connects $p$ to $K$. Moreover, $K$ is still contained in the chain recurrence class of $p$ for $g_1$. We take two steps to get heteroclinic connections between $p$ and $K$. First, since $K\subset \mathcal{C}(p,g_1)$, by Proposition 3, we can connect $K$ by a true orbit to any neighborhood of $p$ by a $C^1$ small perturbation. Then, by the hyperbolicity of $p$, we use the uniform connecting lemma to ``push'' this orbit onto the stable manifold of $p$. We will see that in these two steps, the orbit $\orb(y)$ that connects $p$ to $K$ is not changed.

\begin{lem}\label{second perturbation}
Assume $f\in\mathcal{R}_0\cap \mathcal{R}$, then for any neighborhood $\mathcal{U}$ of $f$ in $\diff^1(M)$, there are a diffeomorphism $g_2\in\mathcal{U}$ and two points $y,y'\in M$, such that,
\begin{itemize}
\item $(1)$ $y\in W^u(p,g_2)$ and $\omega(y,g_2)\subset K$,
\item $(2)$ $y'\in W^s(p,g_2)$ and $\alpha(y',g_2)\subset \omega(y,g_2)$,
\item $(3)$ $g_2$ coincides with $f$ on the set $\omega(y,g_2)\cup \orb(p)$.
\end{itemize}
\end{lem}

\begin{proof}
We take several steps to prove the lemma.

\paragraph{Choice of neighborhoods.}
For any neighborhood $\mathcal{U}$ of $f$ in $\diff^1(M)$, there are a neighborhood $\mathcal{U}_1\subset \mathcal{U}$ and three numbers $\rho>1$, $\delta_0>0$ and $N\in\mathbb{N}$ that satisfy the uniform connecting lemma (Theorem~\ref{uniform connecting}). We can also assume that the fixed point $p$ has a continuation for any $g\in\mathcal{U}_1$. There are a smaller neighborhood $\mathcal{U}'\subset \mathcal{U}_1$ of $f$ and an integer $T$ satisfying the conclusions of Proposition 3. By the hyperbolicity of periodic orbits of $f$, for the integer $T$, there is a neighborhood $\mathcal{U}_2\subset\diff^1(M)$ of $f$, such that, for any diffeomorphism $h\in\mathcal{U}_2$, any periodic point of $h$ with period less than or equal to $T$ is hyperbolic. Take a neighborhood $\mathcal{U}_3$ of $f$ in $\diff^1(M)$, such that $\overline{\mathcal{U}_3}\subset \mathcal{U}_2\cap\mathcal{U}'$.

\paragraph{The connection from $K$ to a neighborhood of $p$ by pseudo-orbits.}
By Lemma~\ref{first perturbation}, there are a diffeomorphism $g_1\in\mathcal{U}_3$ and a point $y\in M$, such that:
\begin{itemize}
\item $g_1$ coincides with $f$ on the set $K\cup \orb(p)\cup \orb^-(y)$,
\item $y\in W^u(p,g_1)$ and $\omega(y,g_1)\subset K\subset \mathcal{C}(p,g_1)$.
\end{itemize}
Denote $K_0=\omega(y,g_1)$.

\begin{claim}\label{connecting by pseudo orbit}
For any neighborhood $V$ of $p$, there are a negatively $g_1$-invariant compact set $X$ and a point $z\in V\cap X$, satisfying that
\begin{itemize}
\item the point $p\notin X$,
\item for any $\vep>0$, there is a $g_1$-$\vep$-pseudo-orbit $Y_{\vep}=(y_0,\cdots,y_m)$ contained in $X$ such that $y_0\in K_0$ and $y_m=z$.
\end{itemize}
\end{claim}

\begin{proof}
For any neighborhood $V$ of $p$, take a smaller neighborhood $V_0$ of $p$, such that $\overline{V_0}\subset V$. For any $k\geq 1$, there is a $g_1$-$\frac{1}{k}$-pseudo-orbit $X_{k}=\{x_0^k,x_1^k,\cdots,x_{m_k}^k\}$, such that, $X_{k}\cap K_0=\{x_0^k\}$, and $X_{k}\cap V_0=\{x_{m_k}^k\}$. Take a subsequence of $\{X_{k}\}_{k\geq 1}$ if necessary, we assume $X_{k}$ converges to a compact set $X$ and $x_{m_k}^k$ converges to a point $z\in\overline{V_0}\subset V$ as $k$ goes to $+\infty$. Obviously, $X$ is a negatively $g_1$-invariant set, $p \notin X$ and $X\cap K_0\neq \emptyset$.

Now we prove that for any $\vep>0$, there is a $g_1$-$\vep$-pseudo-orbit contained in $X$ from $K_0$ to $z$. By the continuity of $g_1$, for any $\vep>0$, there is $k>\frac{3}{\vep}$, such that for all $x,y\in M$, if $d(x,y)<\frac{1}{k}$, then $d(g_1(x),g_1(y))<\frac{\vep}{3}$. Then we take a $\frac{1}{k'}$-pseudo-orbit $X_{k'}=\{x_0^{k'},x_1^{k'},\cdots,x_{m_k}^{k'}\}$, such that $x_0^{k'}\in K_0$ and $x_{m_{k'}}^{k'}\in V_0$ for a number $k'>k$. By choosing $k'$ large enough, we can assume that $d_H(X_{k'},X)<\frac{1}{k}$ and there is a point $y_0\in X\cap K_0$ such that $d(y_0,x_0^{k'})<\frac{1}{k}$ and $d(z,x_{m_{k'}}^{k'})<\frac{1}{k}$. By the assumption, for any $1\leq i\leq m_{k'}-1$, there is $y_i\in X$, such that $d(x_i^{k'},y_i)<\eta$. We prove that $Y_{\vep}:=(y_0,\cdots,y_{m_{k'}}=z)$ is a $\vep$-pseudo-orbit of $g_1$. In fact, for any $0\leq i\leq m_{k'}-1$,
   \begin{center}
     $d(g_1(y_i),y_{i+1})\leq d(g_1(y_i),g_1(x_i^{k'}))+d(g_1(x_i^{k'}),x_{i+1}^{k'})+d(x_{i+1}^{k'},y_{i+1})<\frac{\vep}{3}+\frac{1}{k'}+\frac{1}{k}<\vep$.
   \end{center}
Hence $Y_{\vep}\subset X$ is a $\vep$-pseudo-orbit of $g_1$ from the set $K_0$ to the point $z$.
\end{proof}

\paragraph{The perturbation to connect $K$ to a neighborhood of $p$.}
We take a local stable manifold $W^s_{loc}(p,g_1)$ of $p$, and take a compact fundamental domain $I_{g_1}$ of $W^s_{loc}(p,g_1)$. Then there is a number $\delta<\delta_0$, where $\delta_0$ is chosen in the paragraph {\bf Choice of neighborhoods.} such that, for any point $w\in I_{g_1}$, the $N$ balls $(g_1^j(B(w,2\delta)))_{0\leq j\leq N-1}$ are each of size smaller than $\delta_0$, pairwise disjoint and disjoint from the set $K\cup \orb(y,g_1)\cup \orb(p)$. By the compactness of $I_{g_1}$, there are finite points $w_1,w_2,\cdots,w_L\in I_{g_1}$ such that $(B(w_i,\delta/\rho))_{1\leq i\leq L}$ is a finite open cover of $I_{g_1}$. Then there is a number $\eta>0$ such that, for any diffeomorphism $h\in\mathcal{U}_1$ that is $\eta$-$C^0$ close to $g_1$, we have that:
\begin{itemize}
\item $(a)$ $W^s_{loc}(p_h,h)$ is $C^0$ close to $W^s_{loc}(p,g_1)$,
\item $(b)$ $(B(w_i,\delta/\rho))_{1\leq i\leq L}$ is still a finite open cover of a fundamental domain $I_{h}$ of $W^s_{loc}(p_h,h)$
\item $(c)$ for any $1\leq i\leq L$, the $N$ balls $(h^j(B(w_i,2\delta))_{0\leq j\leq N-1}$ are each of size smaller than $\delta_0$, pairwise disjoint and disjoint with the set $K\cup \orb(y,g_1)\cup \orb(p,g_1)$.
\end{itemize}
\medskip

Recall that $p\notin X$ and $X$ is negatively invariant by Claim~\ref{connecting by pseudo orbit}. Since $y\in W^u(p,g_1)$, we have that $\orb(y,g_1)\cap X=\emptyset$. By the choice of $g_1$, we have that all periodic orbits of $g_1$ contained in $X$ with period less than or equal to $T$ are hyperbolic. Under all these hypotheses, if $(X\setminus K_0) \cap \overline{\orb(y,g_1)}=\emptyset$, then there is a neighborhood $U_0$ of $X\setminus K_0$ such that $U_0\cap \overline{\orb(y,g_1)}=\emptyset$. By Proposition~\ref{asymptotic connecting}, there is a diffeomorphism $h\in \mathcal{U}_1$ which is $\eta$-$C^0$ close to $g_1$, such that $h=g_1=f|_{\{p\}\cup \orb(y)\cup K_0}$, and $\alpha(z,h)\subset K_0$. Thus the above items $(a)$, $(b)$ and $(c)$ are satisfied for such a diffeomorphism $h$.

\paragraph{The perturbation to get a heteroclinic connection between $p$ and $K$.}
By the hyperbolicity of the periodic point $p$, if we take the neighborhood $V$ of $p$ small enough, then the diffeomorphism $h$ and the point $z$ chosen above  satisfy that the negative orbit of $z$ under $h$ intersects with $B(w_i,\delta/\rho)$ for some $i\in \{1,2,\cdots,L\}$. Since $\alpha(z,h)\subset K_0$ and $B(w_i,\delta/\rho)\cap K_0=\emptyset$, there is a point $w=h^{-t}(z)$ for some integer $t>0$, such that $\orb^-(w)\cap B(w_i,\delta/\rho)=\emptyset$ and $w$ has a positive iterate under $h$ contained in $B(w_i,\delta/\rho)$. By the item $(b)$, there is a point $y'\in W^s(p,h)$, such that $\orb^+(y',h)\cap (\cup_{0\leq j\leq N-1}h^i(B(w_i,\delta/\rho)))=\emptyset$ and $y'$ has a negative iterate under $h$ contained in $B(w_i,\delta/\rho)$. By Theorem~\ref{uniform connecting}, there is a diffeomorphism $g_2\in\mathcal{U}$, such that $y'$ is on the positive iterate of $w$ under $g_2$. Moreover, $g_2=g_1$ on the set $K_0\cup \orb(y)\cup \orb(p)\cup \orb^-(w)\cup \orb^+(y')$, hence $g_2=f$ on the set $\orb(p)\cup K_0$, where $K_0=\omega(y,g_1)=\omega(y,g_2)$. Thus the three items of the lemma are satisfied for $g_2$. This finishes the proof of Lemma~\ref{second perturbation}.
\end{proof}

\subsection{Last perturbation to get a weak periodic orbit}

The following lemma estimates the average contraction along the bundle $E$ on periodic orbits. Recall that $\lambda_0$ satisfies that $E\oplus F$ is $(1,\lambda_0^2)$-dominated splitting, and, for the hyperbolic fixed point $p$, we have $\|Df|_{E(p)}\|<\lambda_0$.

\begin{lem}\label{choose time}
Assume $f\in\mathcal{R}_0\cap\mathcal{R}$. Then for any neighborhood $\mathcal{U}$ of $f$ in $\diff^1(M)$, for any integer $L>0$, any neighborhood $U_p$ of $p$, there is $g\in\mathcal{U}$, which coincides with $f$ on $\{p\}$, satisfying that, $g$ has a periodic point $q\in U_p$ with period $\tau>L$ such that $\orb(q)$ has $\lambda_0^2$-dominated splitting $E\oplus F$, and
   \begin{displaymath}
     {\lambda_1}^{\tau}\leq \prod_{0\leq i\leq \tau-1} \|Dg|_{E(g^{i}(q))}\|\leq {\lambda_2}^{\tau}.
   \end{displaymath}
\end{lem}

\begin{proof}
We take several steps to prove the lemma. We take the $\lambda_2$-$E$-weak set $K\subset H(p)$ of $f$ obtained by Lemma~\ref{weak set and pliss point}. Take two numbers $\lambda_1'$ and $\lambda_2'$, such that $\lambda_1<\lambda_1'<\lambda_2'<\lambda_2$.

\paragraph{Choice of neighborhoods and constants.}
There is a neighborhood $V$ of $H(p)$ and a neighborhood $\mathcal{V}\subset \diff^1(M)$ of $f$, such that, for any $h\in\mathcal{V}$, the following properties are satisfied.
\begin{itemize}
\item The maximal invariant compact set of $h$ in $V$ has a dominated splitting which is a continuation of $E\oplus F$. To simplify the notations, we still denote this domination by $E\oplus F$.
\item The fixed point $p$ has a continuation $p_{h}\in V$ for $h$, and $\|Dh|_{E(p_{h})}\|<\lambda_0$.
\item The chain recurrence class $\mathcal{C}(p_{h},h)$ of $p_{h}$ is contained in $V$.
\end{itemize}

Moreover, since $K$ is a $\lambda_2$-$E$-weak set for $f$, there are a neighborhood $U_{K}\subset V$ of $K$ and a number $N_{K}$, satisfying the following property: consider a point $z$ whose whole orbit is contained in $V$, hence $\orb(z)$ has the dominated splitting $E\oplus F$ by the choice of $V$, if the piece of orbit $(z,f(z),\cdots,f^n(z))$ is contained in $\overline{U_{K}}$ with $n\geq N_{K}$, then we have:
   \begin{displaymath}
     \prod_{0\leq i\leq n-1}\|Df|_{E(f^i(z))}\|>{\lambda_2}^{n}.
   \end{displaymath}
To simplify the proof, we just assume that $N_{K}=1$, but the general case is identical.

We can take the neighborhoods $\mathcal{V}$ and $U_p$ small, such that for any diffeomorphism $h\in\mathcal{V}$, the following additional properties are satisfied.
\begin{itemize}
\item For any point $z\in U_p$ whose orbit under $h$ is contained in $V$, hence $\orb(z,h)$ has the dominated splitting $E\oplus F$ by the choice of $V$ and $\mathcal{V}$, we have that $\frac{\lambda_1}{\lambda_1'}<\frac{\|Dh|_{E(z)}\|}{\|Df|_{E(p)}\|}<\frac{\lambda_2}{\lambda_2'}$.
\item For any point $z\in U_{K}$ whose orbit under $h$ is contained in $V$, we have that $\|Dh|_{E(z)}\|>\lambda_2$.
\end{itemize}

We can also assume that $\overline{U_K}\cap \overline{U_p}=\emptyset$ and $\overline{U_K}\cup \overline{U_p}\subset V$. And moreover, we can assume that $\overline{\mathcal{U}}\subset\mathcal{V}$.

By Lemma~\ref{second perturbation}, there are a diffeomorphism $g_2\in\mathcal{U}$ and two points $y,y'\in M$, satisfying that:
\begin{itemize}
\item $y\in W^u(p,g_2)$ and $\omega(y,g_2)\subset K$,
\item $y'\in W^s(p,g_2)$ and $\alpha(y',g_2)\subset \omega(y,g_2)$,
\item $g_2$ coincides with $f$ on the set $\omega(y,g_2)\cup \orb(p)$.
\end{itemize}
We denote $K_0=\omega(y,g_2)$. Since all periodic points of $f$ are hyperbolic and $g_2=f|_{K_0}$, then $K_0$ contains no non-hyperbolic periodic point of $g_2$.

\paragraph{Choice of time.}

Now we fix the neighborhoods $U_p$ and $U_{K_0}$. Then there are two integers $l$ and $n_0$ satisfying the conclusion of Proposition~\ref{time control} for $g_2$ and the neighborhood $\mathcal{U}$. Then we take $T_{K_0}>L$ large, such that for any $h\in\mathcal{V}$, the inequality
   \begin{displaymath}
     m(h)^{l+n_0}{\lambda_2}^{T_{K_0}}>(\lambda_2')^{T_{K_0}+l+n_0}
   \end{displaymath}
holds. Here $m(h)$ is the mininorm of $h$, i.e. $m(h)=\min_{\|v\|=1}\|Dh(v)\|=\|Dh^{-1}\|^{-1}$.

By the first item of Proposition~\ref{time control}, there is a diffeomorphism $h\in\mathcal{U}$, such that
\begin{itemize}
\item $h$ coincides with $g_2$ on $\orb(p)\cup \orb^-(y)\cup \orb^+(y')$ and outside $U_K$;
\item the point $y'$ is on the positive orbit of $y$ under $h$, with $n_{K_0}=\# (\orb(y,h)\cap U_{K_0})\geq T_{K_0}$ and $n_c=\# ((\orb(y,h)\setminus (U_K\cup U_p))\leq n_0$.
\end{itemize}
Hence by the choice of $T_{K_0}$ and the neighborhoods, we have that
   \begin{displaymath}
     \prod_{h^i(y)\not \in U_p}\|Dh|_{E(h^i(x))}\|>(\lambda_2')^{n_{K_0}+n_c}.
   \end{displaymath}

\begin{claim}\label{estimation}
There is an integer $m>0$, such that:
   \begin{displaymath}
     (\lambda_1')^{n_{K_0}+n_c+m+l}<\|Df|_{E(p)}\|^{l+m}\cdot\prod_{h^i(y)\not \in U_p}\|Dh|_{E(h^i(x))}\|<(\lambda_2')^{n_{K_0}+n_c+m+l}.
   \end{displaymath}
\end{claim}

\begin{proof}
Let $\bar{\lambda}>0$ be the constant such that
   \begin{displaymath}
     \prod_{h^i(y)\not \in U_p}\|Dh|_{E(h^i(x))}\|=\bar{\lambda}^{n_{K_0}+n_c},
   \end{displaymath}
then we have $\bar{\lambda}>\lambda_2'$.

The inequality in the claim is equivalent to
   \begin{displaymath}
      \frac{(n_{K_0}+n_c)\log\frac{\bar{\lambda}}{\lambda_2'}}{\log\frac{\lambda_2'}{\|Df|_{E(p)}\|}}<l+m
      <\frac{(n_{K_0}+n_c)\log\frac{\bar{\lambda}}{\lambda_1'}}{\log\frac{\lambda_1'}{\|Df|_{E(p)}\|}}.
   \end{displaymath}

By the choice of $T_{K_0}$ and $n_{K_0}\geq T_{K_0}$, we have that
   \begin{displaymath}
     \frac{(n_{K_0}+n_c)\log\frac{\bar{\lambda}}{\lambda_2'}}{\log\frac{\lambda_2'}{\|Df|_{E(p)}\|}}>l.
   \end{displaymath}
So we only need that
   \begin{displaymath}
     \frac{(n_{K_0}+n_c)\log\frac{\bar{\lambda}}{\lambda_1'}}{\log\frac{\lambda_1'}{\|Df|_{E(p)}\|}}-
                  \frac{(n_{K_0}+n_c)\log\frac{\bar{\lambda}}{\lambda_2'}}{\log\frac{\lambda_2'}{\|Df|_{E(p)}\|}}>1.
   \end{displaymath}
It is equivalent to
   \begin{displaymath}
     (n_{K_0}+n_c)\left((\frac{1}{\log\frac{\lambda_1'}{\|Df|_{E(p)}\|}}-\frac{1}{\log\frac{\lambda_2'}{\|Df|_{E(p)}\|}})\log\bar{\lambda}
                    +\frac{\log\lambda_2'}{\log\frac{\lambda_2'}{\|Df|_{E(p)}\|}}-\frac{\log\lambda_1'}{\log\frac{\lambda_1'}{\|Df|_{E(p)}\|}}\right)>1.
   \end{displaymath}
Since $\bar{\lambda}>\lambda_2'$, and $n_{K_0}>T_{K_0}$, it is sufficient to require that
   \begin{displaymath}
     \frac{T_{K_0}(\log\lambda_2'-\log\lambda_1')}{\log\frac{\lambda_1'}{\|Df|_{E(p)}\|}}>1.
   \end{displaymath}
By taking $T_{K_0}$ large enough, the above inequality is satisfied.
\end{proof}

\paragraph{Choice of the diffeomorphism $g$.}
We take $g=h_m\in\mathcal{U}$ from item 2 of Proposition~\ref{time control}, then $g$ has a periodic orbit $O=\orb(q)$, such that, $O\setminus U_p=\orb(y,h)\setminus U_p$, and $\# (O\cap U_p)=l+m$. Denote by $\pi(O)$ the period of $O$, then we have $\pi(O)=n_{K_0}+n_c+m+l$. By the choice of the neighborhood $\mathcal{U}$ and the constants, we have
   \begin{displaymath}
      \prod_{0\leq i\leq \pi(O)-1} \|Dg|_{E(g^{i}(q))}\|=\prod_{g^i(q) \in U_p}\|Dg|_{E(g^i(q))}\|\prod_{h^i(y)\not \in U_p}\|Dh|_{E(h^i(x))}\|.
   \end{displaymath}

By the choice of the neighborhoods $\mathcal{V}$ and $U_p$, and the constants $\lambda_1'$ and $\lambda_2'$, we have that
   \begin{displaymath}
     \left(\frac{\lambda_1}{\lambda_1'}\right)^{l+m}\|Df|_{E(p)}\|^{l+m}<\prod_{g^i(q) \in U_p}\|Dg|_{E(g^i(q))}\|<\left(\frac{\lambda_2}{\lambda_2'}\right)^{l+m}\|Df|_{E(p)}\|^{l+m}.
   \end{displaymath}
Then by the estimation in Claim~\ref{estimation}, we can see that
   \begin{displaymath}
     {\lambda_1}^{\pi(O)}\leq  \prod_{0\leq i\leq \pi(O)-1} \|Dg|_{E(g^{i}(q))}\| \leq {\lambda_2}^{\pi(O)}
   \end{displaymath}
This finishes the proof of Lemma~\ref{choose time}.
\end{proof}

\subsection{The genericity argument}\label{generic argument}

Lemma~\ref{choose time} is a perturbation result to get weak periodic orbits. To get the conclusion of Theorem B, we have to do the genericity argument based on Lemma~\ref{choose time}, see for instance~\cite{gw}.

Take a countable basis $(V_n)_{n\geq 1}$ of $M$, and take the countable family $(U_n)_{n\geq 1}$, where each $U_n$ is a union of finitely many sets of $(V_n)_{n\geq 1}$. Take the countable pairs $(\eta_n,\gamma_n)_{n\geq 1}$ of rational numbers contained in $(\lambda_0,1)$ with $\eta_n<\gamma_n$ for each $n\geq 1$.

Let $\mathcal{H}_{n,m}$ be the set of $C^1$ diffeomorphisms $h$ such that, every $h_1$ in a $C^1$ neighborhood $\mathcal{V}\subset\diff^1(M)$ of $h$ has a hyperbolic periodic point $q\in U_n$ satisfying that the hyperbolic splitting $E^s\oplus E^u$ of $\orb(q,h_1)$ is a $\lambda_0^2$-dominated splitting and
   \begin{displaymath}
     {\eta_m}^{\tau(q)}< \prod_{0\leq i\leq \tau(q)-1} \|Dh_1|_{E^s(h_1^{i})}\|< {\gamma_m}^{\tau(q)},
   \end{displaymath}\\
where $\tau(q)$ is the period of $q$. Let $\mathcal{N}_{n,m}$ be the set of $C^1$ diffeomorphisms $h$ such that every $h_1$ in a $C^1$ neighborhood $\mathcal{V}\subset\diff^1(M)$ of $h$ has no hyperbolic periodic point $q\in U_n$ satisfying that the hyperbolic splitting $E^s\oplus E^u$ of $\orb(q,h_1)$ is a $\lambda_0^2$-dominated splitting and
   \begin{displaymath}
     {\eta_m}^{\tau(q)}< \prod_{0\leq i\leq \tau(q)-1} \|Dh_1|_{E^s(h_1^{i})}\|< {\gamma_m}^{\tau(q)},
   \end{displaymath}\\
where $\tau(q)$ is the period of $q$.

Notice that $\mathcal{N}_{n,m}=\diff^1(M)\setminus\overline{\mathcal{H}_{n,m}}$. Hence $\mathcal{H}_{n,m}\cup \mathcal{N}_{n,m}$ is $C^1$ open and dense in $\diff^1(M)$. Let
   \begin{center}
     $\mathcal{R}_1=\bigcap_{n\geq1,m\geq1}(\mathcal{H}_{n,m}\cup \mathcal{N}_{n,m})$.
   \end{center}
Then $\mathcal{R}_1$ is a residual subset of $\diff^1(M)$, and $\mathcal{R}_0\cap\mathcal{R}_1\cap\mathcal{R}$ is also a residual subset of $\diff^1(M)$.

\begin{claim}\label{average estimation}
Assume $f\in\mathcal{R}_0\cap\mathcal{R}_1\cap\mathcal{R}$. Then for any two numbers $\lambda_1<\lambda_2\in(\lambda_0,1)$, for any neighborhood $U_p$ of $\orb(p)$ and any integer $L>0$, there is a periodic point $q\in U_p$ with period $\tau>L$ such that $\orb(q)$ has the $\lambda_0^2$-dominated splitting $E\oplus F$, and
   \begin{displaymath}
     {\lambda_1}^{\tau}\leq \prod_{0\leq i\leq \tau-1} \|Df|_{E(f^{i}(q))}\|\leq {\lambda_2}^{\tau}.
   \end{displaymath}
\end{claim}
\begin{proof}
We take two rational numbers $\eta_i,\gamma_i\in(\lambda_0,1)$, such that $\lambda_1<\eta_i<\gamma_i<\lambda_2$, and take $U_j$ from the countable basis of $M$, such that $U_j\subset U_p$. Then by Lemma~\ref{choose time}, there is a diffeomorphism $g$ arbitrarily $C^1$ close to $f$, such that $g$ has a periodic point $q\in U_p$ with period $\tau>T$ such that the $\lambda_0^2$-dominated splitting $E\oplus F$ is the hyperbolic splitting on $\orb(q,g)$, and
   \begin{displaymath}
     {\eta_i}^{\tau}\leq \prod_{0\leq i\leq \tau-1} \|Dg|_{E(g^{i}(q))}\|\leq {\gamma_i}^{\tau}.
   \end{displaymath}
Then $f\notin\mathcal{N}_{j,i}$, thus $f\in\mathcal{H}_{j,i}$ and $f$ satisfies the conclusion of Claim~\ref{average estimation}.
\end{proof}

\begin{claim}
Theorem B holds for any diffeomorphisms in $\mathcal{R}_0\cap\mathcal{R}_1\cap\mathcal{R}$.
\end{claim}
\begin{proof}
Assume $f\in\mathcal{R}_0\cap\mathcal{R}_1\cap\mathcal{R}$ and $f$ satisfies the assumptions of Theorem B. By Claim~\ref{average estimation}, we get a sequence of periodic orbits $\orb(q_k)$ of $f$, such that $q_k\rightarrow p$ with $\tau(q_k)\rightarrow\infty$, and
   \begin{displaymath}
     {\lambda_1}^{\tau(q_k)}\leq \prod_{0\leq i\leq \tau(q_k)-1} \|Df|_{E(f^{i}(q_k))}\|\leq {\lambda_2}^{\tau(q_k)}.
   \end{displaymath}
Hence by the $\lambda_0^2$-domination of $E\oplus F$, we have that
   \begin{displaymath}
     \prod_{0\leq i\leq \tau(q_k)-1} \|Df^{-1}|_{F(f^{-i}(q_k))}\|\leq {\lambda_2}^{\tau(q_k)}.
   \end{displaymath}
Then by item 2 of Lemma~\ref{cor of pliss} and item 2 of Lemma~\ref{property of pliss point}, there is a $\lambda_2$-bi-Pliss point $r_k$ on $\orb(q_k)$ for each $k$. Taking a subsequence if necessary, we assume $(r_k)$ is a converging sequence. Then there is $l>0$, such that for any $m,n\geq l$, the stable and unstable manifolds of $r_m$ and $r_n$ intersect respectively, since $r_k$ has uniform stable and unstable manifolds. Hence $(\orb(q_m))_{m\geq l}$ are homoclinically related together, thus $p\in H(q_k)$. By item 2 of Lemma~\ref{generic properties}, we have that $q_k\in H(p)$. This finishes the proof of the claim.
\end{proof}

The proof of Theorem B is now completed.

\section{Periodic orbits around a periodic orbit and a set: proof of Proposition 1}\label{proposition 1}

In this section, we give the proof of Proposition~\ref{time control}. We consider the periodic point $p$ and the invariant compact set $K$ satisfying that $p\notin K$ and $K$ contains no non-hyperbolic periodic point. Recall that the point $x\in W^u(p)$ satisfies $\omega(x)\cap K\neq \emptyset$ and the point $y\in W^s(p)$ satisfies $\alpha(y)=K$.

Taking a smaller neighborhood if necessary, we can assume that the element of $\mathcal{U}$ is of the form $f\circ\phi$ with $\phi\in\mathcal{V}$, where $\mathcal{V}$ is a $C^1$ neighborhood of $Id$ and satisfies the property (F) in Definition~\ref{property F}. By Theorem~\ref{Thm:connecting lemma}, there is an integer $N$ associated to the neighborhood $\mathcal U$. By Lemma~\ref{basic perturbation}, there are two numbers $\theta>1$ and $r_0>0$ associated to $\mathcal U$.

It is easy to see that we only need to prove the proposition for $U_p$ and $U_K$ small. More precisely, we assume that $U_p\cap U_K=\emptyset$ and $(\orb^-(x)\cup \orb^+(y))\cap U_K=\emptyset$. Moreover, by the hyperbolicity of periodic orbits in $K$, we assume that there is no periodic point with period less than or equal to $N$ contained in $U_K\setminus K$.

To simplify the notation, we assume that $p$ is a hyperbolic fixed point of $f$, and the proof of the general case is similar. The only difference is that, in item 2 of the conclusion, the condition $\# (O\cap U_p)\in \{l+m\tau,l+m\tau+1,\cdots,l+(m+1)\tau-1\}$ should be $\# (O\cap U_p)=l+m$. In the general case, the number $\# (O\cap U_p)$ cannot be chosen arbitrarily, but we can make sure it is contained in an interval whose length is the period of $p$.

Now we fix the neighborhoods $\mathcal{U}$, $U_p$ and $U_K$, and the numbers $N$, $\theta$ and $r_0$.

\subsection{The choice of $n_0$, the point $z_1$ and the perturbation domain at $z_1$.}

Recall that the point $x\in W^u(p)$ satisfies $\omega(x)\cap K\neq \emptyset$.

\begin{lem}\label{property of prec}
There is a point $z_1\in U_K \setminus K$, such that:
\begin{itemize}
\item for any neighborhood $V_{z_1}$ of $z_1$, there is $n\geq 1$ such that $f^n(x)\in V_{z_1}$;
\item $z_1\prec_{U_K} K$ and $\orb^+(z_1)\subset U_K$.
\end{itemize}
\end{lem}

\begin{proof}
The proof is similar to that of Lemma~\ref{prec}. We take a smaller open neighborhood $V$ of $K$ such that $\overline{V}\subset U_K$. Since $\omega(x)\cap K\neq\emptyset$, for any $k\geq 1$, there is $n_k\geq 1$, such that $f^{n_k}(x)\in B(K,\frac{1}{k})$. Take the smallest integer $m_k$, such that the piece of orbit $(f^{m_k}(x),f^{m_k+1}(x),\cdots,f^{n_k}(x))$ is contained in $V$. Taking a convergent subsequence if necessary, we assume that the sequence $\{f^{m_k}(x)\}_{k\geq 1}$ converges to a point $z_1\in\overline{V}\subset U_K$ and the sequence $\{f^{n_k}(x)\}_{k\geq 1}$ converges to a point $z_2\in K$. Then we have that $z_1\prec_{U_K} z_2$, and the pieces of orbit that connects the neighborhoods of $z_1$ and $z_2$ are $(f^{m_k}(x),f^{m_k+1}(x),\cdots,f^{n_k}(x))_{k\geq 1}$. Since $z_2\in K$, we have that $z_1\prec_{U_K} K$. By the choice of $m_k$, we have that $f^{m_k-1}(x)\in M\setminus V$. Since $M\setminus V$ is compact, and $f^{-1}(z_1)$ is a limit point of the sequence $\{f^{m_k-1}(x)\}_{k\geq 1}$, we have that $f^{-1}(z_1)\in M\setminus V$. By the invariance of $K$, we have that $z_1\notin K$ and $n_k-m_k$ goes to $+\infty$. Since $(f^{m_k}(x),f^{m_k+1}(x),\cdots,f^{n_k}(x))$ is contained in $U_K$ and $f^{m_k}(x)$ converges to $z_1$, we have that $\orb^+(z_1)\subset\overline{V}\subset U_K$. Thus the second item is satisfied. The first item is a trivial fact by the choice of $z_1$.
\end{proof}

By the assumption on $U_K$, we have that $z_1$ is not a periodic point with period less than or equal to $N$. Also, since $y\in W^s(p)$ and $\orb^+(z_1)\subset U_K$, we have $z_1\notin \orb(y)$. Moreover, by the fact that $\alpha(y)=K$, we know $z_1\notin \overline{\orb(y)}$. Then there are two neighborhoods $V_{z_1}\subset U_{z_1}$ of $z_1$ satisfying the conclusion of Theorem~\ref{Thm:connecting lemma} for the triple $(f,\mathcal{U},N)$, and also satisfying the following conditions:
\begin{itemize}
\item $U_{z_1}\cup f(U_{z_1})\cup \cdots \cup f^N(U_{z_1}) \subset U_K\setminus K$;
\item $(U_{z_1}\cup f(U_{z_1})\cup \cdots \cup f^N(U_{z_1})) \cap \orb(y)=\emptyset$.
\end{itemize}

By Lemma~\ref{property of prec}, there is $n_1\in \mathbb{N}$ such that $f^{n_1}(x)\in V_{z_1}$. Moreover, since $\alpha(y)=K$, there is $n_2\in \mathbb{N}$, such that, for any $n\geq n_2$, we have $f^{-n}(y)\in U_K$. Let $n_0=n_1+n_2$.

\subsection{The choices of points and perturbation domains in $K$ and to get $h$.}\label{get h}

Take any integer $T_K$. By Lemma~\ref{property of prec}, we have that $z_1\prec_{U_K} K$, that is to say, there is a point $z_2\in K$ such that $z_1\prec_{U_K} z_2$. Now we consider two cases, depending on whether there is such a point $z_2$ that is not a periodic point with period less than or equal to $N$.

\subsubsection{The non-periodic case}
Assume that there is a point $z_2\in K$ which is not a periodic point with period less than or equal to $N$, and $z_1\prec_{U_K} z_2$. Then there are two neighborhoods $V_{z_2}\subset U_{z_2}$ of $z_2$ satisfying the conclusion of Theorem~\ref{Thm:connecting lemma} for the triple $(f,\mathcal{U},N)$ and also satisfying the following conditions:
\begin{itemize}
\item $U_{z_2}\cup f(U_{z_2})\cup \cdots \cup f^N(U_{z_2}) \subset U_K$;

\item $(U_{z_1}\cup f(U_{z_1})\cup \cdots \cup f^N(U_{z_1})) \cap (U_{z_2}\cup f(U_{z_2})\cup \cdots \cup f^N(U_{z_2}))=\emptyset$;

\item $f^n(x)\notin U_{z_2}\cup f(U_{z_2})\cup \cdots \cup f^N(U_{z_2})$, for any $-\infty<n\leq n_1$;

\item $f^{-n}(y)\notin U_{z_2}\cup f(U_{z_2})\cup \cdots \cup f^N(U_{z_2})$, for any $n\leq n_2+T_K$.
\end{itemize}
Then there is $n_3>n_2+T_K$, such that $f^{-n_3}(y)\in V_{z_2}$. Since we have the fact that $z_1\prec_{U_K} z_2$, there is a piece of orbit $(w,f(w),\cdots,f^k(w))$ contained in $U_K$, such that $w\in V_{z_1}$ and $f^k(w)\in V_{z_2}$. Moreover, since $U_{z_1}\cap \orb(y)=\emptyset$ and $w\in V_{z_1}\subset U_{z_1}$, we have that $w\notin \orb(y)$.

\paragraph{\textit{\textmd{Perturbations to get $h$ in the non-periodic case.}}} Now we do the perturbations step by step to get the conclusion.

\textbf{Step 1.} From the choice of points and neighborhoods above, we can see that the point $x$ has a positive iterate $f^{n_1}(x)\in V_{z_1}$ and the point $f^k(w)$ has a negative iterate $w\in V_{z_1}$. Then by Theorem~\ref{Thm:connecting lemma}, there is a diffeomorphism $f_1\in\mathcal{U}$, such that $f_1$ coincides with $f$ outside $U_{z_1}\cup f(U_{z_1})\cup \cdots \cup f^N(U_{z_1})$ and $f^k(w)$ is on the positive orbit of $x$ under $f_1$.

\textbf{Step 2.} For the diffeomorphism $f_1$, the point $x$ has a positive iterate $f^k(w)\in V_{z_2}$, and the point $y$ has a negative iterate $f^{-n_3}(y)\in V_{z_2}$. Since $f_1$ coincides with $f$ outside $U_{z_1}\cup f(U_{z_1})\cup \cdots \cup f^N(U_{z_1})$ and $(U_{z_1}\cup f(U_{z_1})\cup \cdots \cup f^N(U_{z_1})) \cap (U_{z_2}\cup f(U_{z_2})\cup \cdots \cup f^N(U_{z_2}))=\emptyset$, then by Theorem~\ref{Thm:connecting lemma}, there is a diffeomorphism $h\in\mathcal{U}$, such that $y$ is on the positive orbit of $x$ under $h$ and $h$ coincides with $f_1$ outside $U_{z_2}\cup f(U_{z_2})\cup \cdots \cup f^N(U_{z_2})$.

By the constructions above, $h$ coincides with $f$ outside $(U_{z_1}\cup f(U_{z_1})\cup \cdots \cup f^N(U_{z_1}))\cup(U_{z_2}\cup f(U_{z_2})\cup \cdots \cup f^N(U_{z_2}))$. Hence $h$ coincides with $f$ on $\orb(p)\cup \orb^-(x)\cup \orb^+(y)$ and outside $U_K$. Moreover, $\# (\orb(x,h)\cap U_K)\geq n_3-n_2\geq T_K$ and $\# (\orb(x,h)\setminus (U_K\cup U_p))\leq n_1+n_2\leq n_0$.

\subsubsection{The periodic case}
Assume that any point $z_2\in K$ satisfying $z_1\prec_{U_K} z_2$ is a periodic point with period less than or equal to $N$. We take such a point $q\in K$. In this case, we cannot use Theorem~\ref{Thm:connecting lemma} at the point $q$ since its period is small but we can do perturbations at the stable and unstable manifolds of $q$ since it is hyperbolic. To simplify the proof, we assume that $q$ is a hyperbolic fixed point of $f$, but the general case is identical. We take a neighborhood $U_q$ of $q$ such that $\overline{U_q}\subset U_K\setminus (U_{z_1}\cup f(U_{z_1})\cup \cdots \cup f^N(U_{z_1}))$, and such that for any point $w$ satisfying $\orb^+(w)\subset U_q$ (rep. $\orb^-(w)\subset U_q$), we have $w\in W^s(q)$ (resp. $w\in W^u(q)$).

Since $z_1\prec_{U_K} q$, by a similar argument as in Lemma~\ref{property of prec}, we can get that, there is a point $x'\in U_q$, such that $z_1\prec_{U_K} x'$ and $\orb^+(x)\subset U_q$. By the choice of $U_q$, we have that $x'\in W^s(q)$ and $x'\notin U_{z_1}\cup f(U_{z_1})\cup \cdots \cup f^N(U_{z_1})$. Notice that, since $x'$ is not a periodic point and $z_1\prec_{U_K} x'$, we have $x'\notin K=\alpha(y)$. Hence $x'\notin\overline{\orb(y)}$.

Since $q\in \alpha(y)$, there is $y'\in W^u(q)\cap U_q$, such that for any neighborhood $U$ of $y'$, there is an integer $n\geq 1$, such that $f^{-n}(y)\in U$. (In fact, if $\alpha(y)=\{q\}$, we can choose $y'$ to be a negative iterate of $y$. If $\{q\}\subsetneqq \alpha(y)$, we can choose $y'$ to be contained in $\alpha(y)\cap W^u(q)$). Moreover, we have $y'\notin\orb(x)$.

By the $\lambda$-Lemma, there are two neighborhoods $W_{x'}$ and $W_{y'}$ of $x'$ and $y'$ respectively such that, for any two smaller neighborhoods $W'_{x'}\subset W_{x'}$ and $W'_{y'}\subset W_{y'}$ of $x'$ and $y'$ respectively, there is a piece of orbit $(z',f(z'),\cdots,f^{t}(z'))$ contained in $U_q$, such that $z'\in W'_{x'}$, $f^{t}(z')\in f^{-1}(W'_{y'})$, $f^i(z')\notin W'_{x'}\cup W'_{y'}$ for any $i\in \{1,2,\cdots,t\}$ and $t\geq T_K$.

Now we construct the perturbation domains at the points $x'$ and $y'$ respectively. Recall that $\theta>1$ and $r_0>0$ are the two constants obtained by Lemma~\ref{basic perturbation} associated to $\mathcal U$ and $f$.
\medskip

{\it Perturbation domain at $x'$} We can take two neighborhoods $V_{x'}\subset U_{x'}$ of $x'$ that satisfy the conclusions of Theorem~\ref{Thm:connecting lemma} for the triple $(f,\mathcal{U},N)$, and also satisfy that
\begin{itemize}
\item $U_{x'}\subset W_{x'}$;

\item $U_{x'}\cup f(U_{x'})\cup\cdots,\cup f^N(U_{x'})\subset U_q$;

\item $f^n(x)\notin U_{x'}\cup f(U_{x'})\cup \cdots \cup f^N(U_{x'})$, for any $-\infty<n\leq n_1$;

\item $U_{x'}\cap \orb(y)=\emptyset$ and $q\notin U_{x'}$.
\end{itemize}
Then, there is a piece of orbit $(w',\cdots,f^{k'}(w'))$ contained in $U_K$, such that $w'\in V_{z_1}$ and $f^{k'}(w')\in V_{x'}$. Moreover, $w'\notin \orb(y)$.
\medskip



{\it Perturbation domain at $y'$} We take a number $r'<r_0$ small enough, such that: if we take the neighborhood $U_{y'}=f(B(f^{-1}(y'),\theta r'))$ of $y'$, then the following properties are satisfied:
\begin{itemize}
\item $U_{y'}\subset W_{y'}$;

\item $U_{y'}\cup f^{-1}(U_{y'})\subset U_q\setminus (\{q\}\cup U_{x'}\cup f(U_{x'})\cup\cdots,\cup f^N(U_{x'}))$;

\item $U_{y'}\cap f^{-1}(U_{y'})=\emptyset$;

\item $f^n(x)\notin U_{y'}\cup f^{-1}(U_{y'})$, for any $-\infty<n\leq n_1$;

\item $\{w',\cdots,f^{k'}(w')\}\cap (U_{y'}\cup f^{-1}(U_{y'}))=\emptyset$.
\end{itemize}
\medskip

Then by the choice of $U_{x'}$ and $U_{y'}$, there is a piece of orbit $(z',f(z'),\cdots,f^{n_4}(z'))$ contained in $U_q$, such that $z'\in V_{x'}$, $f^{n_4}(z')\in B(f^{-1}(y'),r')$, $f^i(z')\notin U_{x'}$ for any $i\in \{1,2,\cdots,n_4\}$ and $n_4\geq T_K$. By the choice of $y'$, there is a negative iterate $f^{-n_5}(y)$ of $y$ contained in $B(f^{-1}(y'),r')$.




\paragraph{\textit{\textmd{Perturbations to get $h$ in the periodic case.}}}
From the above constructions, we can see that the perturbation domains are pairwise disjoint and contained in $U_K$, and the pieces of orbits that connects two perturbation domains are pairwise disjoint and disjoint from the other perturbation domains. Then we can do the perturbations step by step as in \textit{Case 1}.

\textbf{Step 1.} By Lemma~\ref{basic perturbation}, there is $f_1\in \mathcal{U}$, such that, $f_1$ coincides with $f$ outside $f^{-1}(U_{y'})$ and $f_1(f^{n_4}(z'))=f^{-n_5+1}(y)$. Since $f^n(x)\notin U_{y'}\cup f^{-1}(U_{y'})$, for any $-\infty<n\leq n_1$, we have that $f_1$ coincides with $f$ on $\{f^n(x)\}_{-\infty<n\leq n_1}$.

\textbf{Step 2.} For the diffeomorphism $f_1$, the point $y$ has a negative iterate $z'\in V_{x'}$, and the point $w'$ has a positive iterate $f^{k'}(w')\in V_{x'}$. Then by Theorem~\ref{Thm:connecting lemma}, there is $f_2\in\mathcal{U}$, such that $f_2$ coincides with $f_1$ outside $U_{x'}\cup f(U_{x'})\cup \cdots \cup f^N(U_{x'})$, and $w'$ is on the negative orbit of $y$ under $f_2$. Since $f^n(x)\notin U_{x'}\cup f(U_{x'})\cup \cdots \cup f^N(U_{x'})$, for any $-\infty<n\leq n_1$, we have that $f_2$ coincides with $f$ on $\{f^n(x)\}_{-\infty<n\leq n_1}$.

\textbf{Step 3.} For the diffeomorphism $f_2$, the point $y$ has a negative iterate $w'$ in $V_{z_1}$ and the point $x$ has a positive iterate $f^{n_1}(x)\in V_{z_1}$. By Theorem~\ref{Thm:connecting lemma}, there is $h\in\mathcal{U}$, such that, $h$ coincides with $f_2$ outside $U_{z_1}\cup f(U_{z_1})\cup \cdots \cup f^N(U_{z_1})$ and $y$ is on the positive orbit of $x$ under $h$.

By the constructions above, the diffeomorphism $h$ coincides with $f$ outside $(U_{z_1}\cup f(U_{z_1})\cup \cdots \cup f^N(U_{z_1}))\cup (U_{x'}\cup f(U_{x'})\cup \cdots \cup f^N(U_{x'}))\cup f^{-1}(U_{y'})$. Hence $h$ coincides with $f$ on $\orb(p)\cup \orb^-(x)\cup \orb^+(y)$ and outside $U_K$. Moreover, $\# (\orb(x,h)\cap U_K)\geq n_4\geq T_K$ and $\# (\orb(x,h)\setminus (U_K\cup U_p))\leq n_1+n_2\leq n_0$.

\begin{rem}
We point out that, in the construction, each of the perturbation supports are pairwise disjoint with each other. Since $\mathcal{U}=f\circ\mathcal{V}$, where $\mathcal{V}$ is a $C^1$ neighborhood of $Id$ and satisfies the property (F) in Definition~\ref{property F}, the perturbations $f_1$, $f_2$ and $h$ are still contained in $\mathcal U$.
\end{rem}

\subsection{The choice of $l$ and the perturbation domains at $x$ and $y$, and to get $h_m$.}

From the constructions in section~\ref{get h}, we get the diffeomorphism $h$ that satisfies the first item of Proposition~\ref{time control}. In this section, we do the perturbations to get the diffeomorphism $h_m$.

Assume $h^t(x)=y$. By replacing $x$ and $y$ to a negative or positive iteration, we assume that $x,y\in U_p$ and $\orb^-(x,h)\cup \orb^+(y,h)\subset U_p$. Assume that $\#(\{h^n(x)\}_{1\leq n\leq t}\cap U_p)=m_0$. We take a number $r<r_0$ small enough, such that, if we take the neighborhood $U_x=h(B(h^{-1}(x),\theta r))$ of $x$ and the neighborhood $U_y=B(y,\theta r)$ of $y$, then, the four sets $U_x$, $h^{-1}(U_x)$, $U_y$ and $h(U_y)$ are contained in $U_p$ and pairwise disjoint from each other and disjoint with $\{h^n(x)\}_{1\leq n\leq t}$. By the $\lambda$-Lemma, there is $l_0\in\mathbb{N}$, such that, for any $m\geq 1$, there is a piece of orbit $(h(z),h^2(z),\cdots,h^{l_0+m-1}(z))$ contained in $U_p$, such that $h(z)\in B(y,r)$, $h^{l_0+m-1}(z)\in B(h^{-1}(x),r)$ and $h^i(z)\notin U_y\cup h^{-1}(U_x)$ for any $i=2,3,\cdots,l_0+m-2$. Let $l=l_0+m_0$.

By Lemma~\ref{basic perturbation} and the disjointness of $U_y$, $f^{-1}(U_x)$ and $U_K$, there is $h_m\in \mathcal{U}$, such that, $h_m$ coincides with $h$ outside $U_y\cup f^{-1}(U_x)$, and $h_m(y)=h^2(z)$, $h_m(f^{l_0+m-1}(z))=h_m^{l_0+m-1}(y)=x$. Hence $h_m$ coincides with $h$ on $\orb(p)$ and outside $U_p$. Moreover, the point $x$ is a periodic point of $h_m$, and putting $O=\orb(x,h_m)$, we have that $O\setminus U_p=\orb(x,h)\setminus U_p$, and $\# (O\cap U_p)=l_0+m_0+m=l+m$.

This finishes the proof of Proposition~\ref{time control}.

\begin{rem}
We point out here that, in the periodic case, we cannot do the same perturbations at $x'$ to connect $f^{k'}(w')$ to $f(z')$ just as at the points $y'$ to connect $f^{n_4}(z')$ and $f^{-n_5+1}(y)$. Because the piece of orbit $(w',\cdots,f^{k'}(w'))$ that connects the neighborhoods $V_{z_1}$ and $V_{x'}$ of $z_1$ and $x'$ respectively may enter into the neighborhood $U_{x'}$ many times before $f^{k'}(w')$. Thus if we use the basic perturbation lemma to connect $f^{k'}(w')$ to $f(z')$, the piece of orbit $(w',\cdots,f^{k'}(w'))$ may be modified and it is not clear if the negative orbit of $y$ can intersect $V_{z_1}$ after such perturbation.
\end{rem}

\section{Asymptotic approximation for true orbits: proof of Proposition 2}\label{proposition 2}

In this section, we give the proof of Proposition~\ref{asymptotic connecting 1}. In fact, the proof is almost the same as the proof of Proposition 10 in~\cite{c1}. Recall that $K$ is an invariant compact set and $\alpha(x)\subset K$. We consider the hyperbolic periodic point $p\notin K$ with $\overline{W^u(p)}\cap K\neq \emptyset$. We assume that for any point $y\in W^u(p,f)$, we have $\omega(y)\setminus K\neq \emptyset$, otherwise there is nothing needed to prove. Also we assume that $x\notin K$ because the other case can be obtained directly if the proposition is true under the assumption $x\notin K$. We take two steps to get our purpose:
\begin{itemize}
\item we choose a sequence of non-periodic points $(z_n)_{n\geq 0}$, such that:
\begin{center}
    $z_0\prec z_1\prec \cdots \prec K$, $z_0\in \overline{W^u(p)}$ and $z_n\notin \overline{\orb^-(x)}$, for any $n\geq 0$,
\end{center}

\item then we perturb at every $z_n$ to connect all the points together and avoid $\orb^-(x)$.
\end{itemize}
\bigskip

In order to prove Proposition 2, we take a decreasing sequence of $C^1$-neighborhoods $(\mathcal{U}_n)$ of $f$ that satisfies the following properties:
\begin{itemize}
\item $\overline{\mathcal{U}_0}\subset \mathcal{U}$,

\item the element of $\mathcal{U}_n$ is of the form $f\circ \phi$ with $\phi\in \mathcal{V}_n$, where $(\mathcal{V}_n)$ is a decreasing sequence of $C^1$ neighborhoods of $Id$ that satisfy the property (F) in Definition~\ref{property F}, and $\cap_n \mathcal{V}_n=\{Id\}$.\\
\end{itemize}
Then we have that $\cap_n \mathcal{U}_n=\{f\}$. Theorem~\ref{Thm:connecting lemma} associates to each pair $(f,\mathcal{U}_k)$ a number $N_k$.
\medskip

Recall that $\alpha(x)\subset K$ and we have assumed that $x\notin K$. We need the following three lemmas for the proof of Proposition~\ref{asymptotic connecting 1}.
\begin{lem}\label{connecting points}
For any neighborhood $W$ of $K$, there is a point $z\in (W\cap\overline{W^u(p)})\setminus K$, such that, $z\prec_W K$ and $\orb^+(z)\subset W$. Moreover, $z\notin \overline{\orb^-(x)}$.
\end{lem}

\begin{proof} The proof is similar to the proof of Lemma~\ref{property of prec}. We assume $W$ is a small neighborhood of $K$ such that $x\not\in \overline{W}$ and $p\notin W$. Take an open neighborhood $V\subset W$ of $K$, such that $\overline{V}\subset W$. Since $\overline{W^u(p)}\cap K\neq \emptyset$, for any $k\geq 1$, there is a point $x_k\in W^u(p)$ and a positive integer $n(k)$, such that $f^{n(k)}(x_k)\in B(K,\frac{1}{k})$. For $k$ large, the set $B(K,\frac{1}{k})$ is contained in $V$. We consider the smallest integer $m(k)$ such that the piece of orbit $(f^{m(k)}(x_k),\cdots,f^{n(k)}(x_k))$ is contained in $V$. By the assumption that $\omega(x_k)\setminus K\neq \emptyset$ for any $k\geq 1$, we can see that $n(k)-m(k)$ goes to infinity as $k$ goes to infinity.

By taking convergent subsequence if necessary, assume the sequence $\{f^{m(k)}(x_k)\}$ converges to a point $z\in \overline{V}$, and $\{f^{n(k)}(x_k)\}$ converges to a point $z'\in K$. It can be obtained directly that $z\in \overline{W^u(p)}$. By a similar argument as in the proof of Lemma~\ref{property of prec}, we can obtain the fact that $z\notin K$, and $z\prec_{W} z'$. Then we have $z\prec_W K$ since $z'\in K$. Since $f^{m(k)}(x_k),\cdots,f^{n(k)}(x_k)$ is contained in $V$ and $n(k)-m(k)$ goes to infinity, we have that $\orb^+(z)\subset W$. Then by the assumption $x\notin \overline{W}$, we have $z\notin \orb^-(x)$. Hence $z\notin \overline{\orb^-(x)}$ because $\overline{\orb^-(x)}\subset\orb^-(x)\cup K$.
\end{proof}

\begin{lem}\label{choose points}
There are a point $y\in W^u(p)$, a sequence of points $(z_k)_{k\geq 1}$, three sequences of neighborhoods $(U_k)_{k\geq 1}$, $(V_k)_{k\geq 1}$, $(W_k)_{k\geq 0}$ and a sequence of finite segment of orbits $Y_k=(y_k,f(y_k),\cdots,f^{m(k)}(y_k))_{k\geq 0}$, such that:
\begin{enumerate}
\item $W_{k+1}\subset W_k$ and $\cap_k W_k=K$;

\item Theorem~\ref{Thm:connecting lemma} can be applied to $z_k\in V_k\subset U_k$ for the triple $(f,\mathcal{U}_k,N_k)$, and $\overline{f^n(U_k)}\subset W_k\setminus W_{k+1}$ for all $0\leq n\leq N_k$;

\item $\overline{U_k}\cap \orb^-(x)=\emptyset$, for any $k\geq 1$;


\item $z_k\prec_{W_k} z_{k+1}$ and $z_k\prec_{W_k} K$, for any $k\geq 1$;

\item the points $f^{m(k)}(y_k)$ and $y_{k+1}$ are contained in $V_{k+1}$ for all $k\geq 0$ where $y_0=y$, and $\orb^-(y)\cap W_1=\emptyset$;

\item $Y_k\subset W_{k}\setminus W_{k+2}$ for all $k\geq 0$ and $Y_k\cap \orb^-(x)=\emptyset$, for all $k\geq 0$;

\end{enumerate}
\end{lem}

\begin{proof} We build all the sequences by induction. Set $W_0=M$. We first choose $W_1$, $z_1$, $V_1$, $U_1$ and $Y_0$.

Since all periodic orbits in $K$ are hyperbolic, there is a neighborhood $W_1$ of $K$, such that there is no periodic points with period less than or equal to $N_1$ contained in $W_1\setminus K$. Also we can assume that $p\notin \overline{W_1}$. By Lemma~\ref{connecting points}, there is $z_1\in W_1\setminus K$, such that $z_1\prec_{W_1} K$, $z_1\notin \overline{\orb^-(x)}$, and $z_1\in \overline{W^u(p)}$. By the choice of $W_1$ and by the fact $z_1\in W_1\setminus K$, there is a neighborhood $U_1$ of $z_1$, that is disjoint from its $N_1$ first iterates and $\overline{f^n(U_1)}\subset W_1\setminus K$ for any $0\leq n\leq N_1$. Moreover, because $z_1\notin \overline{\orb^-(x)}$, we can assume $\overline{U_1}\cap (\orb^-(x)\cup K)=\emptyset$. By Theorem~\ref{Thm:connecting lemma}, there is $V_1\subset U_1$ associated to $(f,\mathcal{U}_1,N_1)$. Then there are a point $y\in W^u(p)\setminus W_1$ and a positive integer $m(0)$, such that $f^{m(0)}(y)\in V_1$. Moreover, by considering a negative iterate of $y$ if necessary, we can assume that $\orb^-(y)\cap W_1=\emptyset$. We take $y_0=y$ and $Y_0=(y,f(y),\cdots,f^{m(0)}(y))$. To sum up, we have obtained $W_1$, $z_1$, $V_1$, $U_1$ and $Y_0$.
\medskip

Now we construct the sequences by induction on $k$. After $W_k$, $z_k$, $V_k$, $U_k$ and $Y_{k-1}$ have been built, there is $W_{k+1}\subset W_k$ such that
\begin{itemize}
\item there is no periodic point with period less than or equal to $N_{k+1}$ contained in $W_{k+1}\setminus K$;

\item $\overline{f^n(U_k)}\cap W_{k+1}=\emptyset$, for all $1\leq n\leq N_k$;

\item $W_{k+1}\cap Y_{k-1}=\emptyset$;

\item $W_{k+1}$ is contained in $B(K,\frac{1}{k})$.
\end{itemize}

By Lemma~\ref{prec}, there is $z_{k+1}\in W_{k+1}\setminus K$, such that $z_k\prec_{W_k} z_{k+1}\prec_ {W_{k+1}} K$ and $\orb^+(z_{k+1})\subset W_{k+1}$. Since $x\notin W_{k+1}$, we have that $z_{k+1}\notin \orb^-(x)$. Moreover, by the fact that $z_{k+1}\notin K$ and $\alpha(x)\subset K$, we have that $z_{k+1}\notin \overline{\orb^-(x)}$. By Theorem~\ref{Thm:connecting lemma}, there are neighborhoods $V_{k+1}\subset U_{k+1}$ of $z_{k+1}$ associated to $(f,\mathcal{U}_{k+1},N_{k+1})$, such that:
\begin{itemize}
\item $\overline{U_{k+1}}\cap (\orb^-(x)\cup K)=\emptyset$,

\item $\overline{f^n(U_{k+1})}\subset W_{k+1}$ for all $0\leq n\leq N_{k+1}$.\\
\end{itemize}
Then there is $Y_{k}=(y_k,f(y_k),\cdots,f^{m(k)}(y_k))$, such that $y_k\in V_k$ and $f^{m(k)}(y_k)\in V_{k+1}$. Since $\overline{U_{k}}\cap \orb^-(x)=\emptyset$, we have that $y_k\notin \orb^-(x)$, and hence $Y_{k}$ is disjoint from $\orb^-(x)$.

Then we finish the proof of Lemma~\ref{choose points}.
\end{proof}

\begin{rem}\label{neighborhood of orb-(x)}
We point out here that, in Lemma~\ref{choose points}, there is an open set $V$ containing $\orb^-(x)$, such that $V\cap U_k=\emptyset$ for all $k\geq 1$. In fact, since $x\notin K$, for any integer $n\in\mathbb{N}$, there is $n_k$, such that $f^{-n}(x)\notin W_{n_k}$. By the fact that $U_k\subset W_k$ and $\overline{U_k}\cap \orb^-(x)=\emptyset$, for any $n\geq 0$, there is an open neighborhood $B_n$ of $f^{-n}(x)$, such that $B_n\cap U_k=\emptyset$ for any $k\geq 1$. Then we take $V=\bigcup_{n\geq 0}B_n$.
\end{rem}

Now we fix the point $y\in W^u(p)$ and the sequences $(z_k)_{k\geq 1}$, $(U_k)_{k\geq 1}$, $(V_k)_{k\geq 1}$, $(W_k)_{k\geq 0}$ and $(Y_k)_{k\geq 0}$ as in Lemma~\ref{choose points}. We have the following lemma.

\begin{lem}\label{perturbations}
There are a sequence of perturbations $(g_k)_{k\geq 0}$ of $f$ and a strictly increasing sequence of integers $(n_k)_{k\geq 0}$, such that,
\begin{enumerate}
\item $g_0=f$ and $n_0=0$;

\item there is $\phi_k\in \mathcal{V}_k$, such that $\phi_k=Id$ outside $U_k\cup \cdots \cup f^{N_k-1}(U_k)$ and $g_k=g_{k-1}\circ \phi_k$, for $k\geq 1$;

\item for any $l=\{0,1,\cdots,k-1\}$, the piece of orbit $(g_k^{n_l}(y),g_k^{n_l+1}(y),\cdots,g_k^{n_{l+1}}(y))$ is contained in $W_l\setminus W_{l+2}$.
\end{enumerate}
\end{lem}

\begin{proof}

We build inductively the sequences $(g_k)$ and $(n_k)$ and another sequence of integers $(m_k)_{k\geq 0}$ which satisfy the conclusions and also the following properties:
\begin{itemize}
\item $m_k>n_k$ and $g_k^{m_k}(y)\in V_{k+1}$;
\item the piece of orbit $(g_k^{n_k}(y),g_k^{n_k+1}(y),\cdots,g_k^{m_k}(y))$ is contained in $W_k\setminus W_{k+2}$.
\end{itemize}

First, we take $g_0=f$ and $n_0=0$. By Lemma~\ref{choose points}, there is $m_0>0$, such that $g_0^{m_0}(y)\in V_1$ and the piece of orbit $(y=g_0^{n_0}(y),g_0(y),\cdots,g_0^{m_0}(y))$ is contained in $W_0\setminus W_2$.

Now assume that $g_k$, $n_k$ and $m_k$ have been built, we explain how to get $g_{k+1}$, $n_{k+1}$ and $m_{k+1}$.

The forward orbit of $g_k^{n_k}(y)$ has a positive iterate $g_k^{m_k}(y)\in V_{k+1}$, and the backward orbit of $f^{m(k+1)}(y_{k+1})$ has a negative iterate $y_{k+1}\in V_{k+1}$. Moreover, these segments of orbit are contained in $W_k\setminus W_{k+3}$. Since $g_k$ coincides with $f$ on the set $U_{k+1}\cup f(U_{k+1})\cup\cdots\cup f^{N_{k+1}}(U_{k+1})$, one can apply Theorem~\ref{Thm:connecting lemma} to $(g_k,\mathcal{U}_{k+1},V_{k+1},U_{k+1})$ and get a diffeomorphism $g_{k+1}$. The new diffeomorphism $g_{k+1}$ is of the form $g_k\circ \phi_{k+1}$, where $\phi_{k+1}=Id$ outside $U_{k+1}\cup \cdots \cup f^{N_{k+1}-1}(U_{k+1})$ and $f\circ \phi_{k+1}\in\mathcal{U}_{k+1}$, thus $\phi_{k+1}\in\mathcal{V}_{k+1}$.

Since $g_{k+1}=g_k$ outside $U_{k+1}\cup \cdots \cup f^{N_{k+1}-1}(U_{k+1})$, the piece of orbit $(y,g_k(y),\cdots,g_k^{n_k}(y))$ under $g_k$ coincides with the one $(y,g_{k+1}(y),\cdots,g_{k+1}^{n_k}(y))$. By the new diffeomorphism $g_{k+1}$, the forward orbit of $g_{k+1}^{n_k}(y)$ has an iterate $f^{m(k+1)}(y_{k+1})$ under $g_{k+1}$ contained in $V_{k+2}$. That is to say, there is an integer $m_{k+1}>n_k$, such that $f^{m(k+1)}(y_{k+1})=g_{k+1}^{m_{k+1}}(y)$. Moreover, there exists an integer $n_{k+1}$ with $n_k<n_{k+1}<m_{k+1}$, such that:
\begin{itemize}
\item[--] the piece of orbit $(g_{k+1}^{n_k}(y),\cdots,g_{k+1}^{n_{k+1}}(y))$ is contained in the union of $\{g_{k}^{n_k}(y),\cdots,g_{k}^{m_{k}}(y)\}$ and $U_{k+1}\cup \cdots \cup f^{N_{k+1}-1}(U_{k+1})$, hence it is contained in $W_{k}\setminus W_{k+2}$,

\item[--] the piece of orbit $(g_{k+1}^{n_{k+1}}(y),\cdots,g_{k+1}^{m_{k+1}}(y))$ is contained in the union of $U_{k+1}\cup \cdots \cup f^{N_{k+1}-1}(U_{k+1})$ and $\{y_{k+1},\cdots,f^{m(k+1)}(y_{k+1})\}$, hence it is contained in $W_{k+1}\setminus W_{k+3}$.
\end{itemize}

Then the conclusions are satisfied for $k+1$. This ends the proof of Lemma~\ref{perturbations}.
\end{proof}

\begin{proof}[End of the proof of Proposition~\ref{asymptotic connecting 1}]
Recall that the supports $U_i\cup\cdots\cup f^{N_i-1}(U_i)$ and $U_j\cup\cdots\cup f^{N_j-1}(U_j)$ of the perturbations $\phi_i$ and $\phi_j$ are disjoint for any $i\neq j$, and $(\mathcal{V}_n)$ satisfy the property (F). Then the sequence $g_k=f\circ\phi_1\circ\cdots\circ\phi_k$ converges to a diffeomorphism $g\in\mathcal{U}_0\subset \mathcal{U}$. By the constructions, $g$ coincides with $f\circ\phi_k$ in the set $U_k\cup\cdots\cup f^{N_k-1}(U_k)$ and with $f$ elsewhere. We take $V$ to be the neighborhood of $\orb^-(x)$ in Remark~\ref{neighborhood of orb-(x)}, then it holds that $g$ coincides with $f$ on the set $\orb(p)\cup K\cup V\cup \orb^-(y)$ and $\omega(y,g)\subset K$. Since $\orb^-(x)\subset V$, we have that $Dg$ coincides with $Df$ on $\orb^-(x)$. Moreover, since $g$ is the limit of the sequence $(g_k)$, by Lemma~\ref{perturbations}, for any $n>n_k$, $g^n(y)\in W_k$. Then we have that $\omega(y,g)\subset K$. This finishes the proof of Proposition~\ref{asymptotic connecting 1}.
\end{proof}

\section{Asymptotic approximation for pseudo-orbits: proof of Proposition 3}\label{proposition 3}

To prove Proposition 3, we use the techniques of~\cite{bc,c1} to get true orbits by perturbing a pseudo-orbit. Similarly to the proof of Proposition 2, we have to perturb infinitely many times in a special neighborhood to keep some part of the initial dynamic unchanged. The proof refers a lot to~\cite{bc} and Section 3.2 of~\cite{c2}.

We take several steps. First, we choose an open set that covers all positive orbits of $X$ that are not on the local stable manifold of periodic orbits with small periods. Actually, we choose a special topological tower of $X$. Second, we construct a sequence of disjoint perturbation domains containing in their interior the special topological tower. Then, we choose an infinitely long pseudo-orbit in $X$ that goes from $z$ to $K$, has jumps only in the tiles of the perturbation domains and accumulates to $K$ in the future. Finally, we perturb in the perturbation domains to construct a true orbit which goes from $z$ to $K$ and accumulates to $K$ in the future.

We take a $C^1$ neighborhood $\mathcal{U}_0$ of the diffeomorphism $f_0$ with $\overline{\mathcal{U}_0}\subset \mathcal{U}$, such that, the element of $\mathcal{U}_0$ is of the form $f\circ\phi$ with $\phi\in\mathcal{V}_0$, where $\mathcal{V}_0$ is a $C^1$-neighborhood of $Id$ that satisfies the property (F) in Definition~\ref{property F}. Then there is a smaller $C^1$ neighborhood $\mathcal{U}'\subset\mathcal{U}_0$ of $f_0$ and an integer $N_0$ associated to $(f_0,\mathcal{U}_0)$ by Theorem~\ref{uniform connecting}. Take the integer $T=10\kappa_d d N_0$ where the integer $\kappa_d\geq 1$ is the number given by Lemma~\ref{ttower}.

From now on, we fix the $C^1$ neighborhoods $\mathcal{U}'\subset\mathcal{U}_0$ of $f_0$ and the integer $T$. Consider a diffeomorphism $f\in\mathcal{U}'$, an invariant compact set $K$, a positive invariant compact set $X$ and a point $z$, satisfying the following properties:
\begin{itemize}
\item[--] all periodic points contained in $K$ are hyperbolic,
\item[--] all periodic points contained in $X$ with period less than or equal to $T$ are hyperbolic,
\item[--] for any $\varepsilon>0$, there is an $\varepsilon$-pseudo-orbit contained in $X$ that connects from $z$ to $K$.
\end{itemize}
Also we fix a neighborhood $U$ of $X\setminus K$.

We take a decreasing sequence of $C^1$-neighborhoods $(\mathcal{U}_n)_{n\geq 1}$ of $f$ such that, $\mathcal U_1\subset \mathcal U_0$, and $\cap_n \mathcal{U}_n=\{f\}$. Moreover, the element of $\mathcal{U}_n$ is of the form $f\circ \phi$ with $\phi\in \mathcal{V}_n$, where $(\mathcal{V}_n)$ is a decreasing sequence of $C^1$ neighborhoods of $Id$ that satisfy the property (F) in Definition~\ref{property F}. For any $k\geq 1$, Theorem~\ref{Thm:connecting lemma} associates to each pair $(f,\mathcal{U}_k)$ an integer $N_k$. We can assume that $(N_k)_{k\geq 0}$ is an increasing sequence. We assume also that $z\notin K$, otherwise there is nothing to prove. For an integer $N$, denote by $Per_N(f)$ the set of periodic points of $f$ whose period is no more than $N$. Fix a small number $\gamma>0$ (to determine the $C^0$ distance between the new created diffeomorphism and $f$).

\subsection{Choice of topological towers}

In this section, we construct a family of special topological towers for the set $X$ with the properties stated in Lemma~\ref{choice of topological tower}.

\begin{lem}\label{choice of topological tower}
For any $\delta>0$, for any decreasing sequence of positive constants $(\gamma_n)_{n\geq 0}$, and for any increasing sequence of integers $(L_k)_{k\geq 0}$ where $L_0=10dN_0$, there are a decreasing sequence of neighborhoods $(U_k)_{k\geq 0}$ of $K$, a sequence of open sets $(W_k)_{k\geq 0}$, and a sequence of compact sets $(D_k)_{k\geq 0}$, such that, putting $X_k=X\cap \overline{(U_{k}\setminus U_{k+1})}$ for all $k\geq 0$, the following properties are satisfied.
\begin{itemize}
\item $U_0=M$, $z\notin\overline{U_1}$, and $\bigcap_{k\geq 0} U_k=K$,
\item For any $k\geq 0$,
\begin{enumerate}
\item\label{item:tower-no periodic orbit}  there is no periodic orbit with period less than $\kappa_d^2 L_{k+1}$ contained in $\overline{U_{k+1}}\setminus K$,

\item\label{item:tower-subset} $f^i(\overline{U_{k+1}})\subset U_{k}$, for all $-4\kappa_d^2 L_{k+1}\leq i\leq 4\kappa_d^2 L_{k+1}$,

\item\label{item:tower-disjointness} the $L_k$ sets $\overline{W_k}, f(\overline{W_k}),\cdots,f^{L_k-1}(\overline{W_k})$ are pairwise disjoint, contained in $U\setminus K$, and also contained in $U_{k-1}\setminus \overline{U_{k+2}}$, where we put $U_{-1}=U_0=M$,

\item\label{item:tower-second disjointness} for any $-4\kappa_d^2 L_{k+2}\leq i\leq 4\kappa_d^2 L_{k+2}$, we have that $f^i(\overline{U_{k+2}})\cap \overline{(W_0\cup\cdots\cup W_{k})}=\emptyset$,

\item\label{item:tower-third disjointness} for any $l<k$, any $0\leq i\leq L_l$ and any $0\leq j\leq L_{k}$, we have that $\overline{f^i(W_l)}\cap \overline{f^j(W_k)}=\emptyset$,

\item\label{item:tower-iterates} the set $D_k$ is contained in $W_k$, such that, any point in $X_0\setminus (\bigcup_{p\in Per_{L_0}(f)}W^s_{\delta}(p))$ has a positive iterate in $\int(D_0)$, and for $k\geq 1$, any point in $X_k$ has a positive iterate in $\int(D_k\cup D_{k-1})$,
\item\label{item:tower-diameter} the diameter of every connected component of $W_k$ is smaller than $\gamma_k$.
\end{enumerate}
\end{itemize}
\end{lem}

\begin{rem}
The set $W_0\cup \cdots \cup W_k$ can be seen as a special topological tower for $X_0\cup\cdots\cup X_k$, from the items~\ref{item:tower-disjointness},~\ref{item:tower-third disjointness} and~\ref{item:tower-iterates}.
\end{rem}

\begin{proof}

We build inductively the sequences $(U_k)_{k\geq 0}$, $(W_k)_{k\geq 0}$ and $(D_k)_{k\geq 0}$ from a sequence of open sets $(W_k')_{k\geq 0}$ and a sequence of compact sets $(D_k')_{k\geq 0}$, which satisfy the following additional properties: for any $k\geq 0$,
\begin{itemize}
\item \textit{$1'$. the set $W_k$ is contained in a small neighborhood of $W_k'$, and $W_k'\subset W_k$,}

\item \textit{$2'$. the sets $\overline{W'_k}, f(\overline{W'_k}),\cdots,f^{L_k-1}(\overline{W'_k})$ are pairwise disjoint, contained in $U\setminus K$, and also contained in $U_{k-1}\setminus \overline{U_{k+2}}$,}

\item \textit{$3'$. for all $-4\kappa_d^2 L_{k+2}\leq i\leq 4\kappa_d^2 L_{k+2}$, we have $f^i(\overline{U_{k+2}}))\cap \overline{(W_0'\cup\cdots\cup W_k')}=\emptyset$,}

\item \textit{$4'$. for any $l<k$, any $0\leq i\leq L_l$ and any $0\leq j\leq L_{k}$, we have $\overline{f^i(W_l)}\cap \overline{f^j(W_k')}=\emptyset$,}

\item \textit{$5'$. $D_k=(D'_k\cup D'_{k+1})\cap W_k$}, where $D_0'\subset W_0'$ and $D'_{k+1}\subset W_k\cup W_{k+1}'$,

\item \textit{$6'$. any point in $X_0\setminus (\bigcup_{p\in Per_{L_0}(f)}W^s_{\delta}(p))$ has a positive iterate in $\int(D_0')$ and any point in $X_k$ has a positive iterate contained in $\int(D_k')$ for any $k\geq 1$,}

\item \textit{$7'$.the diameter of every connected component of $W_k'$ is smaller than $\frac{\gamma_k}{2}$.}

\end{itemize}
\medskip

Put $U_{-1}=U_0=M$. We construct inductively the sets $U_{k+1}$, $W_{k}'$, $D_k'$, $W_{k-1}$ and $D_{k-1}$.


\paragraph{The sets $U_1$, $W_0'$, and $D_0'$: the case where $k=0$.}
By the assumption of hyperbolicity of periodic orbits in $K$, we can take a neighborhood $U_1\subset B(K,1)$ of $K$ such that $z\notin\overline{U_1}$ and there is no periodic orbit with period less than $\kappa_d^2 L_1$ in $\overline{U_1}\setminus K$. Notice that $U_0=M$. The properties~\ref{item:tower-no periodic orbit} and~\ref{item:tower-subset} are satisfied.

Recall that $X$ contains no non-hyperbolic periodic orbit with period less than $T$, where $T=10\kappa_d d N_0=\kappa_d L_0$. Hence $X_0$ contains no non-hyperbolic periodic orbit with period less than $\kappa_d L_0$. By Lemma~\ref{ttower}, there are an open set $W_0'\subset U$ whose closure $\overline{W_0'}$ is a compact $d$-dimensional sub-manifold with boundary and a compact set $D'_0\subset W_0'$, such that $\overline{W_0'}$ is disjoint from its $L_0$ iterates, which is the first property of item $2'$. Moreover, any point in $X_0\setminus (\bigcup_{p\in Per_{L_0}(f)}W^s_{\delta}(p))$ has a positive iterate contained in $\int(D'_0)$ and hence the item $6'$ is satisfied. By the item~\ref{item:ttower-neighborhood} of Lemma~\ref{ttower}, the set $\overline{W_0'}$ is contained in a small neighborhood of $X_0\cup f(X_0)\cup\cdots\cup f^{\kappa_dL_0}(X_0)$. Moreover, we can choose $W_0'$ such that the diameter of every connected component of $W_0'$ is smaller than $\frac{\gamma_0}{2}$, which is the item $7'$. Hence we can assume that $\bigcup_{i=0}^{L_0}f^i(\overline{W_0'})\subset U\setminus K$, since $X$ is positively invariant. To make the sequences complete, we could put $D'_{-1}=D_{-1}=W_{-1}=\emptyset$. Notice that, we do not need to check other items for the case $k=0$.

\paragraph{The sets $U_{k+2}$, $W_{k+1}'$, $D'_{k+1}$, $W_{k}$ and $D_k$.}
Assume $U_{j+1}$, $W_{j}'$,  $D'_{j}$, $W_{j-1}$ and $D_{j-1}$ have been constructed for any $0\leq j\leq k$. Now we build the sets $U_{k+2}$, $W_{k+1}'$,  $D'_{k+1}$, $W_{k}$ and $D_k$.
\medskip

We take a neighborhood $U_{k+2}\subset U_{k+1}\cap B(K,\frac{1}{k+2})$ of $K$, such that:
\begin{itemize}
\item there is no periodic orbit with periodic less than $\kappa_d^2 L_{k+2}$ in $\overline{U_{k+2}}\setminus K$, which is the property~\ref{item:tower-no periodic orbit},
\item $f^i(\overline{U_{k+2}})\subset U_{k+1}$, for all $-4\kappa_d^2 L_{k+2}\leq i\leq 4\kappa_d^2 L_{k+2}$, which is the property~\ref{item:tower-subset},
\item $\overline{W_k'}\cap f^i(\overline{U_{k+2}})=\emptyset$ for all $-4\kappa_d^2 L_{k+2}\leq i\leq 4\kappa_d^2 L_{k+2}$, which implies the item $3'$ and the last property of item $2'$.
\end{itemize}
\medskip

Consider the compact set $X_{k+1}=X\cap\overline{U_{k+1}\setminus U_{k+2}}$. Notice that $X_{k+1}$ contains no periodic orbit of period less than or equal to $\kappa_d^2 L_{k+1}$. By Lemma~\ref{ttower}, there is an open set $V_{k+1}'$ whose closure $\overline{V_{k+1}'}$ is a compact $d$-dimensional sub-manifold with boundary such that any point in $X_{k+1}$ has a positive iterate contained in $V_{k+1}'$. Moreover, the set $\overline{V_{k+1}'}$ is disjoint from its first $\kappa_d L_{k+1}$ first iterates and $\overline{V_{k+1}'}$ is contained in a small neighborhood of $\bigcup_{i=0}^{\kappa_d^2 L_{k+1}}f^{i}(X_{k+1})$. By taking this neighborhood small, we can assume that $V_{k+1}'$ satisfies the following properties:
\begin{itemize}
\item Since $f^i(\overline{U_{k+1}})\subset U_{k}$ for any $-4\kappa_d^2 L_{k+1}\leq i\leq 4\kappa_d^2 L_{k+1}$, and $X_{k+1}\subset \overline{U_{k+1}}$, we have that $f^i(\overline{V'_{k+1}})\subset U_{k}$, for any $-2\kappa_d^2 L_{k+1}\leq i\leq 2\kappa_d^2 L_{k+1}$.
\item Since $f^i(\overline{U_{k+1}})\cap (\overline{W_0\cup\cdots\cup W_{k-1}})=\emptyset$ for any $-4\kappa_d^2 L_{k+1}\leq i\leq 4\kappa_d^2 L_{k+1}$, we have that $f^i(\overline{V'_{k+1}})\cap (\overline{W_0\cup\cdots\cup W_{k-1}}\cup K)=\emptyset$ for any $-2\kappa_d^2L_{k+1}\leq i\leq 2\kappa_d^2L_{k+1}$.
\end{itemize}
Moreover, we can choose $V_{k+1}'$ such that the diameter of all its connected components is small enough, such that all the $i^{th}$ iterates of every connected component of $V_{k+1}'$ is of diameter smaller than $\frac{\gamma_{k+1}}{3}$, for any $0\leq i\leq \kappa_dL_{k+1}$.
\medskip

Recall that $\overline{W_k'}$ and $\overline{V'_{k+1}}$ are two compact $d$-dimensional sub-manifolds with boundary and $\overline{V_{k+1}'}$ is disjoint from its first $\kappa_d L_{k+1}$ first iterates. By Lemma~\ref{choose neighborhoods}, considering $\overline{W_k'}$ and $\overline{V'_{k+1}}$ as $W'$ and $V'$, and considering the integer $L_{k+1}$ as the integer $T$, there is an open set $S_k=W_{k}\cup V_{k+1}$ satisfying the following properties:
\begin{itemize}
\item $\overline{W_{k}}$ and $\overline{V_{k+1}}$ are two compact $d$-dimensional sub-manifolds with boundary.
\item $W_{k}\cup V_{k+1}\subset U\setminus K$.
\item $W_k$ is a small neighborhood of $W_k'$, and hence $\overline{W_{k}}$ is disjoint with its first $L_k$ iterates. Moreover, the properties $2'$, $3'$,  $4'$ of $W_{k}'$ implies the properties of~\ref{item:tower-disjointness},~\ref{item:tower-second disjointness},~\ref{item:tower-third disjointness} of $W_k$. The property $1'$ is automatically satisfied.
\item $V_{k+1}'\subset \bigcup_{i=0}^{\kappa_d L_{k+1}} f^{-i}(S_{k})$,
\item $\overline{W_{k}}\cap f^{i}(\overline{V_{k+1}})=\emptyset$ for all $i=0,\pm 1,\cdots,\pm L_{k+1}$,
\item $\overline{V_{k+1}}$ is contained in a small neighborhood of $V_{k+1}'\cup f(V_{k+1}')\cup\cdots\cup f^{\kappa_dL_{k+1}}(V_{k+1}')$ and disjoint from its $L_{k+1}$ iterates. Thus we can assume that $K\cap \overline{V_{k+1}}=\emptyset$, and for all $-\kappa_d L_{k+1}\leq i\leq \kappa_d L_{k+1}$, we have $f^i(\overline{V_{k+1}}))\cap \overline{(W_0\cup\cdots\cup W_{k-1})}=\emptyset$ and $f^i(\overline{V_{k+1}}))\subset U_{k}$.
\end{itemize}
Moreover, by the assumption of the diameter of every connected component of $W_{k}'$ and $V_{k+1}'$, we can take $W_k$ and $V_{k+1}$ such that every connected component of $W_k$ is of diameter less than $\gamma_k$ and every connected component of $V_{k+1}$ is of diameter less than $\frac{\gamma_k}{2}$. Then the item~\ref{item:tower-diameter} is satisfied.

\medskip

By the fact that any point in $X_{k+1}$ has a positive iterate contained in $V_{k+1}'$, and $V_{k+1}'\subset \bigcup_{i=0}^{\kappa_d L_{k+1}} f^{-i}(V_{k+1})$, one can see that any point in $X_{k+1}$ has a positive iterate contained in $S_{k}$. By the compactness of $X_{k+1}$, there is a compact set $D'_{k+1}\subset S_{k}$, such that all such iterates are contained in $\int(D'_{k+1})$. Put $W_{k+1}'=V_{k+1}$ and $D_k=(D_k'\cup D'_{k+1})\cap W_k$. Then we have $D'_{k+1}\subset S_k=W_k\cup W'_{k+1}$. From the construction of $V_{k+1}$, we can see that $W_{k+1}'$ and $D_{k+1}'$ satisfy the properties~\ref{item:tower-iterates}, $2', 3', 4', 5', 6', 7'$.

This finishes the construction of the sets $U_{k+2}$, $W_{k+1}'$, $D'_{k+1}$ $W_k$ and $D_{k}$.
\medskip

Notice that $\bigcap_{k\geq 0} U_k=K$ and $z\notin\overline{U_1}$ are obviously satisfied by the choice of $U_k$. This finishes the proof of Lemma~\ref{choice of topological tower}.
\end{proof}

\subsection{Construction of perturbation domains}

We take $L_k=10dN_k$ for all $k\geq 0$, and take a small number $\delta>0$, such that for any two different hyperbolic periodic points $q_1,q_2\in Per_{N_0}(f)\cap X$, we have $W^{\sigma_1}_{\delta}(q_1)\cap W^{\sigma_2}_{\delta}(q_2)=\emptyset$, where $\sigma_i\in \{u,s\}$. Take a decreasing sequence of positive constants $(\gamma_k)_{k\geq 0}$, such that for any subset of $M$ whose diameter is smaller than $\gamma_k$, then its $i^{th}$ iterate is of diameter smaller than $\gamma$ for any $0\leq i\leq L_k$. By Lemma~\ref{choice of topological tower}, we get the sequences $(U_k)_{k\geq 0}$, $(W_k)_{k\geq 0}$ and $(D_k)_{k\geq 0}$. We still denote $X_k=X\cap \overline{(U_{k}\setminus U_{k+1})}$ for all $k\geq 0$.

Now we build the perturbation domains for the family $(X_k)$. The techniques are mainly from Section 4.1 and 4.2 of~\cite{bc}. First, we build the perturbation domains that covers the points which are not on the local stable manifolds of periodic orbits with period less than or equal to $N_0$. The proof is essentially due to Corollaire 4.1 of~\cite{bc}. They deal with a family of perturbation domains with the same order, thus the union forms a perturbation domain. Here we have a sequence of perturbation domains with different orders, however, the construction of each perturbation domain can be separated.

\begin{lem}\label{perturbation domains 1}
There is a perturbation domain $B_k$ of order $N_k$ for $(f,\mathcal{U}_k)$ for each $k\geq 0$, such that the sequence $({B}_k)_{k\geq 0}$ satisfies the following properties.
\begin{enumerate}
\item The supports of the perturbations domains $B_k$ are pairwise disjoint, contained in $U$, and also contained in $U_{k-1}\setminus \overline{U_{k+2}}$.
\item Any point of $X_0\setminus (\bigcup_{p\in Per_{N_0}(f)}W^s_{\delta}(p))$ has a positive iterate in the interior of one tile of the perturbation domain ${B}_0$ and any point of $X_k$ has a positive iterate in the interior of one tile of the perturbation domain $B_{k-1}\cup B_k$ for $k\geq 1$.
\end{enumerate}

In consequence, for any $k\geq 0$, there is a finite family of tiles $\mathcal{C}_k$ associated to ${B}_k$, and a family of compact sets $\mathcal{D}_k$ contained in the interior of tiles of  $\mathcal{C}_k$, such that:
\begin{itemize}
\item each tile of $\mathcal{C}_k$ contains exactly one element of $\mathcal{D}_k$, for all $k\geq 0$ and each element of $\mathcal{D}_k$ is contained in a tile of $\mathcal{C}_k$,
\item any point of $X_0\setminus (\bigcup_{p\in Per_{N_0}(f)}W^s_{\delta}(p))$ has a positive iterate in the interior of one element of $\mathcal{D}_0$ and any point of $X_k$ has a positive iterate in the interior of one element of $\mathcal{D}_{k-1}\cup \mathcal{D}_k$ for $k\geq 1$.
\end{itemize}
Moreover, the diameter of any connected component of $B_k$ is smaller than $\gamma$.
\end{lem}

\begin{proof}
Consider the sequence of open sets $(W_k)_{k\geq 0}$ and the sequence of compact sets $(D_k)_{k\geq 0}$ obtained by Lemma~\ref{choice of topological tower}. Moreover, by the item~\ref{item:tower-diameter} of Lemma~\ref{choice of topological tower}, the diameters of components of each $W_k$ can be chosen small enough such that all their first $L_k$ iterates are contained in a perturbation domain of order $L_k$ by Theorem~\ref{existence of perturbation domain}.

Assume $W$ is a component of $W_k$, and put $D=D_k\cap W$. By assumption, $W$ is contained in a chart of perturbation $\varphi:W\rightarrow \mathbb{R}^d$. We can tile $W$ with tiles of proper size such that any cube that intersects $\varphi(D)$ is contained in $\varphi(W)$. We do the same thing for all other components of $W_k$ that has non-empty intersection with $D_k$ and we get a finite family $\mathcal{P}_0$ of perturbation domains, each of them being an open set, pairwise disjoint, contained in $W_k$, and the union of their closure contains $D_k$ in its interior. Denote $\Phi_0$ the family of perturbation charts in the construction of $\mathcal{P}_0$.

Repeat the construction for $f^{2iN_k}(W_k)$ and $f^{2iN_k}(D_k)$, $i\in \{1,\cdots,5d-1\}$, and we get the families $\mathcal{P}_i$ of perturbation domains contained in $f^{2iN_k}(W_k)$, pairwise disjoint and the union of their closure contains $f^{2iN_k}(D_k)$ in its interior. Denote $\Phi_i$ the family of perturbation charts corresponding to $\mathcal{P}_i$. Consider the family $f^{-2iN_k}(\mathcal{P}_i)$ contained in $W_k$. The union of the closure of all cubes of $f^{-2iN_k}(\mathcal{P}_i)$ contains $D_k$ in its interior. By a $C^1$ small perturbation of $\Phi_i$, we can suppose that a point in $D_k$ can only be contained on the boundary of at most $d$ different cubes of all cubes contained in $\cup_{i=0}^{5d-1}f^{-2iN_k}(\mathcal{P}_i)$ \footnote{In~\cite{bc}, they call the sets of $\cup_{i=0}^{5d-1}f^{-2iN_k}(\mathcal{P}_i)$ on general position. We do not introduce this definition in our paper. The reader can refer to Section 3.3 of~\cite{bc} for more details.}. Since there are at least $5d$ families of cubes, we get that any point of $D_k$ is contained in the interior of at least $4d$ families of such cubes.

We replace every cube in $\mathbb{R}^d$ by another one with the same center and homothetic with rate $\rho<1$ close to $1$. Then we get the families $\mathcal{P}_{i,\rho}$ of perturbation domains whose closures are pairwise disjoint. If we choose $\rho$ close enough to $1$, then any point of $D_k$ is still contained in the interior of a cube of at least $4d$ families of $(f^{-2dN_k}(\mathcal{P}_{i,\rho}))_{0\leq k\leq 5d-1}$. By the compactness of $D_k$, for each $i$, there is a finite family $\Gamma_i$ of tiles of the domains $f^{-2iN_k}(\mathcal{P}_{i,\rho})$, such that the union $\Sigma_i$ of the tiles of $\Gamma_i$ satisfies: any point of $D_k$ is contained in the interior of at least $4d$ compact $(f^{-2iN_k}(\Sigma_i))_{0\leq k\leq 5d-1}$.

By another $C^1$ small perturbation of $\Phi_i$, we can suppose that any point of $D_k$ is contained on the boundary of the tiles of at most $d$ families of $(f^{-2iN_k}(\Gamma_i))_{0\leq k\leq 5d-1}$. Any point is contained in at least $4d$ families of tiles, hence any point is contained in the interior of at least one of these tiles. Define $B_k$ and $\mathcal{C}_k$ to be the union of the families $\mathcal{P}_{i,\rho}$ and the union of the families $\Gamma_i$ respectively.

Then the compact set $D_k$ is covered by the interior of the tiles of the family $f^{-2iN_k}(\Gamma_i)$. We can take all the components of the intersection of $f^{2iN_k}(D_k)$ and the elements of the family $\Gamma_i$, and this is the family $\mathcal{D}_k$.

Finally, by the assumption that $L_k=10dN_k$ and the choice of $W_k$ in Lemma~\ref{choice of topological tower}, the supports of perturbation domains $(B_k)_{k\geq 0}$ are pairwise disjoint and are contained in $U$. Also, the support of the perturbation domain $B_k$ is also contained in $U_{k-1}\setminus \overline{U_{k+2}}$. By the choice of the sequence $(\gamma_k)_{k\geq 0}$, one can see that the diameter of any connected component of $B_k$ is smaller than $\gamma$. This finishes the proof of Lemma~\ref{perturbation domains 1}.
\end{proof}

We also have to construct perturbation domains that cover the stable and unstable manifolds of periodic orbits contained in $X_0\cap Per_{N_0}(f)$. By the assumption of hyperbolicity of periodic orbits, $X_0\cap Per_{N_0}(f)$ is a finite set. By Proposition 4.2 in~\cite{bc}, we can construct in the following way.

\begin{lem}[Proposition 4.2 of~\cite{bc}]\label{perturbation domains 2}
For any periodic orbit $Q\subset X\cap Per_{N_0}(f)$, any neighborhood $V$ of $Q$, there are a neighborhood $W$ of $Q$, two perturbation domains $B_s$ and $B_u$ of order $N_0$ for $(f,\mathcal{U}_0)$, two finite families of tiles $\mathcal{C}_s$ and $\mathcal{C}_u$ associated to $B_s$ and $B_u$ respectively, two finite families of compact sets $\mathcal{D}_s$ and $\mathcal{D}_u$, and an integer $n_0(Q)$, such that:
\begin{enumerate}
\item $V$ contains $\overline{W}$ and $\bigcup_{0\leq i\leq N_0-1} f^i(B_s \cup B_u)$.
\item $f^i(B_s) \cap f^j(B_u)=\emptyset$ for all $0\leq i,j\leq N_0-1$,
\item each element of $\mathcal{D}_s$ is contained in the interior of an element of $\mathcal{C}_s$, and  each element of $\mathcal{D}_u$ is contained in the interior of an element of $\mathcal{C}_u$. Moreover, each tile of  $\mathcal{C}_s$ and $\mathcal{C}_u$ contains exactly an element of $\mathcal{D}_s \cup \mathcal{D}_u$
\item for any two pairs $D_s\in \mathcal{D}_s$ and $D_u\in \mathcal{D}_u$, there is $n\in \{0,\cdots,n_0(Q)\}$, such that $f^n(D_s)\cap D_u\neq \emptyset$.
\item for any point $z\in W\setminus W^s_{loc}(Q)$, there is $n>0$ and $D\in \mathcal{D}_u$, such that $f^n(z)\in \int(D)$ and $f^i(z)\in V$ for all $0\leq i\leq n$. Moreover, if $f(z)\not \in W$, then $n\leq n_0(Q)$.
\item for any point $z\in W\setminus W^u_{loc}(Q)$, there is $n>0$ and $D\in \mathcal{D}_s$, such that $f^{-n}(z)\in \int(D)$ and $f^{-i}(z)\in V$ for all $0\leq i\leq n$. Moreover, if $f^{-1}(z)\not \in W$, then $n\leq n_0(Q)$.
\end{enumerate}
Moreover, the diameter of any connected component of $B_s$ and $B_u$ is smaller than $\gamma$.
\end{lem}

\subsection{Choice of a pseudo-orbit}

By Lemma~\ref{perturbation domains 1}, we have the sequences of perturbation domains $(B_k)_{k\geq 0}$, tiles $(\mathcal{C}_k)_{k\geq 0}$ and families of compact sets $(\mathcal{D}_k)_{k\geq 0}$. Notice that there are only finitely many periodic orbits contained in $Per_{N_0}(f)\cap X$, and they are all outside $\overline{U_1}$. Hence for each periodic orbit $Q$ contained in $X$ with period no more than $N_0$, we can take an open neighborhood $V(Q)\subset U$ that are pairwise disjoint, disjoint from $\overline{U_1}$ and disjoint from $f^i(B_k)$ for any $0\leq i\leq N_k-1$ and any $k\geq 0$. By Lemma~\ref{perturbation domains 2}, we have for each $Q$ the open set $W(Q)$, the perturbation domains $B_s(Q)$ and $B_u(Q)$, the families of tiles $\mathcal{C}_s(Q)$ and $\mathcal{C}_u(Q)$, the families of compact sets $\mathcal{D}_s(Q)$ and $\mathcal{D}_u(Q)$ and the number $n_0(Q)$. By the choice of $V(Q)$, we have that $f^i(B_{\sigma}(Q))\cap f^j(B_k)=\emptyset$ for any $\sigma=s,u$, any $0\leq i\leq N_k-1$ and any $k\geq 0$.

We take the union of $(B_s(Q),\mathcal{C}_s(Q),\mathcal{D}_s(Q))$, $(B_u(Q),\mathcal{C}_u(Q),\mathcal{D}_u(Q))$ and $(B_0,\mathcal{C}_0,\mathcal{D}_0)$, and to simplify the notations, we still denote the union by $(B_0,\mathcal{C}_0,\mathcal{D}_0)$. By Remark~\ref{union of perturbation domains}, we know the newly created $(B_0,\mathcal{C}_0,\mathcal{D}_0)$ is still a perturbation domain of order $N_0$ for $(f,\mathcal{U}_0)$. Denote $D'_k$ the union of the compact sets of the family $\mathcal{D}_k$ for each $k\geq 0$. By a similar argument as in~\cite[Section 4.3]{bc}, we assume that $z$ is not in any of the perturbation domains that we have chosen.

Recall that the support of the perturbation domain $B_k$ is $\supp(B_k)=\bigcup_{0\leq n\leq N_k-1}f^n(B_k)$. From the above constructions, the supports of the perturbation domains $(B_k)_{k\geq 0}$ are pairwise disjoint and are contained in $U$. Moreover, we have that $\supp(B_k)\subset U_{k-1}\setminus \overline{U_{k+2}}$ for any $k\geq 0$.

\begin{lem}\label{choice of pseudo-orbit}
There is an infinitely long pseudo-orbit $Y=(y_0,y_1,\cdots)$ for $f$ contained in $X$ that has jumps only in tiles of $(\mathcal{C}_k)_{k\geq 0}$ with $y_0=z$ and $d(y_n,K)\rightarrow 0$ as $n\rightarrow \infty$. Moreover, for each $k\geq 0$, there is a minimal number $l_k$, such that $y_i\in U_k$ for all $i\geq l_k$.
\end{lem}

\begin{proof}
By the former constructions, any point $x\in X_0$ has a positive iterate contained in the union of the interior of the compact set $D'_0$ and the open sets $W(Q)$ for all periodic orbits $Q\subset X\cap Per_{N_0}(f)$. Any point $x\in X_k$ has a positive iterate contained in the union of the interior of compact sets $D'_k$ for $k\geq 1$. By the compactness of the sets $X_k$, there are integers $T_k$, compact sets $\tilde{D_k}\subset D_k'$, and compact sets $\tilde{W}(Q)\subset W(Q)$, such that
\begin{itemize}
\item all points $x\in X_0$ will enter the union of $\tilde{D_0}$ and $\tilde{W}(Q)$ for all $Q\subset X\cap Per_{N_0}(f)$ in time bounded by $T_0$,
\item all points $x\in X_k$ will enter in $\tilde{D_k}$ for $k\geq 1$ in time bounded by $T_k$.
\end{itemize}
We can assume that $T_0$ is larger than $n_0(Q)$, for any $Q\subset X\cap Per_{N_0}(f)$ (recall that $n_0(Q)$ is obtained from Lemma~\ref{perturbation domains 2}).

\paragraph{Setting of the constants.}
For $k\geq 0$, set $\eta_k$ to be smaller than half of the minimum of the distances between a point in $(\bigcup_{Q\subset X\cap Per_{N_0}(f)}\tilde{W}(Q))\cup (\bigcup_{0\leq i\leq k}\tilde{D_i})$ and a point in the completement of $(\bigcup_{Q\subset X\cap Per_{N_0}(f)}W(Q))\cup (\bigcup_{0\leq i\leq k}D'_i)$. Moreover, we also assume that $\eta_k$ is smaller than half of the minimum of the distances between a point in $f(\overline{M\setminus U_k})$ to a point in $U_{k+1}$, and smaller than the minimum of the distances between a point in a compact set $D\in \mathcal{D}_k$ and a point on the boundary of the tile $C\in \mathcal{C}_k$ that contains $D$. Then for any $k\geq 0$, there is a number $0<\vep_k<\eta_k$, such that for any $\vep_k$-pseudo-orbit $(x_0,\cdots,x_{T_k})$, we have $d(x_i,f^i(x_0))<\frac{1}{2}\eta_k$, and $d(x_i,f^{i-T_k}(x_{T_k}))<\frac{1}{2}\eta_k$ for all $0\leq i\leq T_k$. For each $\vep_k$, there is a number $\delta_k\in(0,\frac{1}{3}\vep_k)$, such that, for any two points $x,y\in M$, if $d(x,y)<\delta_k$, then $d(f(x),f(y))<\frac{1}{3}\vep_k$. Without loss of generality, we can assume that the sequences $(\eta_k)_{k\geq 0}$, $(\vep_k)_{k\geq 0}$ and $(\delta_k)_{k\geq 0}$ are strictly decreasing sequences.

\paragraph{The sets $\tilde{X_k}$ and the pseudo-orbits $Z_k$.}
Now we take a finite $\delta_k$-dense set $\tilde{X}_k$ of $X_k$ for any $k\geq 0$, such that $z\in \tilde{X_0}$. For any $k\geq 0$, take a $\delta_k$-pseudo-orbit $(y_1^k,\cdots,y_{m_k}^k)$ in $X\setminus K$, such that $y_1^k=z$ and $d(y_{m_k}^k,K)<\delta_k$. Then we project this pseudo-orbit to the set $\bigcup_{i\geq 0}\tilde{X}_i$: if $y_j^k\in X_i\setminus X_{i+1}$, then there is $z_j^k\in \tilde{X_i}$, such that $d(y_j^k,z_j^k)<\delta_i$. Then the pseudo-orbit $Z_k=(z_1^k,\cdots,z_{m_k}^k)$ is a pseudo-orbit contained in $\bigcup_{i\geq 0}\tilde{X}_i$ that connects $z$ to $K$, where $z_1^k=z$.

Recall that $(U_k)_{k\geq 0}$ is a sequence of decreasing neighborhoods of $K$ and $X_k=X\cap\overline{(U_k\setminus U_{k+1})}$. Hence, if $y_j^k,y_{j+1}^k\in U_i$, then we have $d(f(z_j^k),z_{j+1}^k)\leq d(f(z_j^k),f(y_j^k))+d(f(y_j^k),y_{j+1}^k)+d(y_{j+1}^k,z_{j+1}^k)<\frac{1}{3}\vep_i+\delta_k+\delta_i<\frac{2}{3}\vep_i+\frac{1}{3}\vep_k$. Thus $d(f(z_j^k),z_{j+1}^k)<\vep_i$ when $k\geq i$. For any $k\geq 0$, by cutting some part of $Z_k$, we can assume that $z_j^k\neq z_l^k$ for any $j\neq l$. Then for any $k\geq 0$, there is a minimal integer $l(m,k)$, such that $z_i^k\in U_m$ for all $i> l(m,k)$.

\paragraph{The infinitely long pseudo-orbit $Z$.}
Since $\tilde{X_k}$ is a finite set for any $k\geq 0$, one can extract a subsequence $(Z^1_k)$ of $(Z_k)$, such that all pseudo-orbits in this subsequence have the same piece before staying in $U_1$, that is to say, $(z_1^k,\cdots,z_{l(1,k)}^k)$ are equal to each other for any $Z_k\in \{Z^1_k\}$. Similarly, there is a subsequence $(Z^2_k)$ of $(Z^1_k)$, such that all pseudo-orbits in this subsequence have the same piece before staying in $U_2$. We can continue this process, and finally, by taking the limit, we can get an infinitely long pseudo-orbit $Z=(z_1,z_2,\cdots)$ such that $z_1=z$, $d(z_n,K)\rightarrow 0$ as $n\rightarrow \infty$. Moreover, if $z_j,z_{j+1}\in X_i$, then $d(f(z_j),z_{j+1})<\vep_i$, since $Z$ is a limit set of $(Z_k)$.

By the analysis of Lemma $4.6$ in~\cite{bc}, the pseudo-orbit $Z=(z_1,z_2,\cdots)$ has the property stated in the following claim. We omit the proof here since it follows exactly the proof of Lemma $4.6$ in~\cite{bc}.
\begin{claim}\label{bound of time}
There is a strictly increasing sequence $t_0=1,t_1,\cdots$, such that for $j>0$, $z_{t_j}$ is contained in a compact set $E_j$ of $\bigcup_{k\geq 0}\mathcal{D}_k$. Moreover, for any $j\geq 0$,
\begin{itemize}
\item if $E_j\in \mathcal{D}_0$, then either $t_j-t_{j-1}<T_1$ or there is $Q\subset X\cap Per_{N_0}(f)$, such that $E_{j-1}\in \mathcal{D}_s(Q)$ and $E_{j}\in \mathcal{D}_u(Q)$,
\item if $E_j\in \mathcal{D}_k$ for some $k\geq 1$, then $t_j-t_{j-1}<T_k$.
\end{itemize}
\end{claim}

\paragraph{Construction of the pseudo-orbit $Y$ from $Z$.}
Now we replace some part of $Z$ to get an infinitely long pseudo-orbit that connects $\tilde{U}$ to $K$, accumulates to $K$ in the future, and has jumps only in the tiles of the perturbation domains. Using Claim~\ref{bound of time}, we construct $Y$ as the following.
\begin{itemize}
\item If $E_j\in \mathcal{D}_0$ and $t_j-t_{j-1}<T_1$ or if $E_j\in \mathcal{D}_k$ where $k\geq 1$, we replace the piece of pseudo-orbit $(z_{t_{j-1}+1},\cdots,z_{t_j})$ by the piece of true orbit $(f(z_{t_{j-1}}),f^2(z_{t_{j-1}}),\cdots,f^{t_j-t_{j-1}}(z_{t_{j-1}}))$.

\item If $E_j\in \mathcal{D}_0$ and $t_j-t_{j-1}\geq T_1$, we have that there is $Q\subset X\cap Per_{N_0}(f)$, such that $E_{j-1}\in \mathcal{D}_s(Q)$ and $E_{j}\in \mathcal{D}_u(Q)$. By Lemma~\ref{perturbation domains 2}, there is a piece of true orbit $(x,f(x),\cdots,f^t(x))$ such that $x\in E_{j-1}$, $f^t(x)\in E_{j}$ and $t\leq n_0(Q)<T_0$. Then we replace the piece of pseudo-orbit $(z_{t_{j-1}+1},\cdots,z_{t_j})$ by the piece of true orbit $(f(x),f^2(x),\cdots,f^t(x))$.
\end{itemize}
Then we get a new pseudo-orbit $Y=(y_0,y_1,\cdots)$.

\paragraph{The property of the pseudo-orbit $Y$.}
From the construction of the pseudo-orbit $Y$, we can see that two nearby points $y_i,y_{i+1}$ satisfy the following properties.
\begin{itemize}
\item If $y_{i}\notin \mathcal{D}_k$ for any $k\geq 0$, then $f(y_i)=y_{i+1}$.

\item If there exists $k\geq 0$, such that $y_{i}\in \mathcal{D}_k$, then $y_i$ and $f^{-1}(y_{i+1})$ are in a same tile of $\mathcal{C}_k$.
\end{itemize}
This implies that $Y$ has jumps only in tiles of $(\mathcal{C}_k)_{k\geq 0}$ with $y_0=z$ and $d(y_n,K)\rightarrow 0$ as $n\rightarrow \infty$. Moreover, there is a minimal number $l_k$, such that $y_i\in U_k$ for all $i\geq l_k$ and all $k\geq 0$.

This finishes the proof of Lemma~\ref{choice of pseudo-orbit}.
\end{proof}

\begin{rem}
In Lemma~\ref{choice of pseudo-orbit}, we only need to guarantee that the pseudo-orbit $Y$ obtained has jumps only in the tiles of $(\mathcal{C}_k)_{k\geq 0}$. We do not have to consider the scale of jumps at each step.
\end{rem}

\subsection{The connecting processes}
We take the infinitely long pseudo-orbit $Y=(y_0,y_1,\cdots)$ with $y_0=z$ contained in $X$ from Lemma~\ref{choice of pseudo-orbit}. Recall that $Y$ has jumps only in tiles of $(\mathcal{C}_k)_{k\geq 0}$ and $d(y_n,K)\rightarrow 0$ as $n\rightarrow \infty$. Moreover, for each $k\geq 0$, there is a minimal integer $l_k$, such that $y_i\in U_k$ for all $i\geq l_k$. We have the following lemma.

\begin{lem}\label{sequence of diffeomorphisms}
For each $k\geq 0$, there are a diffeomorphism $f_k$, an infinitely long pseudo-orbit $Y_k=(y_0^k,y_1^k,\cdots)$ of $f_k$ with $y_0^k=z$, and two sequences of positive integers $(m_k)_{k\geq 0}$ and $(n_k)_{k\geq 0}$, such that for any $k\geq 0$, the following properties are satisfied.
\begin{enumerate}
\item\label{item:f_k} There is $\phi_k\in\mathcal{V}_k$, such that $\phi_k|_{M\setminus \supp(B_k)}=Id|_{M\setminus \supp(B_k)}$, and $f_{k}=f_{k-1}\circ \phi_k$, where we put $f=f_{-1}$.
\item\label{item:m_k} The integer $m_{k}$ is the smallest positive integer, such that $f_{k+1}^{m_{k}}(z)\in U_{k}$. Moreover, we have $m_{k}<m_{k+1}$.
\item\label{item:piece of orbit} The piece of pseudo-orbit $(y_0^{k+1},y_1^{k+1},\cdots,y_{m_{k}-1}^{k+1})$ of $f_{k+1}$ coincides with $(z,f_{k+1}(z),\cdots,$ $f_{k+1}^{m_{k}-1}(z))$.
\item\label{item:n_k} $n_k\leq l_{k+2}$, and $y_{n_{k}+i}^k=y_{l_{k+2}+i}$, for all $i\geq 0$.
\item\label{item:jumps} The pseudo-orbit $Y_k$ of $f_k$ has jumps only in the tiles $\{\mathcal{C}_{k+1},\mathcal{C}_{k+2},\cdots\}$.
\end{enumerate}
\end{lem}

\begin{proof}
We build the sequences by induction. We construct $f_{k+1}$, $Y_{k+1}$, $n_{k+1}$ and $m_k$ after $f_{k}$, $Y_{k}$, $n_{k}$ and $m_{k-1}$ has been built.

\paragraph{The constructions of $f_{0}$, $Y_{0}$, $n_{0}$: the case $n=0$.} Recall that the pseudo-orbit $Y$ of $f$ has jumps only in tiles of $(\mathcal{C}_k)_{k\geq 0}$ and $l_k$ is the minimal integer, such that $y_i\in U_k$ for all $i\geq l_k$.

We consider the finite pseudo-orbit $(y_0,y_1,\cdots,y_{l_2})$, which also has jumps only in titles of $(\mathcal{C}_k)_{k\geq 0}$ since it is a piece of $Y$. Moreover, we have that $y_0=z$ and $y_0,y_{l_k}\notin \supp(B_0)$ by the former constructions. By Lemma~\ref{union of perturbation domains}, there are a diffeomorphism $f_0\in \mathcal{U}_0$, a positive integer $n_0$ and a new pseudo-orbit $Y_0^0=(\hat{y}_0^0,\hat{y}_1^0,\cdots,\hat{y}_{n_0}^0)$ of $f_0$, satisfying the following three properties.
\begin{itemize}
\item $\hat{y}_0^0=y_0=z$ and $\hat{y}_{n_0}^0=y_{l_2}$.
\item The diffeomorphism $f_0$ coincides with $f$ outside $\supp(B_0)$, hence there is $\phi_0\in\mathcal{V}_0$, such that $\phi_0|_{M\setminus \supp(B_0)}=Id|_{M\setminus \supp(B_0)}$, and $f_0=f\circ \phi_0$, which is the item~\ref{item:f_k}.
\item The pseudo-orbit $Y_0^0$ has only jumps in the tiles $\{\mathcal{C}_1,\mathcal{C}_2,\cdots\}$.
\item  $n_0\leq l_2$.
\end{itemize}

Now we consider the infinitely long pseudo-orbit $Y_0=(y_0^0,y_1^0,\cdots)$ of $f_0$ which is a composition of $Y_0^0$ and $(y_{l_2},y_{l_2+1}\cdots)$. That is to say $y_i^0=\hat{y}_i^0$ when $0\leq i\leq n_0$ and $y_i^0=y_{l_2+i-n_0}$ when $i> n_0$. Then the diffeomorphism $f_0$, the pseudo-orbit $Y_0$ and the integer $n_0$ satisfy the following properties.
\begin{itemize}
\item $Y_0$ has jumps only in the tiles $\{\mathcal{C}_1,\mathcal{C}_2,\cdots\}$, which is the item~\ref{item:jumps}.
\item $n_0\leq l_2$ and $y_{n_0+i}^0=y_{l_2+i}$ for any $i\geq 0$, which is the item~\ref{item:n_k}.
\end{itemize}

Notice that we do not have to check the items~\ref{item:m_k} and~\ref{item:piece of orbit} for the case $k=0$.


\paragraph{The constructions $f_{k+1}$, $Y_{k+1}$, $n_{k+1}$ and $m_k$: the case $n=k+1$.}
We assume that $f_k$, $Y_k$, $m_{k-1}$ and $n_k$ have been built. Then we have that the infinitely long pseudo-orbit $Y_k=(y_0^k,y_1^k,\cdots)$ of $f_k$ with $y_0^k=z$, has only jumps in the tiles $\{\mathcal{C}_{k+1},\mathcal{C}_{k+2},\cdots\}$. Moreover, the piece of the pseudo-orbit $(y_0^k,y_{1}^k,\cdots,y_{m_{k-1}-1}^k)$ coincides with $(z,f_k(z),\cdots,f_k^{m_{k-1}-1}(z))$ and the piece of the pseudo-orbit $(y_{n_{k}}^k,y_{n_{k}+1}^k,\cdots)$ coincides with $(y_{l_{k+2}},y_{l_{k+2}+1},\cdots)$.

Similarly to the construction in the case $k=0$, we consider the finite pseudo-orbit $(y_0^k,y_{1}^k,\cdots,y_{n_k}^k,\cdots,y_{n_k+l_{k+3}-l_{k+2}}^k)$ of $f_k$ which also has only jumps in the tiles $\{\mathcal{C}_{k+1},\mathcal{C}_{k+2},\cdots\}$ since it is a piece of $Y_k$. Notice that $y_{n_k+l_{k+3}-l_{k+2}}^k=y_{l_{k+3}}$ and $\supp(B_{k+1})\cap\overline{U_{k+3}}=\emptyset$. By Lemma~\ref{union of perturbation domains}, there are a diffeomorphism $f_{k+1}$, a positive integer $n_{k+1}$ and a new pseudo-orbit $Y_{k+1}^0=(\hat{y}_0^{k+1},\hat{y}_1^{k+1},\cdots,\hat{y}_{n_{k+1}}^{k+1})$ of $f_{k+1}$, satisfying the following three properties.
\begin{itemize}
\item $\hat{y}_0^{k+1}=z$ and $\hat{y}_{n_{k+1}}^{k+1}=y_{l_{k+3}}$.
\item The diffeomorphism $f_{k+1}$ coincides with $f_k$ outside $\supp(B_{k+1})$, hence there is $\phi_{k+1}\in\mathcal{V}_{k+1}$, such that $\phi_{k+1}|_{M\setminus \supp(B_{k+1})}=Id|_{M\setminus \supp(B_{k+1})}$, and $f_{k+1}=f_k\circ \phi_{k+1}$, which is the item~\ref{item:f_k}.
\item The pseudo-orbit $Y_{k+1}^0$ has only jumps in the tiles $\{\mathcal{C}_{k+2},\mathcal{C}_2,\cdots\}$.
\item  $n_{k+1}\leq l_{k+3}$.
\end{itemize}

Similarly to the case when $k=0$, we consider the infinitely long pseudo-orbit $Y_{k+1}=(y_0^{k+1},y_1^{k+1},\cdots)$ of $f_{k+1}$ which is a composition of $Y_{k+1}^0$ and $(y_{l_{k+3}},y_{l_{k+3}+1}\cdots)$. That is to say $y_i^{k+1}=\hat{y}_i^{k+1}$ when $0\leq i\leq n_{k+1}$ and $y_i^{k+1}=y_{l_{k+3}+i-n_{k+1}}$ when $i> n_{k+1}$. Then the diffeomorphism $f_{k+1}$, the pseudo-orbit $Y_{k+1}$ and the integer $n_{k+1}$ satisfy the following properties.
\begin{itemize}
\item $Y_{k+1}$ has jumps only in the tiles $\{\mathcal{C}_{k+2},\mathcal{C}_{k+3},\cdots\}$, which is the item~\ref{item:jumps}.
\item $n_{k+1}\leq l_{k+3}$ and $y_{n_{k+1}+i}^0=y_{l_{k+3}+i}$ for any $i\geq 0$, which is the item~\ref{item:n_k}.
\end{itemize}
\medskip

Then we take the smallest integer $m_{k}$, such that $f_{k+1}^{m_{k}+1}(z)\in U_{k+1}$. To be precise, we take $m_{k}$ in the following way:
\begin{itemize}
\item we take $m_0=1$,
\item when $k\geq 1$, we take $m_k$ such that $f_{k+1}^{m_{k}+1}(z)\in U_{k+1}$, and for all $0\leq i<m_k$, we have $f_{k+1}^{i}(z)\notin U_{k}$.
\end{itemize}
Since $\supp(B_{k+1})\subset U_{k}\setminus\overline{U_{k+3}}$, the diffeomorphism $f_{k+1}$ coincides with $f_k$ on the piece of orbit $(z,f_k(z),\cdots,f_k^{m_{k-1}-1}(z))$. By the item~\ref{item:tower-subset} of Lemma~\ref{choice of topological tower}, we have that $m_{k-1}<m_k$, which is the property~\ref{item:m_k}. The property~\ref{item:piece of orbit} is satisfied by the choice of $m_k$. This finishes the proof of Lemma~\ref{sequence of diffeomorphisms}.

\end{proof}

\begin{proof}[End of the proof of Proposition~\ref{asymptotic connecting}]
Now we consider the sequences $(f_k)_{k\geq 0}$, $(Y_k)_{k\geq 0}$, $(m_k)_{k\geq 0}$ and $(n_k)_{k\geq 0}$ from Lemma~\ref{sequence of diffeomorphisms}.
Recall that $\mathcal{U}_k=f\circ\mathcal{V}_k$ where $\mathcal{V}_k$ satisfies the property (F) in Definition~\ref{property F}. Then the sequence of diffeomorphism $f_k=f\circ\phi_0\circ\cdots\circ\phi_k$ converges to a diffeomorphism $g\in\mathcal{U}$. Since the supports of all perturbation domains of $(B_k)_{k\leq 0}$ are contained in $U$, we have that $g=f|_{M\setminus U}$. Moreover, since the diameter of any connected component of the perturbation domains is smaller than $\gamma$, we have that the $C^0$ distance of $g$ and $f$ is smaller than $\gamma$.

Since $\supp(B_{k+1})\subset U_{k}\setminus \overline{U_{k+3}}$, by the items~\ref{item:m_k} and~\ref{item:piece of orbit} of Lemma~\ref{sequence of diffeomorphisms}, the piece of orbit $(z,f_k(z),\cdots,f_k^{m_{k-1}-1}(z))$ is also a piece of orbit of $f_n$, when $n\geq k+1$. This implies that the limit of the sequence of pseudo-orbits $Y_k$ is the positive orbit of $z$ under $g$ since the sequence $(m_k)_{k\geq 0}$ is strictly increasing. By the item~\ref{item:n_k} of Lemma~\ref{sequence of diffeomorphisms}, we can see that $\orb^+(z,g)$ has only finitely many points outside $U_k$ for any $k\geq 0$ (bounded by $n_k$), hence $\omega(z,g)\subset K$. This finishes the proof of Proposition~\ref{asymptotic connecting}.
\end{proof}

\section{Proofs of the applications}\label{applications}

In this section, we give the proofs of the applications of the main theorem.

\subsection{Structural stability and hyperbolicity}

To prove Corollary~\ref{stuctually stable}, we use some of the results in~\cite{sv,ww,wgw,wenx}. We take two steps: first, we prove that the statement is true for a residual subset of $\diff^1(M)$, and then we prove it for all diffeomorphisms in $\diff^1(M)$.

Recall that an invariant compact set $\Lambda$ is {\it shadowable}, if for any $\varepsilon>0$, there is $\delta>0$, such that any $\varepsilon$-pseudo orbit $(x_0,x_1,\cdots,x_n)$ in $\Lambda$ can be $\delta$-shadowed by a segment of orbit $(x,f(x),\cdots,f^n(x))$ in $M$, i.e. there exists $x\in M$ such that $d(f^i(x),x_i)<\delta$ for any $0\leq i\leq n$.

Assume that $H(p)$ is the homoclinic class of a hyperbolic periodic point $p$ of a diffeomorphism $f\in\diff^1(M)$. We state two properties as follows:
\begin{itemize}
\item $(P1)$ There are $m\in\mathbb{N}$, $C>0$ and $0<\lambda<1$, such that $H(p)$ admits an $(m,\lambda)$-dominated splitting $T_{H(p)}M=E\oplus F$ with $\dim(E)=\ind(p)$. And for any periodic point $q$ homoclinically related to $p$, where $\tau(q)$ denotes the period of $q$, the followings are satisfied:
   \begin{displaymath}
        \prod_{0\leq i<\tau(q)/m} \|Df^m|_{E(f^{im}(q))}\|< C{\lambda}^{\tau(q)},
   \end{displaymath}

   \begin{displaymath}
        \prod_{0\leq i<\tau(q)/m} \|Df^{-m}|_{F(f^{-im}(q))}\|< C{\lambda}^{\tau(q)}.
   \end{displaymath}

\item $(P2)$ $H(p)$ is shadowable and every periodic pseudo-orbit can be shadowed by a periodic orbit.
\end{itemize}


Now we state the following two Lemmas, whose proofs will be omitted.

\begin{lem}[Theorem 1.2 of~\cite{wenx}]\label{no weak periodic orbit}
Assume that $f$ is a diffeomorphism in $\diff^1(M)$. If a homoclinic class $H(p)$ is structurally stable, then the property $(P1)$ is satisfied for $H(p)$.
\end{lem}

\begin{lem}[Proposition 7 of~\cite{wenx}]\label{wxdot}
Assume that $f$ is a diffeomorphism in $\diff^1(M)$ and $p$ is a hyperbolic periodic point. If the two properties $(P1)$ and $(P2)$ are satisfied for $H(p)$, then $H(p)$ is hyperbolic. 
\end{lem}




\begin{proof}[Proof of Corollary~\ref{stuctually stable}]

We take a residual subset $\mathcal{B}\subset \diff^1(M)$ such that any $f\in \mathcal{B}$ satisfies the conclusion of Theorem A and Lemma~\ref{generic properties}. We claim that, for any $f\in\mathcal{B}$, if a homoclinic class $H(p)$ of $f$ is structurally stable, then it is hyperbolic. In fact, by Lemma~\ref{no weak periodic orbit}, the property $(P1)$ is satisfied for $H(p)$. Consider the splitting $T_{H(p)}M=E\oplus F$ stated in property $(P1)$. By the two inequalities in the property $(P1)$, the Lyapunov exponents of all periodic points homoclinically related to $p$ are uniformly bounded away from zero. Then by Theorem A and Remark~\ref{rem:theorem a}, we can see that the bundle $E$ is contracting and the bundle $F$ is expanding. Hence $H(p)$ is hyperbolic.

Now we assume that $f$ is an arbitrary diffeomorphism in $\diff^1(M)$. If a homoclinic class $H(p)$ of $f$ is structurally stable, then there is a $C^1$ neighborhood $\mathcal{U}$ of $f$, such that $H(p_h)$ is conjugated to $H(p)$ through a homomorphism $\phi_h$ for any $h\in\mathcal{U}$. For a diffeomorphism $g\in\mathcal{B}\cap \mathcal{U}$, we know that $H(p_g)$ is structurally stable.
Then $H(p_g)$ is hyperbolic by the argument above, hence $H(p_g)$ satisfies the property $(P2)$. It is easy to see that the property $(P2)$ is invariant under conjugacy, thus is satisfied by $H(p)$ since $H(p)$ is conjugated to $H(p_g)$. The property $(P1)$ is satisfied by $H(p)$ by Lemma~\ref{no weak periodic orbit}. Then by Lemma~\ref{wxdot}, we have that $H(p)$ is hyperbolic. This finishes the proof of Corollary~\ref{stuctually stable}.
\medskip



This finishes the proof of Corollary~\ref{stuctually stable}.
\end{proof}

\subsection{Partial hyperbolicity}

Now we give the proofs of Corollary~\ref{application 1} and Corollary~\ref{application 2}.

\begin{proof}[Proof of Corollary~\ref{application 1}]
We assume that $T_{H(p)}M=E\oplus F$ where $\dim(E)$ is smaller than the smallest index of periodic orbits contained in $H(p)$. In particular we have that $\dim(E)<\ind(p)$. Assume by contradiction that $E$ is not contracting. By Theorem A, there are periodic orbits $\orb(q_n)\subset H(p)$ such that $\ind(q_n)=\dim(E)$, which contradicts to the assumption that $\dim(E)<\ind(q_n)$. This proves Corollary~\ref{application 1}.
\end{proof}

\begin{proof}[Proof of Corollary~\ref{application 2}]
We assume that $f$ satisfies the properties in Lemma~\ref{generic properties} and Theorem A. For a homoclinic class $H(p)$ of $f$, we denote by $j\geq 1$ and $l\leq d-1$ the smallest and the largest index of periodic point contained in $H(p)$. By item 4 of Lemma~\ref{generic properties} (Theorem 1 of~\cite{abcdw}), for any $j\leq i\leq l$, there are periodic orbits of index $i$ contained in $H(p)$. Moreover, for any $j\leq i\leq l$, the hyperbolic periodic points with index $i$ are dense in $H(p)$. By Theorem A of~\cite{w1}, $H(p)$ admits a dominated splitting with index $i$. By Remark~\ref{bundle of ds}, we have that $H(p)$ admits a dominated splitting $T_{H(p)}M=E^{cs}\oplus E^c_1\oplus\cdots\oplus E^c_n\oplus E^{cu}$, where $\dim(E^{cs})=j$, $n=l-j$, and $\dim(E^c_i)=1$ for all $1\leq i\leq n$. Since $H(p)$ contains hyperbolic periodic points with index $i$ for all $j\leq i\leq l$, we can see easily that the central bundle $E^c_i$ is neither contracting nor expanding, for any $1\leq i\leq n$.

Now we consider whether the two bundles $E^{cs}$ and $E^{cu}$ are hyperbolic. Here $E^{cs}$ is hyperbolic means it is contracting and $E^{cu}$ is hyperbolic means it is expanding.
\medskip
\paragraph{Case 1: both $E^{cs}$ and $E^{cu}$ are hyperbolic} In this case, the splitting $T_{H(p)}M=E^{cs}\oplus E^c_1\oplus\cdots\oplus E^c_n\oplus E^{cu}$ is a partially hyperbolic splitting. We put $E^s=E^{cs}$ and $E^{u}=E^{cu}$. Then the the smallest and the largest index of periodic orbits contained in $H(p)$ are $\dim(E^s)$ and $d-\dim(E^u)$ respectively.

\paragraph{Case 2: at least one bundle of $E^{cs}$ and $E^{cu}$ is not hyperbolic} Without loss of generality, we assume that the bundle $E^{cs}$ is not contracting. By Theorem A, there are periodic orbits $\orb(q_n)$ with arbitrarily long period and index $j$, whose largest Lyapunov exponent along $E^{cs}$ converges to 0. Moreover, such periodic orbits form a dense set in $H(p)$.  Then by the Franks' Lemma and a genericity argument like in Section~\ref{generic argument}, the homoclinic class $H(p)$ can be accumulated by periodic orbits with index $j-1$, see~\cite[Lemma 2.1]{csy}. Therefore, by Theorem A of~\cite{w1} and Remark~\ref{bundle of ds}, $E^{cs}$ has a dominated splitting $E^{cs}=E^s\oplus E^c_0$ with $\dim(E^c_0)=1$. By Corollary~\ref{application 1}, we have that $E^s$ is contracting. Symmetrically, if the bundle $E^{cu}$ is not expanding, then it admits a finer dominated splitting $E^{cu}=E^c_{n+1}\oplus E^u$ such that $E^u$ is expanding  and $\dim(E^c_{n+1})=1$.

By the analysis above, we can obtain that $H(p)$ admits a partially hyperbolic splitting $T_{H(p)}M=E^{s}\oplus E^c_1\oplus\cdots\oplus E^c_k\oplus E^{u}$ with every center bundle of dimension one. Moreover, we have that $E^{cs}=E^s$ or $E^{cs}=E^{s}\oplus E^c_1$ and symmetrically $E^{cu}=E^u$ or $E^{cu}=E^c_k\oplus E^u$. Hence the smallest index of periodic orbits contained in $H(p)$ is $\dim(E^s)$ or $\dim(E^s)+1$ and the largest one is $d-\dim(E^u)$ or $d-\dim(E^u)-1$.
\medskip

This ends the proof of Corollary~\ref{application 2}.

\end{proof}

\subsection{Lyapunov stable homoclinic classes}

Now we give the proofs of Corollary~\ref{application 3} and Corollary~\ref{application 4}.

\begin{proof}[Proof of Corollary~\ref{application 3}]
From Corollary~\ref{application 1}, we only have to prove the case where $\dim(E)$ equals the largest index of periodic orbits contained in $H(p)$. The idea of the proof follows from~\cite{po} and Section 3 of~\cite{acco}. We just give an explanation here and for more details, the reader should refer to~\cite{po} and Section 3 of~\cite{acco}.

Assume $f\in\mathcal{R}$ satisfies Theorem A where $\mathcal{R}$ is the residual set in Lemma~\ref{generic properties}. There is a neighborhood $\mathcal{U}$ of $f$, such that the items 3, 4 and 6 stated in Lemma~\ref{generic properties} are satisfied for $(f,H(p),\mathcal{U})$. We can assume that $p$ has the largest index among the periodic points contained in $H(p)$, hence $\dim(E)=\ind(p)$. Assume that the bundle $F$ is not expanding, then by the conclusion of the Theorem A, we can get a sequence of periodic orbits $\orb(q_n)$ homoclinically related to $\orb(p)$ with arbitrarily long period such that the smallest Lyapunov exponent of $\orb(q_n)$ along the bundle $F$ can be arbitrarily close to 0. By Lemma 2.3 of~\cite{gy}, we can assume that all the eigenvalues of $\|Df\|$ along $\orb(q_n)$ are real. Then by Theorem 1 of~\cite{g2} (Theorem 2.5 in~\cite{acco}) and a proper construction of a path of diffeomorphism (see~\cite{acco}), there is a diffeomorphism $g\in\mathcal{U}$ and a periodic point $q$ of $g$ with index larger than $\dim(E)$ such that $W^s(q)\cap W^u(p_g)\neq\emptyset$. By $C^1$ small perturbation, we can assume that $W^s(q)$ intersects $W^u(p_g)$ transversely. This property is persistent under $C^1$ perturbation, since $\ind(q)>\ind(p_g)$. Hence there is a neighborhood $\mathcal{V}\subset\mathcal{U}$ of $g$, such that for any $h\in\mathcal{V}$, we have $W^s(q_h)\pitchfork W^u(p_h)\neq\emptyset$. Take a diffeomorphism $h\in\mathcal{V}\cap\mathcal{R}$, we have that $H(p_h)$ is Lyapunov stable by the item 6 of Lemma~\ref{generic properties}. Then $W^u(p_h)$ is contained in $H(p_h)$. By the fact $W^s(q_h)\pitchfork W^u(p_h)\neq\emptyset$, we have that $q_h\in H(p_h)$. This contradicts the item 4 of Lemma~\ref{generic properties} by the choice of $\mathcal{U}$, since $\ind(q_h)>\ind(p_h)=\ind(p)$.
\end{proof}

\begin{proof}[Proof of Corollary~\ref{application 4}]
We assume that the second item does not happen. By Lemma~\ref{generic properties}, all periodic orbits contained in $H(p)$ have the same index. By~\cite{po}, $H(p)$ has a dominated splitting $T_{H(p)}M=E\oplus F$ such that $\dim(E)=\ind(p)$. By Corollary~\ref{application 3}, we have that the bundle $F$ is expanding. With the same argument to $f^{-1}$ and the bundle $E$, we get that $E$ is contracting for $f$. Hence the splitting $T_{H(p)}M=E\oplus F$ is hyperbolic. Then $H(p)$ is a hyperbolic chain recurrence class by item $2$ of Lemma~\ref{generic properties}. Hence by a standard argument using the shadowing lemma, $H(p)$ is an isolated chain recurrence class. By Theorem 5 of~\cite{abd}, since $M$ is connected, the homoclinic class $H(p)$ is in fact the whole manifold, hence $f$ is Anosov.
\end{proof}

\section*{Acknowledgments} This work was done when I was at University Paris-Sud 11 as a joint PhD student under the supervision of Prof. Lan Wen and Prof. Sylvain Crovisier. I am very grateful to Lan Wen and Sylvain Crovisier, who have given me great help and encouragement. Professor Sylvain Crovisier gave me many useful suggestions both in solving the problem and in the writing of the paper, and Professor Lan Wen carefully listened to the proof and gave me many useful comments. I would like to thank Shaobo Gan, Dawei Yang, Yongluo Cao, Rafael Potrie, Xiao Wen and Nicolas Gourmelon for reading the proof and for useful discussions. Dawei Yang pointed out to me that one can prove that structurally stable homoclinic classes are hyperbolic with the conclusion of the main theorem. I would also like to thank the referee for carefully reading the paper and for the many useful comments.

\vskip 1cm

\flushleft{\bf Xiaodong Wang} \\
School of Mathematical Sciences, Shanghai Jiao Tong University, Shanghai, 200240, China\\
School of Mathematical Sciences, Peking University, Beijing, 100871, China\\
Laboratoire de Math\'ematiques d'Orsay, Universit\'e Paris-Sud 11, Orsay, 91405, France\\
\textit{E-mail:} \texttt{xdwang1987@sjtu.edu.cn, xdwang1987@gmail.com}\\

\end{document}